\documentclass[preprint,sort&compress,12pt]{elsarticle}
\usepackage{bbm}
\usepackage{amstext}
\usepackage{amsmath}
\allowdisplaybreaks[2]
\usepackage{amssymb}
\usepackage{mathrsfs}
\usepackage{stmaryrd}
\usepackage{bm}
\usepackage{pstricks}
\usepackage{pst-coil}
\usepackage{pst-3d}
\usepackage{amsfonts}
\usepackage{float}

\usepackage[titles]{tocloft}

\usepackage{amsthm}
\usepackage{CJK}

\def\ftimes{\hfill }

\newcommand{\be}{\begin{equation}}
\newcommand{\bea}{\begin{equation}\begin{aligned}}
\newcommand{\beas}{\begin{equation*}\begin{aligned}}
\newcommand{\eeas}{\end{aligned}\end{equation*}}
\newcommand{\eea}{\end{aligned}\end{equation}}
\newcommand{\ee}{\end{equation}}

\begin{document}
\begin{CJK*}{GBK}{song}
\begin{frontmatter}
\title{
Stability of the Couette flow 
 for  3D Navier-Stokes \\ equations with rotation 
 }

\author{Wenting Huang$^1$}

\author{Ying Sun$^2$}

\author{Xiaojing Xu$^{1}$}

\address{1. School of Mathematical Sciences, Beijing Normal University, Beijing, 100875, PR  China}

\address{2. School of Science, Beijing University of Posts and Telecommunications, Beijing, 100876, PR  China}

\tnotetext[2]{E-mails:  hwting702@163.com (W. Huang), sunying@bnu.edu.cn(Y. Sun), xjxu@bnu.edu.cn (X. Xu). }

\begin{abstract}
Rotation significantly influences the stability characteristics of both laminar and turbulent shear flows. This study examines the stability threshold of the three-dimensional Navier-Stokes equations with rotation, in the vicinity of the Couette flow at high Reynolds numbers ($\mathbf{Re}$) in the periodical domain $\mathbb{T} \times \mathbb{R} \times \mathbb{T}$, where the rotational strength is equivalent to the Couette flow. Compared to the classical Navier-Stokes equations, rotation term brings us more two primary difficulties: the linear coupling term involving in the equation of $u^2$ and the lift-up effect in two directions. To address these difficulties, we introduce two new good unknowns that effectively capture the phenomena of enhanced dissipation and inviscid damping  to suppress the lift-up effect. Moreover, we establish the stability threshold for initial perturbation $\left\|u_{\mathrm{in}}\right\|_{H^{\sigma}} < \delta  \mathbf{Re}^{-2}$ for any $\sigma > \frac{9}{2}$ and some $\delta=\delta(\sigma)>0$ depending only on $\sigma$.
\end{abstract}
\end{frontmatter}


\newtheorem{thm}{Theorem}[section]
\newtheorem{lem}{Lemma}[section]
\newtheorem{pro}{Proposition}[section]
\newtheorem{concl}{Conclusion}[section]
\newtheorem{cor}{Corollary}[section]
\newproof{pf}{Proof}
\newdefinition{rem}{Remark}[section]
\newtheorem{definition}{Definition}[section]

\tableofcontents

\section{Introduction}\label{introud}
\numberwithin{equation}{section}
\textit{1.1. Presentation of the problem.} 
Since the pioneering works of Rayleigh on inviscid flow (1880) and Reynolds experiment (1883), research on the stability and transition to turbulence in various flow systems has been ongoing. The impact of rotation on shear flows are relevant in a wide range of fields, including industry, geophysics, atmospheric and astrophysical phenomena. Perhaps the simplest laboratory realization is rotating plane Couette flow, which can be experimentally achieved by using a plane Couette system on a rotating table (e.g. \cite{TA1996,TTA2010}). Recently, Oxley and Kerswell \cite{OK2024} have made the first attempt to study the linear stability of the three-dimensional (3D) Couette-Poiseuille flow by incorporating the effects of  stratification, rotation and viscosity simultaneously in physical experiment. 

Coriolis force induced by rotation may drastically change the flow behavior both for laminar and turbulent shear flows. For plane Couette flow with rotation, the Coriolis force will either be stabilizing or destabilizing across the full channel width.  
In addition, there is a type of instability that may occur in a fluid subject to body forces. Examples of such forces include buoyancy due to density differences and Coriolis forces due to system rotation. 

Despite the fact that rotation has been studied for over a century (e.g. \cite{Drazin2002}), the impact of rotation on the stability of Couette flow is still largely left to the physical experiments mentioned above, with little mathematical results available. Mathematically, our recent work \cite{HSX2024} considered the effect of rotation on the stability threshold of 3D Navier-Stokes equations near the Couette flow $(y, 0, 0)$, and gave a positive response to this problem. 

For 3D cases, currently in mathematics, perhaps the results of the stability of the Navier-Stokes equations near Couette flow have developed maturely. Without the effects of physical boundaries, Bedrossian, Germain and Masmoudi \cite{MR3612004} studied the stability threshold for 3D Navier-Stokes equations, and showed that if the initial perturbation satisfies $\left\|  u_0 \right\|_{H^{\sigma}} \leqslant \delta \nu^{\frac{3}{2}}$ for any $\sigma>\frac{9}{2}$, then the solution is global in time, remains within $O(\nu^{\frac{1}{2}})$ of the Couette flow in $L^2$ for all time. Subsequently, in the infinite regularity Gevrey class,  they \cite{MR4126259, MR4458538} studied the nonlinear subcritical transition (i.e., the transition from laminar flow to turbulence) of the Navier-Stokes equations near the Couette flow. In particular,  Wei and Zhang \cite{MR4373161} showed that the transition threshold $\gamma=1$ of 3D Navier-Stokes equations in Sobolev space $H^2$ is optimal. They further provided a mathematically rigorous proof that the regularity of initial perturbations does not affect this transition threshold. At present, the transition threshold in their results is considered  optimal. In the presence of physical boundary, at high Reynolds number regime, the boundary layer could affect the stability of the flow. Chen et al. \cite{MR4121130} have developed the resolvent estimate method for the two-dimensional case. Later on, Chen, Wei and Zhang \cite{CWZ2024} extended this method from \cite{MR4121130} to the 3D case with boundary conditions. They were the first to obtain the stability threshold $\gamma=1$ of the 3D Navier-Stokes equations near Couette flow with non-slip boundary condition, which has significant practical value. 

Compared with the 3D case, the results of stability thresholds in 2D case are more abundant. For details, please refer to \cite{MR3974608, BHIW1-2024, BHIW2-2024, MR3448924, MR3867637, MR4121130, CWZ2023, MR4176913, MR4451473, MR4322283, WZ2023TJM}. In fact, the excellent vortex structure in the 2D case prevents the occurrence of the lift-up effect, which is the fundamental difference between the 2D and 3D problems regarding their treatment approaches. Of course, there are a lot of works on physical experiments and numerical simulations devoted to estimating $\gamma$, also see \cite{Lundbladh1994, MR1886008, MR2736445, MR1631950, MR3185102}.

We are mainly interested in  the Couette flow on the stability of the 3D Navier-Stokes equations with rotation. Consider the following 3D incompressible Navier-Stokes equations with rotation, which is called Navier-Stokes-Coriolis equations in $\mathbb{T} \times \mathbb{R} \times \mathbb{T}$ (the torus $\mathbb{T}$ is the periodized interval $[0, 1]$):

\begin{equation}\label{1.1}
(NSC) \quad \quad \begin{cases}
\partial_{t} V-\nu \Delta V+V \cdot \nabla V+\beta e_{3} \times V+\nabla P=0, \\
\nabla \cdot V=0, \\
V(t=0, x, y, z)=V_{\mathrm{in}}(x, y, z),
\end{cases}
\end{equation}
where $V=\left(V^1, V^2, V^3\right)^{\text{T}}  \in \mathbb{R}^3$ is the velocity field and $P \in \mathbb{R}$ is the pressure of fluid. The term $e_3=(0, 0, 1)^{\text{T}}$ is the unit vector in the vertical direction. The term $\beta e_{3} \times V$ denotes the Coriolis force with the rotation speed $\beta \in \mathbb{R}\backslash\{0\}$. The constant $\nu>0$ is the viscosity coefficient, which  is the inverses of the Reynolds number ($\mathbf{Re}$), i.e., $\nu=\mathbf{Re}^{-1}$.  

In the previous work \cite{HSX2024}, we proved that the nonlinear stability of \eqref{1.1} near the Couette flow when the rotation speed $\beta>1$ or $\beta<0$.  
Note that when $\beta=0$, \eqref{1.1} corresponds to the classical Navier-Stokes equations. There are rich results on the stability of the Navier-Stokes equations.  In this paper,  
we  focus on the case when  $\beta=1$ and study the dynamic stability behavior of solutions of \eqref{1.1}  around the  Couette flow 
\begin{equation}\label{1.2}
V^s=(  y, 0, 0)^{\text{T}},  \quad \partial_{y} P^s=-  y.
\end{equation}

We introduce the perturbation $u=V-V^s$, which satisfies
\begin{equation}\label{1....3}
\begin{cases}
\partial_{t} u+ y \partial_{x} u+\begin{pmatrix}
0 \\
u^1 \\
0
\end{pmatrix}-\nu \Delta u+\nabla p^{L} =-u \cdot \nabla u-\nabla p^{NL}, \\
\nabla \cdot u=0,\\
u(t=0)=u_{\mathrm{in}},
\end{cases}
\end{equation}
where the pressure $p^{NL}$ and $p^{L}$ are determined by 
\begin{align}
	 \Delta p^{NL}&=-\partial_{i}u_j \partial_{j}u_i, \label{1.4} \\
	 \Delta p^{L}&=- \partial_{x} u^2- \partial_{y} u^1. \label{1.5}
\end{align}
The term $p^{NL}$  is  nonlinear contribution to the pressure due to the convective term and   $p^{L}$  is linear contribution to the pressure due to the interaction between the disturbance near the Couette flow and  the rotation of fluid. 

The goal of  this paper is to find the stability threshold for the perturbation. More precisely, we study the following problem:
\begin{itemize}
	\item Given $\sigma$, what is the smallest $\gamma>0$ such that if the initial perturbation is such that $\|u_{\mathrm{in}}\|_{H^{\sigma}}=   \varepsilon< \nu^{\gamma}$, then $u$ remains close to the Couette flow in a suitable sense and converges back to the Couette flow.
\end{itemize}
In other words, given a norm $\|\cdot\|_{H^{\sigma}}$, find a $\gamma=\gamma(\sigma)$ so that
\begin{align*}
\left\|  u_{\mathrm{in}}  \right\|_{H^{\sigma}}& \leqslant \;  \nu^{\gamma} \,\, \Rightarrow  \text{ stability}, \nonumber \\
\left\|  u_{\mathrm{in}}  \right\|_{H^{\sigma}}&  \gg  \;  \nu^{\gamma} \,\, \Rightarrow  \text{ instability},
\end{align*}
where the exponent $\gamma=\gamma (\sigma)>0$  is called the \textit{stability threshold}. And it is also known as the \textit{transition threshold} in almost all applied literature. Obviously, the \textit{transition threshold} problem is more stringent and complicated than the nonlinear stability problem.

\textit{1.2. Linear effects of (NSC).} In this paper, we consider the  nonlinear stability threshold for 3D  (NSC) near Couette flow. To establish nonlinear stability, the first step is to analyze the linear effects of the 3D Couette flow. Ignoring the nonlinear terms in \eqref{1....3}, the linearized perturbation equation satisfies 
\begin{align}\label{1....4}
&\partial_{t} u+  y \partial_{x} u+\begin{pmatrix}
0 \\
  u^1 \\
0
\end{pmatrix}-\nu \Delta u+  \nabla (-\Delta)^{-1} \partial_{x} u^2+  \nabla (-\Delta)^{-1} \partial_{y} u^1=0.
\end{align}
Alternatively, it can be written as follows
\begin{align}\label{1..24}
\partial_{t} u+  y \partial_{x} u-\nu \Delta u+  \mathcal{L} u=0,
\end{align}
where the linear operator $\mathcal{L}$ is defined as
\begin{align*}
\mathcal{L}\triangleq \begin{pmatrix}
- \partial_{x} \Delta^{-1} \partial_{y} & -\partial_{x} \Delta^{-1}\partial_{x} & 0\\
1-\partial_{y} \Delta^{-1} \partial_{y} & -\partial_{y} \Delta^{-1}\partial_{x} & 0 \\
- \partial_{z} \Delta^{-1} \partial_{y} & -\partial_{z} \Delta^{-1}\partial_{x} & 0
\end{pmatrix}.
\end{align*} 

In general, there are four kinds of linear effects: lift-up, enhanced dissipation, inviscid damping  and boundary layer,  which play a key role in determining the stability threshold of fluid dynamic.

To better describe the linear effects, we first define the projections on the zero frequency and the non-zero frequencies in $x$ of a function $f$  as follows:
\begin{equation}\label{HSX333}
P_{0}f=f_{0}=\int_{\mathbb{T}} f(x,y,z)dx, \quad	P_{\neq}f=f_{\neq}=f-f_{0}. 
\end{equation}
Furthermore, the projection onto the zero frequency in $x$ and $z$ of a function $f(x,y,z)$ is denoted by
\begin{equation*}
\overline{f}_0=f_{0, 0}=\int_{\mathbb{T}} \int_{\mathbb{T}} f(x,y,z)dx dz.
\end{equation*} 

In this paper, we investigate the spatial domain $\mathbb{T} \times \mathbb{R} \times \mathbb{T}$, enabling us to disregard the boundary layer effect. Subsequently, our analysis will be focused on the remaining three linear effects.

\textit{1.2.1. Lift-up effect.} It is noted that the lift-up effect of zero frequency velocity field in 3D is a crucial factor affecting the fluid stability. To this end, we consider the projection onto zero frequency in $x$ of \eqref{1..24}. Notice that $u_0$ satisfies
\begin{align}\label{1.25}
\partial_{t} u_0-\nu \Delta_{y, z} u_0+ \mathcal{L}_{k=0} u_0=0,
\end{align}	
where 
\begin{align}
\mathcal{L}_{k=0}\triangleq \begin{pmatrix}
0 & 0 & 0\\
1-\partial_{y} \Delta_{y, z}^{-1} \partial_{y} & 0 & 0 \\
- \partial_{z} \Delta_{y, z}^{-1} \partial_{y} & 0 & 0
\end{pmatrix}.
\end{align}
Notice that $u_0^1$ is decoupled from both $u_0^2$ and $u_0^3$. Hence, $u_0^1$ follows the standard heat equation. Applying the Fourier transform to the spatial variable and the Duhamel principle, we immediately obtain the explicit solution of \eqref{1.25}  in Fourier space
\begin{align*}
\begin{cases}
\widehat{\widetilde{u}_0^1}(t, \eta, l)=e^{-\nu |\eta, l|^2 t } \widehat{\widetilde{\left({u}^{1}_\mathrm{in} \right)_0}},  \\
\widehat{\widetilde{u}_0^2}(t, \eta, l)=e^{-\nu |\eta, l|^2 t } \left(\widehat{\widetilde{\left({u}^{2}_\mathrm{in} \right)_0}}-  \dfrac{l^2}{|\eta, l|^2} t \,\widehat{\widetilde{\left({u}^{1}_\mathrm{in} \right)_0}}\right),  \\
\widehat{\widetilde{u}_0^3}(t, \eta, l)=e^{-\nu |\eta, l|^2 t } \left(\widehat{\widetilde{\left({u}^{3}_\mathrm{in} \right)_0}}+  \dfrac{\eta l}{|\eta, l|^2} t \, \widehat{\widetilde{\left({u}^{1}_\mathrm{in} \right)_0}}\right).
\end{cases}
\end{align*}
This means that the \textit{lift-up} effect (which was first observed in \cite{Ellingsen1975}) occurs in both the $\widetilde{u}_{0}^2$ and $\widetilde{u}_{0}^3$ directions. 

\textit{1.2.2. Enhanced dissipation.}
Let us start by the diffusion-convection equation in $\mathbb{T} \times \mathbb{R} \times \mathbb{T}$:
\begin{equation}\label{1.18}
\partial_{t} f+  y \partial_{x} f  - \nu \Delta f = 0.
\end{equation}
In the following, we define a linear change of variable 
\begin{equation}\label{2.5}
X=x-  t y, \quad \quad Y=y, \quad \quad Z=z.
\end{equation}
and use the convention of capital letters when considering any function $f$ in the moving frame \eqref{2.5}, hence defining $F$ by $F(t, X, Y, Z)=f(t, x, y, z)$.
Then the solution takes the form
\begin{equation*}
\widehat{F}_{\neq} (t, k, \eta, l) = e^{- \nu \int_0^t k^2+\left( \eta-   k \tau \right)^2+l^2 d \tau} \widehat{f}_{\neq} (0, k, \eta, l).
\end{equation*}
More precisely, there holds
\begin{equation*}
\left\| f_{\neq} (t) \right\|_{L^2} \leqslant e^{-c \nu t^3} \left\| \left(f_\mathrm{in}\right)_{\neq} \right\|_{L^2} \lesssim e^{-c \nu^3 t} \left\| \left(f_\mathrm{in}\right)_{\neq} \right\|_{L^2}.
\end{equation*}
This is so-called the \textit{enhanced dissipation}, which is due to the  effect of friction caused by the change of $\Delta_{L}$ in the new coordinate system \eqref{2.5}. For the enhanced dissipation of the non-zero frequency velocity, we will provide detailed proof in subsection \ref{section3.1}.

\textit{1.2.3. Inviscid damping.} The velocity field can be recovered by the formula $U_{\neq}=\Delta_{L}^{-1} Q_{\neq}$. Furthermore, owing to the intricate coupling structure of $\widehat{Q}^1$ and $\widehat{Q}^2$ as described in equation (3.2), it is not feasible to directly derive a consistent bounded estimate for these variables. Hence, it is necessary to construct two  favorable quantities $K^1 \triangleq |k, l| \left(k^2+ \left(\eta-   kt\right)^2+l^2\right)^{-\frac{1}{2}} \widehat{Q}^1$ and $K^2 \triangleq |k|\left(k^2+ \left(\eta-  kt\right)^2+l^2\right)^{-\frac{1}{2}} \widehat{Q}^2$ to establish a connection for obtaining a consistent bounded estimate of $K^1$ and $K^2$. In addition, we have the following relation
\begin{equation*}
\big| \widehat{U_{\neq}^2} \big|=\frac{1}{k^2+\left(\eta-  kt\right)^2+l^2} \big| \widehat{Q_{\neq}^2} \big|=\frac{1}{|k| \left(k^2+\left(\eta-  kt\right)^2+l^2 \right)^{\frac{1}{2}}} \big| \widehat{K_{\neq}^2} \big|.
\end{equation*}
Due to the bound $\frac{1}{k^2+\left(\eta-  kt\right)^2+l^2} \lesssim {\left\langle {  t} \right\rangle}^{-2} |k, \eta, l|^2$, this leads to a decay estimate of the type
\begin{equation*}
\left\| U_{\neq}^2 \right\|_{H^s} \lesssim \frac{1}{  t} \left\| K_{\neq}^2 \right\|_{H^s}.
\end{equation*}
Similarly, $U_{\neq}^1$ also has the following estimate:
\begin{equation*}
\left\| U_{\neq}^1 \right\|_{H^s} \lesssim \frac{1}{  t} \left\| K_{\neq}^1 \right\|_{H^s}.
\end{equation*}
This decay mechanism is known as \textit{inviscid damping}.

\textit{1.3. Re-formulation of the governing equations.}
To better capture the structural characteristics of  the  perturbation equations \eqref{1....3}, a natural treatment is to use the unknowns $q^{i}=\Delta u^{i}$, $(i=1, 2, 3)$ (as observed in \cite{MR4126259}). Then these unknowns satisfy the system
\begin{align}\label{2.4}
	\left\{\begin{array}{l}
		\partial_{t} q^{1}+  y \partial_{x} q^{1}-\nu \Delta q^{1}+   \partial_{x y} u^{1}-   \partial_{x x} u^{2} \\
		\quad  =-u \cdot \nabla q^{1}-q^{j} \partial_{j} u^{1}-2 \partial_{i} u^{j} \partial_{i j} u^{1}+\partial_{x}\left(\partial_{i} u^{j} \partial_{j} u^{i}\right), \\
		\partial_{t} q^{2}+  y \partial_{x} q^{2}-\nu \Delta q^{2}+  q^1+  \partial_{x y} u^{2}-  \partial_{y y} u^{1}\\
		\quad   =-u \cdot \nabla q^{2}-q^{j} \partial_{j} u^{2}-2 \partial_{i} u^{j} \partial_{i j} u^{2}+\partial_{y}\left(\partial_{i} u^{j} \partial_{j} u^{i}\right), \\
		\partial_{t} q^{3}+  y \partial_{x} q^{3}-\nu \Delta q^{3}+2   \partial_{x y} u^{3}-   \partial_{x z} u^{2}-  \partial_{z y} u^{1} \\
		\quad  =-u \cdot \nabla q^{3}-q^{j} \partial_{j} u^{3}-2 \partial_{i} u^{j} \partial_{i j} u^{3}+\partial_{z}\left(\partial_{i} u^{j} \partial_{j} u^{i}\right), \\
		q(t=0)=q_{\text {in }}.
	\end{array}\right.
\end{align}

In the following, we will use a coordinate transform \eqref{2.5}  described in detail in \cite{MR4126259, MR3612004, Zelati2023},  which is a central tool for our analysis. In the new coordinates, the corresponding differential operators are changed as follows 
\begin{equation}\label{2.6}
	\nabla f(t, x, y, z) =\left(\begin{array}{c}
		\partial_{x} f \\
		\partial_{y} f \\
		\partial_{z} f
	\end{array}\right)=\left(\begin{array}{c}
		\partial_{X} F \\
		\left(\partial_{Y}-  t \partial_{X}\right) F \\
		\partial_{Z} F
	\end{array}\right) \\
	=\left(\begin{array}{c}
		\partial_{X} F \\
		\partial_{Y}^{L} F \\
		\partial_{Z} F
	\end{array}\right)=\nabla_{L} F(t, X, Y, Z) .
\end{equation}
Using the notation
\begin{equation*}
	\Delta_{L}=\nabla_{L} \cdot \nabla_{L}= \partial_{X}^2+(\partial_{Y}^{L})^2+\partial_{Z}^2.
\end{equation*}
The Laplacian transforms as
\begin{equation}\label{2.7}
	\Delta f=\Delta_{L}F=\left(\partial_{X}^2+(\partial_{Y}^{L})^2+\partial_{Z}^2 \right) F.
\end{equation}
Combining the definition of the new coordinates as in (\ref{2.5}) and denoting $U^i(t, X, Y, Z)=u^i(t, x, y, z)$ and $  Q^i(t, X, Y, Z)=q^i(t, x, y, z)$.  \eqref{2.4} then can be  transformed into the following system:
\begin{align}\label{2.10}
	\left\{\begin{array}{l}
		\partial_{t} Q^{1}-\nu \Delta_{L} Q^{1}+  \partial_{X Y}^{L} U^{1}-   \partial_{X X} U^{2} \\
		\quad  =-U \cdot \nabla_{L} Q^{1}-Q^{j} \partial_{j}^{L} U^{1}-2 \partial_{i}^{L} U^{j} \partial_{i j}^{L} U^{1}+\partial_{X}\left(\partial_{i}^{L} U^{j} \partial_{j}^{L} U^{i}\right), \\
		\partial_{t} Q^{2}-\nu \Delta_{L} Q^{2}+  Q^1+  \partial_{X Y}^{L} U^{2}-  \partial_{Y Y}^{L} U^{1}\\
		\quad  =-U \cdot \nabla_{L} Q^{2}-Q^{j} \partial_{j}^{L} U^{2}-2 \partial_{i}^{L} U^{j} \partial_{i j}^{L} U^{2}+\partial_{Y}^{L}\left(\partial_{i}^{L} U^{j} \partial_{j}^{L} U^{i}\right), \\
		\partial_{t} Q^{3}-\nu \Delta_{L} Q^{3}+2   \partial_{X Y}^{L} U^{3}-  \partial_{X Z} U^{2}-  \partial_{Z Y}^{L} U^{1} \\
		\quad  =-U \cdot \nabla_{L} Q^{3}-Q^{j} \partial_{j}^{L} U^{3}-2 \partial_{i}^{L} U^{j} \partial_{i j}^{L} U^{3}+\partial_{Z}\left(\partial_{i}^{L} U^{j} \partial_{j}^{L} U^{i}\right). 
	\end{array}\right.
\end{align}
Although most work is done directly on the system (\ref{2.10}), for certain steps it will be useful to use the momentum form of the equation
\begin{align}\label{2.11}
	&\partial_{t} U-\nu \Delta_{L} U+\begin{pmatrix}
		0 \\
		  U^1 \\
		0
	\end{pmatrix}+  \nabla_{L} (-\Delta_{L})^{-1} \partial_{X} U^2+  \nabla_{L} (-\Delta_{L})^{-1} \partial_{Y}^{L} U^1 \nonumber\\
	&\quad=-U \cdot \nabla_{L} U-\nabla_{L} (-\Delta_{L})^{-1}(\partial_{i}^{L} U^j \partial_{j}^{L} U^i).
\end{align}

\textit{1.4. Summary of main results.} Our first result in this paper describes the linear effects on the linearized perturbation equations.
 
\begin{thm}[\textbf{The linear stability}]\label{1.1.}
Assume the initial data $u_{\mathrm{in}}$ be a divergence-free, smooth vector field.  Then the solution $u$ to the linearized disturbance equation \eqref{1....4}  satisfies

$(1)$ 	For any $s\geqslant 0$, the linear enhanced dissipation  and inviscid damping of the non-zero   frequency $U_{\neq}$:
\begin{align}
& \left\| (U^1, U^2)_{\neq} (t) \right\|_{H^{s}} \lesssim {\left\langle {  t} \right\rangle}^{-1} e^{-\frac{\nu}{6}  k^2 t^3} \left\| u_{\mathrm{in}} \right\|_{H^{s+2}},  \label{1.11}\\
& \left\| U^3_{\neq} (t) \right\|_{H^{s}} \lesssim e^{-\frac{\nu}{6}  k^2 t^3} \left\|  u_{\mathrm{in}}\right\|_{H^{s+2}} \label{1.12}
\end{align}
for $ k \in \mathbb{Z}\backslash\{0\}$.

$(2)$ The lift-up effect of the zero  frequency $u_{0}$:
\begin{align}\label{1.13}
\left\{\begin{array}{l}
u_{0}^1(t, y, z)=e^{\nu \Delta_{y, z}t}\left( u_{\mathrm{in}}^1 \right)_0,\\
u_{0}^2(t, y, z)=e^{\nu \Delta_{y, z}t}{\left( \left( u_{\mathrm{in}}^2 \right)_0-t \,   \partial_{z}\Delta_{y, z}^{-1}\partial_{z} \left( u_{\mathrm{in}}^1 \right)_0  \right)},\\
u_{0}^3(t, y, z)=e^{\nu \Delta_{y, z}t}{\left( \left( u_{\mathrm{in}}^3 \right)_0+t \,  \partial_{z}\Delta_{y, z}^{-1}\partial_{y} \left( u_{\mathrm{in}}^1 \right)_0\right)}.
\end{array}\right.
\end{align}
\end{thm}

\begin{rem}
One of the distinctive characteristics of the interaction between shearing and rotation is a coupling of velocity field which causes $U_{\neq}^{1,2}$ to exhibit the same inviscid damping effect. However, this characteristic also slows down the algebraic decay rate ${\left\langle{t} \right\rangle}^{-1}$ of inviscid damping compared with the rate ${\left\langle {t} \right\rangle}^{-2}$ of Navier-Stokes equations in \cite{MR3612004}. This phenomenon arises from the construction of variables $K_1$ and $K_2$ as detailed in \eqref{3.4}, which play a crucial role in the linearization perturbation.
\end{rem}
\begin{rem}
In \cite{MR3612004}, Bedrossian, Germain and Masmoudi give the estimates on the zero frequencies in $x$  of the solution to  the Navier-Stokes equations. Compared to the Navier-Stokes equations, the lift-up  effect of the   Navier-Stokes equations in rotational framework   is reflected in    $u_0^2$ and $u_0^3$ for $t \lesssim \nu^{-1}$. For smooth initial data of size $\varepsilon$, we  expect at  best the bounds for $u_0^{2, 3}$ 
\begin{align}\label{1.*15}
\left\| u_0^{2, 3} \right\|_{L^{\infty} H^{s}}^2+ \nu \left\| \nabla u_0^{2, 3} \right\|_{L^2 H^{s}}^2 \lesssim  \left( \varepsilon  \nu^{-1} \right)^2.
\end{align}

\end{rem}

Our second result is about the nonlinear stability of (NSC) near Couette flows.

\begin{thm}[\textbf{The nonlinear stability}]\label{1.2.}
Let $\nu \in (0,1)$,   $\sigma > \frac{9}{2}$. There exists a constant $\delta>0$ depending only on $\sigma$ such that if $u_{\mathrm{in}} \in H^{\sigma}$  is divergence-free with
\begin{align}\label{1..11}
\varepsilon=\left\|u_{\mathrm{in}}\right\|_{H^{\sigma}}< \delta   \nu^2,
\end{align}
then the solution   to \eqref{1....3}  with initial data  $u_{\mathrm{in}}$  is global in time and satisfies the following estimates:
\begin{align}
&\left\| u_0^1 \right\|_{L^{\infty} H^{\sigma}}+ \nu^{\frac{1}{2}} \left\| \nabla u_0^1 \right\|_{L^2 H^{\sigma}}\lesssim  \varepsilon,   \label{1...9} \\
&\left\| u_0^{2, 3} \right\|_{L^{\infty} H^{\sigma}}+ \nu^{\frac{1}{2}} \left\| \nabla u_0^{2, 3} \right\|_{L^2 H^{\sigma}}\lesssim  \varepsilon   \nu^{-1},   \label{1...10}\\
&\left\| U_{\neq}^{1, 2} \right\|_{L^{\infty} H^{\sigma-2}}+ \nu^{\frac{1}{6}}\left\| \nabla_{L} U_{\neq}^{1, 2} \right\|_{L^2 H^{\sigma-2}}+  \nu^{\frac{1}{6}}\left\| t U_{\neq}^{1, 2} \right\|_{L^2 H^{\sigma-3}}\lesssim  \varepsilon,   \label{1...11} \\
&\left\| U_{\neq}^3 \right\|_{L^{\infty} H^{\sigma-2}}+ \nu^{\frac{1}{2}}  \left\| t U_{\neq}^3 \right\|_{L^2 H^{\sigma-3}}\lesssim   \varepsilon   \nu^{-\frac{1}{3}}.   \label{1...13}
\end{align}
\end{thm}

\begin{rem}\label{1.2.2}
    The latter two estimates in \eqref{1...11} and the second term in \eqref{1...13} emphasize the enhanced dissipation effects of $U_{\neq}^{1, 2}$ and   $U_{\neq}^3$, respectively, which is discussed later in Proposition \ref{pro4.3}. Notice that the estimate \eqref{1...10} of the zero frequency $u_0^{2, 3}$ is consistent with the bound \eqref{1.*15} which is  influenced by the lift-up effect.
\end{rem}

\begin{rem}	
Compared to our previous work  \cite{HSX2024}, Theorem \ref{1.2.} considers the case where the rotational intensity is equal to the coefficient of Couette flow. Here we emphasize that this situation brings more negative effects to the study of nonlinear stability, since rotation generates lift-up effects in both directions. As a result, when $\beta=1$, we can only obtain the stability threshold  $\gamma= 2$.
\end{rem}

\begin{rem}
   In fact,  if we consider the dynamic stability of \eqref{1.1} near the Couette flow with  shear rate $\beta>1$, that is, the steady state $( \beta y,0,0)$, then Theorem \ref{1.2.} still holds.
   
\end{rem}

\textit{1.5. Brief comment of key ideas.} Compared with the previous work in \cite{MR3612004} where Bedrossian, Germain and Masmoudi treated the stability of the classical Navier-Stokes equations, the Coriolis force causes additional difficulties in developing a priori estimates for solutions to (NSC).  Motivated by the method  in \cite{Zelati2023} and operator semigroup theory, we give the stability of linearized perturbation equations. Specifically, we observed that a new lift-up effect occurring in the  velocity field $u_{0}^{2}$ and $u_{0}^{3}$ due to rotation which differs from the classical Navier-Stokes equations, see subsection \ref{section3.2}. Based on the linear results, Fourier multipliers   precisely encoding the interplay between the dissipation and possible growth are designed.  These multipliers are then used to make energy estimates that lose the minimal amount of information from the linear terms, see subsection \ref{4.1}. We not only have understood and quantified the linear terms,  but also need to understand how this linear behavior interacts with the nonlinearity. Different from  the analysis of nonlinear stability of the classical Navier-Stokes equations,  we do not   directly estimate the non-zero frequency $Q_{\neq}^1$ and $Q_{\neq}^2$ due to the  difficulties in estimating the product of the term such as $\mathcal{LU}2$ (i.e., the term $ Q^{1}-\partial_{YY}^{L}U^{1}$ appeared in $Q^{2}$ equation). Indeed, we can not  construct a proper Fourier multiplier to deal with this linear term. Fortunately, we find that  when one deals with the linear stability problem in section \ref{section3}, $Q^{1}$ and $Q^{2}$ equations are reduced to the linear symmetric structures by using variable substitution \eqref{3.4}. Inspired by these ideas, we construct $\check{K}^{1}$ and $\check{K}^{2}$   mentioned  in section \ref{section4} which can avoid the processing of linear coupling term. Moreover, $\check{K}^{1}$ and $\check{K}^{2}$ satisfy the equations \eqref{HSX88} and \eqref{HSX99}, respectively. Thus, making full use of the linear symmetric structure in  $\check{K}^{1}$ and $\check{K}^{2}$ equations, we can estimate the non-zero frequencies $\check{K}^{1}_{\neq}$ and  $\check{K}^{2}_{\neq}$ together, which produces a good elimination mechanism that allows us to close the a priori assumption of  $\check{K}^{1,2}_{\neq}$. These treatments are the key to estimating nonlinear stability of the rotating flow.

The rest of the paper is organized as follows. In section \ref{Preliminaries} we give some preliminaries for later use. In section \ref{section3}, we discuss the linear stability effects of solutions and give a rigorous mathematical proof of Theorem \ref{1.1.}. Section \ref{section4} is devoted to the brief idea of proof  of Theorem \ref{1.2.}, including the introduction of Fourier multipliers, the bootstrap hypotheses and the weighted energy estimates of non-zero frequency and zero frequency. The  proof details of nonlinear stability are given in sections \ref{sec5}--\ref{sec7}.

\section{Preliminaries}\label{Preliminaries} \label{sec2}
\subsection{Notation}
Given two quantities $A$ and $B$, we 
 denote $A \lesssim B$, if there exists a constants $C$ such that $A\leqslant CB$ where the constant $C>0$ depends only on $\sigma$, but not on $\delta$, $\nu$, $C_0$ and $C_1$ (both $C_0$ and $C_1$ are defined on subsection \ref{4.2}). We similarly denote $A \ll B$ if $A \leqslant cB$ for a small constant $c \in(0, 1)$ to emphasize the small size of the implicit constant.  $\sqrt{1+t^2}$ is represented by ${\left\langle {t} \right\rangle}$ and  $\sqrt{k^2+\eta^2+l^2}$ is  denoted by $|k, \eta, l|$.
We use the shorthand notation $dV=dxdydz.$ For two functions $f$ and $g$ and a norm $\|\cdot\|_{X}$, we write 
$$\|(f,g)\|_{X}=\sqrt{\|f\|_{X}^{2}+\|g\|_{X}^{2}}.$$
Unless specified otherwise, in the rest of this section $f$ and $g$ denote functions from $\mathbb{T} \times \mathbb{R} \times \mathbb{T}$ to $\mathbb{R}^{n}$ for some $n\in \mathbb{N}$.

\subsection{Fourier Analysis and multiplier} 
We introduce the following notation for the Fourier transform  of a function $f= f(x, y, z)$. Given $(k,\eta,l)\in \mathbb{Z}\times \mathbb{R}\times \mathbb{Z}$, define 
\begin{equation}\label{2.1}
\mathcal{F}(f)=\widehat{f}\left(k, \eta, l\right)=\int_{x \in \mathbb{T}}\int_{y \in \mathbb{R}} \int_{z \in \mathbb{T}}f\left(x, y, z\right)e^{- i 2\pi \left(k x+ \eta y+ l z\right)}  dxdydz.
\end{equation}
Denoting the inverse operation to the Fourier transform by $\mathcal{F}^{-1}$ or $\check{\cdot}$
\begin{equation}\label{2.2}
\mathcal{F}^{-1}(f)=\check{f} \left(x, y, z\right)=\sum_{k \in \mathbb{Z}}\int_{\eta \in \mathbb{R}} \sum_{l \in \mathbb{Z}} f \left(k, \eta, l\right)e^{2\pi i \left(k x+ \eta y+ l z\right)}  d\eta.
\end{equation}
For a general Fourier multiplier with symbol  $m(k,\eta,l)$, we write $mf$ to denote $\mathcal{F}^{-1}(m(k,\eta,l)\hat{f})$. Thus, it holds
\begin{equation*}
	\widehat{m(D)f}=m(k,\eta,l)\widehat{f}.
\end{equation*}
We also define the Fourier symbol of $\nabla_{L}$ and  $-\Delta_{L}$ denoted by (\ref{2.6}) and (\ref{2.7}) respectively as
\begin{equation}\label{HSX444}
|\widehat{\nabla_{L}}|=|k,\eta-kt,l|,
\end{equation}
and 
\begin{equation}\label{2.8}
w(t, k, \eta, l)\triangleq \widehat{-\Delta_{L}}=k^2+(\eta- kt)^2+l^2.
\end{equation}
Note that $w$ is time dependent and thus its time derivative is 
\begin{equation}\label{2.9}
	\dot{w}=-2  k (\eta- kt).
\end{equation}

To avoid conflicting with the subscript just defined in \eqref{HSX333}, when $f$ is vector valued we use a superscript to denote the components. For example, if $f$ is valued in $\mathbb{R}^{3}$ then we write $f=(f^{1},f^{2},f^{3})$.

\subsection{Functional spaces}
The Sobolev space $H^{s} (s\geqslant0) $ is given  by  the norm 
\begin{equation*}
\left\| f \right\|_{H^{s}}=\left\| {\left\langle {D} \right\rangle}^{s} f \right\|_{L^2}=\left\| {\left\langle {k, \eta, l} \right\rangle}^{s} \widehat{f} \right\|_{L^2}.
\end{equation*}
Recall that, for $s>\frac{3}{2}$, $H^{s}$ is an algebra. Hence, for any $f,g\in H^{s}$, one has
\begin{align}\label{2.3}
	\left\|f g\right\|_{H^{s}} \lesssim \left\|f \right\|_{H^{s}}\left\|g\right\|_{H^{s}}.
\end{align}
We write the associated inner product as
$$\langle f,g \rangle_{H^{s}}=\int \langle\nabla\rangle^{s}f\cdot \langle \nabla \rangle^{s} g dV.$$
For functions $f(t,x,y,z)$ of space and time defined on the time interval $(a,b)$, we define the Banach space $L^{p}(a,b;H^{s})$ for $1\leqslant p\leqslant \infty$ by the norm
$$\|f\|_{L^{p}(a,b;H^{s})}=\|\|f\|_{H^{s}}\|_{L^{p}(a,b)}.$$
For simplicity of notation, we usually simply write $\|f\|_{L^{p}H^{N}}$ since
the time-interval of integrating in this work will be the same basically everywhere.
\subsection{Littlewood-Paley decomposition and paraproduct (see \cite{BCD2011})}\label{2.*.4}
Let  $\chi$ be a smooth, nonnegative function supported in the annulus  $B(0,5) \backslash B(0,1)$  of  $\mathbb{R}^{3}$ , and satisfy  $\sum_{i=-\infty}^{+\infty} \chi\left(\frac{\xi}{2^{i}}\right)=1$  for  $\xi \neq 0$, and define the Fourier multipliers
\begin{align*}
P_{i}=\chi\left(\frac{D}{2^{i}}\right), \quad P_{\leqslant I}=\sum_{i=-\infty}^{I} \chi\left(\frac{D}{2^{i}}\right), \quad P_{>I}=1-P_{\leqslant I}. 
\end{align*}
These Fourier multipliers enable us to split the product into two pieces such that each corresponds to the interaction of high frequencies of one function with low frequencies of the other:
\begin{align*}
f g=f_{High} g_{Low}+f_{Low} g_{High} = \sum_{i} P_{i} f P_{\leqslant i} g+ \sum_{i} P_{\leqslant i-1} f P_{i} g,
\end{align*}
with
\begin{equation*}
	 f_{High} g_{Low}=\sum_{i} P_{i} f P_{\leqslant i} g,\quad f_{Low} g_{High} =\sum_{i} P_{\leqslant i-1} f P_{i} g.
\end{equation*}
We record the estimate
\begin{equation*}
	\|f_{High}g_{Low}\|_{H^{s}} \lesssim \|f\|_{H^{s}}\|g\|_{H^{\sigma}}, \quad \text{for}~s>0,\sigma >\frac{3}{2}.
\end{equation*}

\subsection{Shorthand} 
It will be useful to define some shorthands for the
various terms appearing in \eqref{2.10}. Let's start with the linear terms, appearing in the equations for $Q^1$, $Q^2$ and $Q^3$:
\begin{align*}
&\mathcal{LU}2= Q^1- \partial_{Y Y}^L U^1   &(\text{lift up term}),  \\
&\mathcal{LU}3=- \partial_{Z Y}^L U^1  &(\text{lift up term}), \\
&\mathcal{LP}=- \partial_{X m}^L U^2, \, m=1, 3  &(\text{linear pressure term}),\\
&\mathcal{LS}1= \partial_{X Y}^L U^1   &(\text{linear stretching term}), \\
&\mathcal{LS}2= \partial_{X Y}^L U^2   & (\text{linear stretching term}), \\
&\mathcal{LS}3=2 \partial_{X Y}^L U^3  & (\text{linear stretching term}).
\end{align*}
Next, we consider the nonlinear terms in \eqref{2.10}. In the following, $i$, $j$ run in \{1, 2, 3\},  and  $m=1, 2, 3$, while $\varepsilon_{1}$ and $\varepsilon_{2}$ may be $0$ or $\neq$:
\begin{equation*}
\begin{aligned}
&\mathcal{TQ}^{m}=U \cdot \nabla_{L} Q^{m}  \quad \quad \quad &(\text{transport term}),  \\
&\mathcal{NLS}1(j,\varepsilon_{1},\varepsilon_{2})=Q_{\varepsilon_{1}}^{j} \partial_{j}^{L} U_{\varepsilon_{2}}^{m}, \quad \quad  \quad &(\text{nonlinear stretching term}),\\
&\mathcal{NLS}2(i,j,\varepsilon_{1},\varepsilon_{2})=2 \partial_{i}^{L} U_{\varepsilon_{1}}^{j} \partial_{i j}^{L} U_{\varepsilon_{2}}^{m}  \quad \quad  \quad & (\text{nonlinear stretching term}), \\
&\mathcal{NLP}(i,j,\varepsilon_{1},\varepsilon_{2})=\partial_{m}^L \left(\partial_{i}^{L} U_{\varepsilon_{1}}^{j} \partial_{j}^{L} U_{\varepsilon_{2}}^{i}\right) \quad \quad \quad & (\text{nonlinear pressure term}).
\end{aligned}
\end{equation*} 
We will often abuse notation slightly and use the same shorthands above to denote a terms contribution to an energy estimate. The origin of these terms is more or less clear except perhaps the stretching terminology, which is due to the similarity these terms have with the vortex stretching term in the 3D Navier-Stokes equations in vorticity form.

\section{The linear stability }\label{section3}	
In this section, we will give a proof of  Theorem \ref{1.1.}. We neglect the effect of nonlinear terms on the right-hand side of (\ref{1....3}),  that is, we only need to consider the  stability of the linear equation (\ref {1....4}). The ideas of the proof are mainly based on a change of variables in the Fourier space symmetrization framework and establish the appropriate energy functional.  

\subsection{Linear enhanced dissipation and inviscid damping}\label{section3.1}
We begin with  the analysis of the linearized perturbation equation (\ref{1....4}) from the nonzero $x$-modes, corresponding to $k\neq 0$ on the Fourier side.  To better understand this problem, it is convenient to work with the new variable $q^{i}=\Delta u^{i}$. Now we ignore the  nonlinear terms on the right hand side of (\ref{2.4}), then $q^{1},q^{2}$ satisfy the following equations
\begin{align}\label{3.1}
	\left\{\begin{array}{l}
		\partial_{t} q^1+ y \partial_{x} q^1- \partial_{xy}(-\Delta)^{-1} q^1+ \partial_{xx}(-\Delta)^{-1} q^2=\nu \Delta q^1,  \\
		\partial_{t} q^2+ y \partial_{x} q^2- \Delta_{x,z}(-\Delta)^{-1} q^1- \partial_{xy}(-\Delta)^{-1} q^2=\nu \Delta q^2.
	\end{array}\right.
\end{align}
In the moving frame (\ref{2.5}),  (\ref{3.1}) can be written on the Fourier side as  
\begin{align}\label{3.2}
	\left\{\begin{array}{l}
		\partial_{t} \widehat{Q}^1+ k(\eta- kt)w^{-1} \widehat{Q}^1- k^2 w^{-1} \widehat{Q}^2=-\nu w \widehat{Q}^1,  \\
		\partial_{t} \widehat{Q}^2+ |k, l|^2 w^{-1} \widehat{Q}^1+ k(\eta- kt) w^{-1} \widehat{Q}^2=-\nu w \widehat{Q}^2.
	\end{array}\right.
\end{align}
Thus, the linearized equations (\ref{1....4}) for $u^3$  translate in this setting
\begin{equation}\label{3.3}
	\partial_{t} \widehat{U}^3+\nu w \widehat{U}^3= kl w^{-1} \widehat{U}^2+  l(\eta- kt) w^{-1} \widehat{U}^1.
\end{equation}

Due to the coupling between $\widehat{Q}^1$  and $\widehat{Q}^2$ in (\ref{3.2}) which brings research difficulties  for the linear stability problems,  we introduce the new unknowns
\begin{equation}\label{3.4}
	K_1\overset{\Delta}{=}|k, l| w^{-\frac{1}{2}} \widehat{Q}^1, \quad  \quad  K_2\overset{\Delta}{=}|k|w^{-\frac{1}{2}} \widehat{Q}^2,
\end{equation}
then \eqref{3.2} is symmetrized into the following structure
\begin{align}\label{3.5}
	\left\{\begin{array}{l}
		\partial_{t} K_1- |k| |k, l| w^{-1} K_2=-\nu w K_1,  \\
		\partial_{t} K_2+ |k| |k, l| w^{-1} K_1=-\nu w K_2.
	\end{array}\right.
\end{align}
Taking the $L^2$-inner product of \eqref{3.5} yields
\begin{equation}\label{3.6}
	\dfrac{1}{2}\dfrac{d}{dt}\left( |K_1|^2 + |K_2|^2 \right)+ \nu w \left( |K_1|^2 + |K_2|^2 \right)=0.
\end{equation}
Then, we integrate \eqref{3.6}  in time
\begin{equation}\label{3.7}
	|K_1 (t)|^2 + |K_2 (t)|^2 =e^{-2 \nu \int_{0}^{t} w(\tau) d\tau} (|K_1 (0)|^2 + |K_2 (0)|^2).
\end{equation}
Note that
\begin{equation}\label{3.8}
	\int_{0}^{t} w(\tau) d\tau=(k^2 + l^2)t+\left[ (\eta-\dfrac{1}{2} kt)^2 +\dfrac{1}{12} k^2 t^2\right]t \geqslant \dfrac{1}{12} k^2 t^3.
\end{equation}
Hence, we have
\begin{equation}\label{3.9}
	|K_1 (t)|^2 + |K_2 (t)|^2 \leqslant e^{-\frac{\nu}{6}  k^2 t^3} \left(|K_1 (0)|^2 + |K_2 (0)|^2 \right).
\end{equation}

Now, to derive estimates on $u^{1}$ and $u^{2}$ in the original variables, as in Theorem \ref{1.1.}, we argue as follows. Note that $k\neq 0$. Recalling $Q^{i}=\Delta_{L}U^{i}, i=1,2$ and  using \eqref{3.4}, \eqref{3.9},  it yields
\begin{align}\label{3.10}
	\|u^{1}(t)\|_{L^{2}}^{2}=\left\|  U^1(t) \right\|_{L^2}^2=&\left\|  \Delta_{L}^{-1} Q^1 \right\|_{L^2}^2=\sum_{(k, l) \in \mathbb{Z}^2}\int_{\mathbb{R}} |k, l|^{-2} w^{-1} |K_1 (t)|^2 d \eta \nonumber\\
	\leqslant &\, \sum_{(k, l) \in \mathbb{Z}^2} e^{-\frac{\nu}{6}  k^2 t^3} \int_{\mathbb{R}} |k, l|^{-2} w^{-1} \left(|K_1 (0)|^2 + |K_2 (0)|^2 \right)d \eta \nonumber\\
	\leqslant &\, \sum_{(k, l) \in \mathbb{Z}^2} e^{-\frac{\nu}{6} k^2 t^3} \int_{\mathbb{R}} |k, l|^{-2} w^{-1} \left(|k, l|^2 |k, \eta, l|^2 |U^1 (0)|^2+ |k|^2 |k, \eta, l|^2 |U^2 (0)|^2 \right)d \eta \nonumber\\
	\lesssim &\, \sum_{(k, l) \in \mathbb{Z}^2} e^{-\frac{\nu}{6}  k^2 t^3} \int_{\mathbb{R}}  {\left\langle { t} \right\rangle}^{-2} |k, \eta, l|^4 \left(|U^1 (0)|^2+ |U^2 (0)|^2 \right)d \eta \nonumber\\
	\lesssim &\, \dfrac{e^{-\frac{\nu}{6}  k^2 t^3}}{{\left\langle { t} \right\rangle}^{2}} \left( \left\|  u_{in}^1 \right\|_{H^2}^2 +  \left\|  u_{in}^2 \right\|_{H^2}^2 \right),
\end{align}
and
\begin{align}\label{3.11}
	\|u^{2}(t)\|_{L^{2}}^{2}=\left\|  U ^2(t) \right\|_{L^2}^2=&\left\|  \Delta_{L}^{-1} Q^2 \right\|_{L^2}^2=\sum_{(k, l) \in \mathbb{Z}^2} \int_{\mathbb{R}} |k|^{-2} w^{-1} |K_2 (t)|^2 d \eta \nonumber \\
	\leqslant &\, \sum_{(k, l) \in \mathbb{Z}^2}e^{-\frac{\nu}{6}  k^2 t^3} \int_{\mathbb{R}} |k|^{-2} w^{-1} \left(|K_1 (0)|^2 + |K_2 (0)|^2 \right)d \eta \nonumber\\
	\lesssim &\, \sum_{(k, l) \in \mathbb{Z}^2} e^{-\frac{\nu}{6}  k^2 t^3} \int_{\mathbb{R}}  {\left\langle { t} \right\rangle}^{-2} |k, \eta, l|^4 \left(|U^1 (0)|^2+ |U^2 (0)|^2 \right)d \eta \nonumber\\
	\lesssim &\, \dfrac{e^{-\frac{\nu}{6}  k^2 t^3}}{{\left\langle { t} \right\rangle}^{2}} \left( \left\|  u_{in}^1 \right\|_{H^2}^2 +  \left\|  u_{in}^2 \right\|_{H^2}^2 \right),
\end{align}
respectively, here the inequality $w^{-1} \lesssim {\left\langle { t} \right\rangle}^{-2} |k, \eta, l|^2$ is used.

Next, by using \eqref{3.3}, \eqref{3.4} and (\ref{3.7}), we immediately have
\begin{align}\label{3.12}
	\frac{d}{dt}|\widehat{U}^3|+\nu w |\widehat{U}^3| &\leqslant   kl w^{-2} |\widehat{Q^{2}}|+ l(\eta- kt)w^{-2} |\widehat{Q^{1}}| \nonumber \\
	&\leqslant   k l |k|^{-1}w^{-\frac{3}{2}} |K_{2}|+ \frac{l}{ |k,l|}(\eta- kt) w^{-\frac{3}{2}} |K_{1}| \nonumber\\
	&\leqslant  w^{-1}|K_{2}|+ w^{-1}|K_{1}|\nonumber\\ 
	&\leqslant \sqrt{2}w^{-1} e^{-\nu \int_{0}^{t} w(s) ds}(|K_1(0)|^2+|K_2(0)|^2)^{\frac{1}{2}},
\end{align}
here we have used the fact $|l|\leqslant w^{\frac{1}{2}}$ and $|\eta- kt|\leqslant w^{\frac{1}{2}}.$
By direct calculation, we have
\begin{align*} 
	\int_{\mathbb{R}} w^{-1}(t)dt&=\int_{\mathbb{R}}\frac{1}{k^{2}+(\eta- kt)^{2}+l^{2}}dt \\&\leqslant 
	\frac{1}{k^{2}}\int_{\mathbb{R}}\frac{1}{1+(\frac{\eta}{k}- t)^{2}}dt\\&\leqslant k^{-2}\pi\\&\leqslant 6|k|^{-1}.
\end{align*}
Thus, we integrate \eqref{3.12} with respect to $t$, which obtains
\begin{align}\label{3.13}
	|\widehat{U}^3|\leqslant &\, e^{-\nu \int_{0}^{t} w(s) ds}\left[|\widehat{U}^3(0)|+\sqrt{2}\int_{0}^{t} w^{-1}(s) e^{-\nu\int_{0}^{s}w(\tau)d\tau}ds (|K_1(0)|^2+|K_2(0)|^2)^{\frac{1}{2}} \right]  \nonumber \\
	\leqslant &\, e^{-\nu \int_{0}^{t} w(s) ds}\left[|\widehat{U}^3(0)|+ \sqrt{2}\int_{0}^{t} w^{-1}(s)ds  (|K_1(0)|^2+|K_2(0)|^2)^{\frac{1}{2}} \right] \nonumber\\
	\leqslant &\, e^{-\frac{\nu}{12} k^2 t^3}\left[|\widehat{U}^3(0)|+ 12 \dfrac{1}{|k|} (|K_1(0)|^2+|K_2(0)|^2)^{\frac{1}{2}} \right].
\end{align}
Using \eqref{3.4} and \eqref{3.13} yields
\begin{align}\label{3.15}
	\|u^{3}(t)\|_{L^{2}}^{2}=\left\|U  ^3(t) \right\|_{L^2}^2 \lesssim &\, \sum_{(k, l) \in \mathbb{Z}^2} e^{-\frac{\nu}{6} k^2 t^3}\int_{\mathbb{R}}\left[|\widehat{U}^3(0)|^2+ \dfrac{1}{|k|^2} \left(|K_1(0)|^2+|K_2(0)|^2 \right) \right] d\eta \nonumber\\
	\lesssim &\, \sum_{(k, l) \in \mathbb{Z}^2} e^{-\frac{\nu}{6} k^2 t^3}\int_{\mathbb{R}}\Big[|\widehat{U}^3(0)|^2+ |k, l|^2 |k, \eta, l|^2 |\widehat{U}^1(0)|^2 \nonumber \\
	&+ |k|^2 |k, \eta, l|^2 |\widehat{U}^2(0)|^2 \Big] d\eta  \nonumber\\
	\lesssim &\, e^{-\frac{\nu}{6} k^2 t^3} \left( \left\| u_{in}^3 \right\|_{L^2}^2+ \left\| u_{in}^1 \right\|_{H^2}^2+ \left\| u_{in}^2 \right\|_{H^2}^2\right).
\end{align}

Finally, note that $k\neq 0$. Hence, when $s=0$, we directly obtain the estimates \eqref{1.11} and \eqref{1.12} on $U_{\neq}^{i}$ $(i=1, 2, 3)$ in Theorem {\ref{1.1.}} by (\ref{3.10}), (\ref{3.11}) and (\ref{3.15}).  Furthermore, for any $s>0$, we  obtain the estimates \eqref{1.11} and \eqref{1.12} by the same way.

\subsection{The lift-up effect }\label{section3.2}
We first review the lift-up effect for the homogeneous Navier-Stokes equations. Recall that we write $f_{0}=\int_{\mathbb{T}}f(x,y,z)dx$ for the $x$-average of a function. Setting $k=0$ and integrating the velocity equation in $x$ yields the system 
\begin{equation}\label{HSX1}
\partial_{t}u_{0}+\begin{pmatrix}
	u_{0}^{2}\\
	0\\
	0
\end{pmatrix}
=\nu \Delta u_{0}.
\end{equation}
The explicit solution of (\ref{HSX1}) can be written as
	\begin{align}\label{1.14}
	\left\{\begin{array}{l}
		u_{0}^1(t, y, z)=e^{\nu \Delta_{y, z}t}{\left( {\left({u}^{1}_\mathrm{in} \right)_0}-t {\left({u}^{2}_\mathrm{in} \right)_0} \right)},\\
		u_{0}^2(t, y, z)=e^{\nu \Delta_{y, z}t}{{\left({u}^{2}_\mathrm{in} \right)_0}},\\
		u_{0}^3(t, y, z)=e^{\nu \Delta_{y, z}t}{{\left({u}^{3}_\mathrm{in} \right)_0}}.
	\end{array}\right.
\end{align}
The lift-up effect refers to the linear in time growth of $u_{0}^{1}$ predicted by the formula above for $t\lesssim \nu^{-1}.$ In general, the best
global in time estimate for the Navier-Stokes equations that one can expect   is  
$$\|u_{0}^{1}\|_{L^{\infty}H^{s}}^{2}+\nu\|u_{0}^{1}\|_{L^{2}H^{s}}^{2}\lesssim\left(\varepsilon \nu^{-1}\right)^{2}.$$

Now we turn our attention to the Navier-Stokes-Coriolis case. Considering projection onto zero frequency in $x$  of \eqref{1....4},  it reads
\begin{align}\label{3.16}
	\left\{\begin{array}{l}
		\partial_{t} u_0^1=\nu \Delta_{y, z} u_0^1,  \\
		\partial_{t} u_0^2+ u_0^1+ \partial_{y}(-\Delta_{y, z})^{-1}\partial_{y} u_0^1=\nu \Delta_{y, z} u_0^2, \\
		\partial_{t} u_0^3+ \partial_{z}(-\Delta_{y, z})^{-1}\partial_{y} u_0^1=\nu \Delta_{y, z} u_0^3.
	\end{array}\right.
\end{align}
Due to incompressibility, we have $\widehat{u}_{0, 0}^2(\eta)=0$. When $l=0$, the system \eqref{3.16} can be degenerated into two decoupled heat equations in $\mathbb{R}$. Hence, we only consider the case when $l \neq 0$. 
Applying the Fourier transform to \eqref{3.16}, we obtain 
\begin{align}\label{3.17}
	\partial_{t} \begin{pmatrix}
		\widehat{u_0^1} \\
		\widehat{u_0^2} \\
		\widehat{u_0^3}
	\end{pmatrix}=\begin{pmatrix}
		-\nu\left(\eta^2+l^2\right) & 0 & 0 \\
		- \dfrac{l^2}{\eta^2+l^2} & -\nu\left(\eta^2+l^2\right) & 0 \\
		 \dfrac{\eta l}{\eta^2+l^2} & 0 & -\nu\left(\eta^2+l^2\right)
	\end{pmatrix}\begin{pmatrix}
		\widehat{u_0^1} \\
		\widehat{u_0^2} \\
		\widehat{u_0^3}
	\end{pmatrix} \overset{\Delta}{=} \mathcal{M}\begin{pmatrix}
		\widehat{u_0^1} \\
		\widehat{u_0^2} \\
		\widehat{u_0^3}
	\end{pmatrix}.
\end{align}
Then the solution to (\ref{3.17}) can be explicitly computed as
\begin{align}\label{3.21}
	\begin{pmatrix}
		\widehat{u_0^1}(t, \eta, l) \\
		\widehat{u_0^2}(t, \eta, l) \\
		\widehat{u_0^3}(t, \eta, l)
	\end{pmatrix}=
	e^{\mathcal{M} t}\begin{pmatrix}
		\widehat{(u_{\mathrm{in}}^1)_0} \\
		\widehat{(u_{\mathrm{in}}^2)_0} \\
		\widehat{(u_{\mathrm{in}}^3)_0} 
	\end{pmatrix}.
\end{align}

We consider the following characteristic equation
\begin{equation}\label{3.18}
	\det\left( \lambda I-\mathcal{M}\right)=0,
\end{equation}
note that the eigenvalue $\lambda=-\nu |\eta, l|^2$ is a triple root. Then we can compute the semigroup $e^{\mathcal{M} t}$ of \eqref{3.17}, namely
\begin{align}\label{3.19}
	e^{\mathcal{M} t}=&\, e^{\lambda t} e^{\left(\mathcal{M}-\lambda I\right)t} \nonumber\\
	=&\, e^{\lambda t} \left[ I+t \left(\mathcal{M}- \lambda I\right)+\frac{t^2}{2!}\left(\mathcal{M}- \lambda I\right)^2\right] \nonumber\\
	=&\, e^{-\nu |\eta, l|^2 t}\begin{pmatrix}
		1 & 0 & 0 \\
		- \dfrac{l^2}{\eta^2+l^2} t & 1 & 0 \\
		 \dfrac{\eta l}{\eta^2+l^2} t & 0 & 1
	\end{pmatrix}.
\end{align}
Putting (\ref{3.19}) into (\ref{3.21}),	 the solution of \eqref{3.17}  in Fourier space can be expressed as follows
\begin{align*}
	\left\{\begin{array}{l}
		\widehat{u_0^1}(t, \eta, l)=e^{-\nu |\eta, l|^2 t } \widehat{{\left({u}^{1}_\mathrm{in} \right)_0}},  \\
		\widehat{u_0^2}(t, \eta, l)=e^{-\nu |\eta, l|^2 t } \left(\widehat{{\left({u}^{2}_\mathrm{in} \right)_0}}- \dfrac{l^2}{|\eta, l|^2} t \,\widehat{{\left({u}^{1}_\mathrm{in} \right)_0}}\right),  \\
		\widehat{u_0^3}(t, \eta, l)=e^{-\nu |\eta, l|^2 t } \left(\widehat{{\left({u}^{3}_\mathrm{in} \right)_0}}+ \dfrac{\eta l}{|\eta, l|^2} t \,\widehat{{\left({u}^{1}_\mathrm{in} \right)_0}}\right),
	\end{array}\right.
\end{align*}
which combined with the inverse Fourier transform, we obtain \eqref{1.13}. Therefore, together with subsection \ref{section3.1} we finish the proof of Theorem \ref{1.1.}.

\section{Outline of the proof of Theorem {\ref{1.2.}}}\label{section4}	
In this section, we will list the outline of proof of Theorem \ref{1.2.}. Firstly, we present several definitions of the relevant Fourier multipliers, which are essential and useful in the process of proving nonlinear instability.

\subsection{The Fourier multipliers}\label{4.1}
Consider the following linear equation
\begin{equation}\label{4...1}
\partial_{t}f+2 \partial_{X Y}^{L} \Delta_{L}^{-1}f-\nu \Delta_{L} f=0,
\end{equation}
which occurs as some of the main linear terms governing $Q^{i}$ $(i=1, 2, 3)$ in (\ref{2.10}). The equation (\ref{4...1}) can be seen as a competition between the linear stretching term 
$ \partial_{X Y}^{L} \Delta_{L}^{-1}f$ and the dissipation term  $\nu \Delta_{L} f$.    
Applying the Fourier transform to \eqref{4...1} yields
\begin{equation*}
\partial_{t}\widehat{f}+2\frac{ k (\eta- kt)}{k^2+(\eta- kt)^2+l^2} \widehat{f}+\nu \left(k^2+(\eta- kt)^2+l^2\right) \widehat{f}=0.
\end{equation*}
If $k\neq 0$, the factor $2\frac{ k (\eta- kt)}{k^2+(\eta- kt)^2+l^2}$ is positive for $ t<\frac{\eta}{k}$,  in which case the term $ \partial_{X Y}^{L} \Delta_{L}^{-1}f$ can be viewed as a damping term. As for the factor $\nu \left(k^2+(\eta- kt)^2+l^2\right)$, this implies enhanced dissipation for $k\neq 0$. When $k\neq 0$, we have  the following inequality, which compares the sizes of these two factors 
\begin{equation}\label{4...3}
\nu \left( k^2+\left(\eta- kt\right)^2+l^2 \right) \gg \frac{ |k \left( \eta- kt\right)|}{k^2+\left(\eta- kt\right)^2+l^2}, \text{ if } \Big| t-\frac{\eta}{k} \Big|\gg \nu^{-\frac{1}{3}}.
\end{equation}
Indeed,    $\frac{ |k \left( \eta- kt\right)|}{\nu \left(k^2+\left(\eta- kt\right)^2+l^2\right)^2} \leqslant \frac{ | t-\frac{\eta}{k}|}{\nu \left(1+| t - \frac{\eta}{k}|^2\right)^2} \leqslant 1$ and note that  $\frac{ x}{\nu(1+x^2)^2} \ll 1$ for $|x|\gg  \nu^{-\frac{1}{3}}$. 
Thus, from \eqref{4...3}, if $\Big| t-\frac{\eta}{k} \Big|\gg \nu^{-\frac{1}{3}}$, the dissipation term $\nu \Delta_{L} f$ overcomes the stretching term $ \partial_{XY}^{L}\Delta_{L}^{-1}f$. 

In fact,  the trickiest is the case when  $0 <  t -\frac{\eta}{k}\lesssim \nu^{-\frac{1}{3}}$.  At this point, the stretching overcomes dissipation. To deal with this rang of $t$,   we  introduce some definitions of  the Fourier multipliers

	$(1)$ Define the multiplier $m(t, k, \eta, l)$ by $m(t=0, k, \eta, l)=1$ and the following ordinary differential  equations:
	\begin{align}\label{4...4}
	\frac{\dot{m}}{m}=\left\{\begin{array}{ll}
	\frac{2  k \left( \eta-  kt \right)}{k^2+\left(\eta- kt\right)^2+l^2} & \text{ if }   t \in \left[ \frac{\eta}{k}, \frac{\eta}{k}+1000 \nu^{-\frac{1}{3}} \right], \\
	0  & \text{ if }   t \notin \left[ \frac{\eta}{k}, \frac{\eta}{k}+1000  \nu^{-\frac{1}{3}} \right].
	\end{array}\right.
	\end{align}
This multiplier is such that if $f$ solves the above equation (\ref{4...1}) and $0< t-\frac{\eta}{k}<1000\nu^{-1/3} $, then $mf$ solves
$$ \partial_{t}(mf)-\nu \Delta_{L}(mf)=0,$$
and this equation is perfectly well behaved. That is, the growth that $f$ undergoes is balanced by the decay of the multiplier $m$.

	$(2)$ Define the additional multiplier $M(t, k, \eta, l)$ by $M(t=0, k, \eta, l)=1$ and
	\begin{itemize}
		\item if  $k=0$, $M(t, k=0, \eta, l)=1$ for all $t$;
		\item if  $k \neq 0$,
\begin{align}\label{4.5}
\frac{\dot{M}}{M}&= \frac{- \nu^{\frac{1}{3}} }{\left[\nu^{\frac{1}{3}} \big| t-\frac{\eta}{k} \big|\right]^{2}+1}.
\end{align}
\end{itemize}

For the expressions and properties of the multiplier $m$, there is a same situation as in \cite{MR3612004}, so we use it as a known conclusion and ignore the proof process.
\begin{lem}
	   $(1)$ The multiplier $m(t, k, \eta, l)$ can be given by the following exact formula:
	   \begin{itemize}
	   	\item if $k=0$: $m(t, k=0, \eta, l)=m(t=0, k, \eta, l)=1$;
	   	\item if  $k \neq 0$, $\frac{\eta}{k}<-1000  \nu^{-\frac{1}{3}}$: $m(t, k, \eta, l)=1$;
	   	\item if  $k \neq 0$, $-1000  \nu^{-\frac{1}{3}}<\frac{\eta}{k}<0$:
	   	\begin{align}
	   	m(t, k, \eta, l)=\left\{\begin{array}{ll}
	   	\dfrac{k^{2}+\eta^{2}+l^{2}}{k^{2}+(\eta- k t)^{2}+l^{2}}  & \text{ if  } 0< t<\frac{\eta}{k}+1000  \nu^{-\frac{1}{3}}, \\
	   	\dfrac{k^{2}+\eta^{2}+l^{2}}{k^{2}+\left(1000 k  \nu^{-\frac{1}{3}}\right)^{2}+l^{2}} & \text{ if  }  t>\frac{\eta}{k}+1000 \nu^{-\frac{1}{3}};
	   	\end{array}\right.
	   	\end{align}
	   	\item if  $k \neq 0$, $\frac{\eta}{k}>0$:
	   	\begin{align}
	   	m(t, k, \eta, l)=\left\{\begin{array}{lll}
	   	1  & \text{ if }   t<\frac{\eta}{k}, \\
	   	\dfrac{k^{2}+l^{2}}{k^{2}+(\eta- k t)^{2}+l^{2}} & \text{ if } \frac{\eta}{k}<  t<\frac{\eta}{k}+1000  \nu^{-\frac{1}{3}}, \\
	   	\dfrac{k^{2}+l^{2}}{k^{2}+\left(1000 k  \nu^{-\frac{1}{3}}\right)^{2}+l^{2}}  & \text{ if }   t>\frac{\eta}{k}+1000 \nu^{-\frac{1}{3}}.
	   	\end{array}\right.
	   	\end{align}
	   \end{itemize}

		$(2)$ In particular, $m(t, k, \eta, l)$ are bounded above and below, but its lower bound depends on $\nu$
		 \begin{equation}\label{4.11}
		\nu^{\frac{2}{3}} \lesssim m\left(t, k, \eta, l \right) \leqslant 1.
		\end{equation}
		
		$(3)$ $m(t, k, \eta, l)$ and the frequency have the following relationship: \begin{equation}\label{lem4.1.11}
		m\left(t, k, \eta, l \right) \gtrsim \frac{k^2+l^2}{k^2+\left(\eta- kt\right)^2+l^2}.
		\end{equation}
		
\end{lem}

The following lemma embodies some properties of the ghost multiplier $M(t, k, \eta, l)$.
\begin{lem}[\cite{Liss2020,MR3612004}]\label{lem4.2}
	$(1)$ The multiplier $M$ have positive upper and lower bounds:
	\begin{equation}\label{4.13}
	0<c<M(t, k, \eta, l) \leqslant 1
	\end{equation}
	for a universal constant $c$ that does not depend on $\nu$ and frequency. 
	
	$(2)$ For $k \neq 0$, we have
	\begin{equation}\label{4.15}
		1 \lesssim \nu^{-\frac{1}{6}} \sqrt{-\dot{M} M}\left(k, \eta, l\right)+\nu^{\frac{1}{3}}\big|k, \eta-  kt, l \big|.
	\end{equation}

\end{lem}

Using lemma \ref{lem4.2}, we directly give the following inference.
\begin{cor}\label{cor4.1}
	For any $g$ and $s \geqslant 0$, the following inequality holds
	\begin{equation}\label{4.17}
    \left\|  g_{\neq}  \right\|_{L^2 H^{s}} \lesssim \nu^{-\frac{1}{6}} \left( \left\|  \sqrt{-\dot{M} M} g_{\neq} \right\|_{L^2 H^{s}}+ \nu^{\frac{1}{2}} \left\|  \nabla_{L} g_{\neq} \right\|_{L^2 H^{s}}\right).
	\end{equation}
\end{cor}
 This corollary will be used frequently in the estimates of nonlinear stability.

In the following, we will introduce the ideas of the  proof of nonlinear stability of  (\ref{1....3}). To begin with, applying the energy methods in \cite{MB02} and \cite{LO97}, we establish the existence of a local solution, which is subsequently extended to a global solution.
\begin{lem}[Local existence]\label{lem4.3}
	Under the condition of Theorem \ref{1.2.}, let $u_{in}$ be divergence-free and satisfy (\ref{1..11}). Then there exists a $T^{\ast}>0$ independent of $\nu$ such that there is a unique strong solution $u(t)\in C([0,T^{\ast}];H^{\sigma})$, which satisfies the initial data and is real analytic for $t\in (0,T^{\ast}]$. Moreover, there exists a maximal existence time  $T_{0}$  with  $T^{\ast}<T_{0} \leqslant \infty$  such that the solution  u(t)  remains unique and real analytic on  $\left(0, T_{0}\right)$  and, if for some  $\tau \leqslant T_{0}$  we have $ \lim \sup _{t \rightarrow \tau}\|u(t)\|_{H^{\sigma-2}}<\infty $, then  $\tau<T_{0}$.
\end{lem}
Drawing upon the analysis presented in \cite{MR3612004}, for sufficiently small values of $\varepsilon \nu^{-2}$, it is feasible to obtain reliable estimates for the local solutions $q^{i}$ and $u^{i}$ up to $t=2$ without encountering any significant challenges.

\begin{lem}\label{lem4.4}
	For  $ \varepsilon \nu^{-2}$  sufficiently small and constants  $C_{0}, C_{1}$  sufficiently large (chosen below), the following estimates hold for  $t \in[0,2] $:
	
	$(1)$ the bounds on $q^i,\;i=1,\,2,\,3,$
	\begin{align}
		\left\|q^{1}(t)\right\|_{H^{\sigma-2}}& \leqslant 2\varepsilon,\nonumber\\
		\left\|q^{2}(t)\right\|_{H^{\sigma-2}}& \leqslant 2 C_{1} \varepsilon, \nonumber\\
		\left\|q^{3}(t)\right\|_{H^{\sigma-2}}& \leqslant 2 C_{0} \varepsilon.\nonumber
	\end{align}
	
	$(2)$ the bounds on $u^i,\;i=1,\,2,\,3,$
	\begin{align}
		\left\|u^{1}(t)\right\|_{H^{\sigma}}&  \leqslant 2 \varepsilon,\nonumber  \\ \left\|u^{2}(t)\right\|_{H^{\sigma}}&  \leqslant 2 C_{1} \varepsilon,
		\nonumber  \\
		\left\|u^{3}(t)\right\|_{H^{\sigma}}&  \leqslant 2 C_{0} \varepsilon.\nonumber
	\end{align}
\end{lem}
The proof of Lemma \ref{lem4.4} can be referred to Lemma 2.7 in subsection 2.7 in \cite{MR3612004},  where we omit the proof details. The local existence result shows that we only need to concern with the times  $T>1$ in following subsection.

It is worth noting that here we only introduce Fourier multipliers $m$ and $M$ such as in \cite{MR3612004}. Due to the interaction of perturbations near the Couette flow   and  the rotation effect of fluid, we are unable to construct suitable multipliers to handle some linear terms such as $ Q^{1}- \partial_{YY}^{L}U^{1}$ in $Q^{2}$ equation in \eqref{2.10}. Fortunately, when we consider the linear stability problem in section \ref{section3}, $Q^{1}$ and $Q^{2}$ equation are reduced to the linear symmetric structures by using introduced variable substitution \eqref{3.4}. This tells us that we can avoid dealing with tricky linear terms appeared in $Q^{1}$ and $Q^{2}$ equation. Inspired by these ideas,  the coupling between $U^{1}$ and $U^{2}$ is effectively addressed at the linearized level using symmetric variables
\begin{equation}\label{5.1}
	\check{K}^1= -|\nabla_{X, Z}| |\nabla_{L}| U^1, \quad \check{K}^2= -|\partial_{X}| |\nabla_{L}| U^2.
\end{equation}
These new variables play a crucial role in the analysis of the linearized system, as they admit a favorable energy  estimate (see the proof of the following Proposition \ref{pro4.1}).
Under the change of variable \eqref{5.1}, we obtain from \eqref{2.11} that
\begin{align}
&\partial_{t} \check{K}^1-\nu \Delta_{L} \check{K}^1+  |\partial_{X}| |\nabla_{X, Z}| |\nabla_{L}|^{-2} \check{K}^2\nonumber\\
&\quad  =|\nabla_{X, Z}| |\nabla_{L}|\left( U \cdot \nabla_{L} U^{1} \right)-|\nabla_{X, Z}| |\nabla_{L}|\partial_{X} |\nabla_{L}|^{-2}\left(\partial_{i}^{L} U^{j} \partial_{j}^{L} U^{i}\right),\label{HSX88} \\
&\partial_{t} \check{K}^2-\nu \Delta_{L} \check{K}^2-  |\partial_{X}| |\nabla_{X, Z}| |\nabla_{L}|^{-2} \check{K}^1 \nonumber\\
&\quad  =|\partial_{X}| |\nabla_{L}|\left( U \cdot \nabla_{L} U^{2} \right)-|\partial_{X}| |\nabla_{L}|\partial_{Y}^{L} |\nabla_{L}|^{-2}\left(\partial_{i}^{L} U^{j} \partial_{j}^{L} U^{i}\right) \label{HSX99},\\
&\partial_{t}U^{3}-\nu\Delta_{L}U^{3}+\partial_{Z}(-\Delta_{L})^{-1}\partial_{X}U^{2}+\partial_{Z}(-\Delta_{L})^{-1}\partial_{Y}^{L}U^{1}\nonumber\\
&\quad= -U\cdot\nabla_{L}U^{3}-\partial_{Z}(-\Delta_{L})^{-1}\left( \partial_{i}^{L}U_{j}\partial_{j}^{L}U_{i}\right)\label{HSX100}.
\end{align}
In particular, with \eqref{5.1}, $U^{1}$ and $U^{2}$ can be recovered from $\check{K}^1$ and $\check{K}^2$ as
\begin{equation} 
	\label{HSX110}
 U^1=-|\nabla_{X, Z}|^{-1} |\nabla_{L}|^{-1}\check{K}^1, \quad  U^2=-|\partial_{X}|^{-1}  |\nabla_{L}|^{-1}\check{K}^2.
\end{equation}
With the help of analysis above, we will show the following bootstrap result. In fact, Theorem \ref{1.2.} is mainly  proved by using the following bootstrap argument.  For convenient, we set $N \overset{\Delta}{=}\sigma-2>\frac{5}{2}$. 
\subsection{Bootstrap argument}\label{4.2}

\begin{pro}\label{pro4.1}
	Fix $C_{0}$ and $C_{1}$ large constants and let $T$ be the largest time $T\geqslant 1$ such that the following estimates hold on $[1,T]$:
	
	 $(1)$ the bounds on $\check{K}_{\neq}^{1, 2}$ and $Q_{\neq}^3$
	 \begin{align}
	 \left\|  M \check{K}_{\neq}^1 \right\|_{L^{\infty} H^N}+ \nu^{\frac{1}{2}} \left\|  M \nabla_{L} \check{K}_{\neq}^1 \right\|_{L^{2} H^N} + \left\| \sqrt{-\dot{M}M}  \check{K}_{\neq}^1  \right\|_{L^{2} H^N} & \leqslant 8\varepsilon,   \label{4..2}\\
	 \left\|  M \check{K}_{\neq}^2 \right\|_{L^{\infty} H^N}+ \nu^{\frac{1}{2}} \left\|  M \nabla_{L} \check{K}_{\neq}^2 \right\|_{L^{2} H^N}+ \left\| \sqrt{-\dot{M}M}  \check{K}_{\neq}^2  \right\|_{L^{2} H^N} & \leqslant 8\varepsilon.     \label{4..5}\\
	 \left\| m M Q_{\neq}^3 \right\|_{L^{\infty} H^N}+ \nu^{\frac{1}{2}} \left\| m M \nabla_{L} Q_{\neq}^3 \right\|_{L^{2} H^N}+ \left\| \sqrt{-\dot{M}M} m Q_{\neq}^3  \right\|_{L^{2} H^N} & \leqslant 8C_0 \varepsilon  \nu^{-\frac{1}{3}}.  \label{4..8}
	 \end{align}
	 
	 $(2)$ the bounds on $Q_{0}$
\begin{align}
\left\| Q_0^1 \right\|_{L^{\infty} H^N}+ \nu^{\frac{1}{2}} \left\| \nabla Q_0^1 \right\|_{L^{2} H^N} & \leqslant 8\varepsilon,    \label{4..1}\\
\left\| Q_0^2 \right\|_{L^{\infty} H^N}+ \nu^{\frac{1}{2}} \left\| \nabla Q_0^2 \right\|_{L^{2} H^N} & \leqslant 8C_1 \varepsilon  \nu^{-1},     \label{4..4}\\
\left\| Q_0^3 \right\|_{L^{\infty} H^N}+ \nu^{\frac{1}{2}} \left\| \nabla Q_0^3 \right\|_{L^{2} H^N} & \leqslant 8C_0 \varepsilon  \nu^{-1},     \label{4..7}
\end{align}

	$(3)$ the bounds on $U_{0}$  
	 \begin{align}
	 \left\| U_0^1 \right\|_{L^{\infty} H^{N-1}}+ \nu^{\frac{1}{2}} \left\| \nabla U_0^1 \right\|_{L^{2} H^{N-1}} & \leqslant 8\varepsilon,     \label{4..11}\\
	 \left\| U_0^2 \right\|_{L^{\infty} H^{N-1}}+ \nu^{\frac{1}{2}} \left\| U_0^2 \right\|_{L^{2} H^{N-1}}+ \nu^{\frac{1}{2}} \left\| \nabla U_0^2 \right\|_{L^{2} H^{N-1}} & \leqslant 8C_1 \varepsilon  \nu^{-1},     \label{4..12}\\
	 \left\| U_0^3 \right\|_{L^{\infty} H^{N-1}}+ \nu^{\frac{1}{2}} \left\| \nabla U_0^3 \right\|_{L^{2} H^{N-1}} & \leqslant 8C_0 \varepsilon  \nu^{-1}.     \label{4..13}
	 \end{align}
Assume that $\left\| u_{\mathrm{in}} \right\|_{H^{N+2}} \leqslant \varepsilon \leqslant \delta  \nu^{2}$, $\nu \in (0, 1)$ and that for some $T>1$, the  estimates \eqref{4..2}--\eqref{4..13} hold on $[0,T]$. Then for $\delta$ sufficiently small depending only on $\sigma$,  $C_0$ and  $C_1$, but not on $T$, there same estimates  hold with all the occurrences of 8 on the right-hand side replaced by 4.
\end{pro}
  
The proof of Theorem \ref{1.1.} follows directly from Proposition \ref{pro4.2}: By standard local well-posedness results which are established in Lemmas \ref{lem4.3} and \ref{lem4.4}, we can assume that there exists $T>1$ such that the assumptions \eqref{4..2}--\eqref{4..13} hold on   $[1,T]$. Proposition \ref{pro4.2} and the continuity of these norms imply that the set of times on \eqref{4..2}--\eqref{4..13} holds is closed, open-empty in $[1,\infty)$. Moreover, combined with the existence of local solution, the global solution is unique and regular on the time $[0,\infty)$.   

\begin{rem}
Most of these estimates are natural ones for the linearized problem. Given the above definition of multipliers, we consider (\ref{4..8}) which is typical. It comprises a global bound $m MQ_{\neq}^{3}$ in $H^{N}$ weighted by the multiplier $m$ (natural due to (\ref{2.10})$_{3}$ and the linear stretching as discussed in subsection \ref{4.1}), a bound of $\nu^{\frac{1}{2}}\|m M\nabla_{L}Q_{\neq}^{3}\|_{L^{2}H^{N}} $ accounting for the viscous dissipation, and a bound of $\|\sqrt{-\dot{M}M}m Q^{3}_{\neq}\|_{L^{2}H^{N}}$ corresponding to the dissipation-like structure arising from the multipliers as explained in subsection \ref{4.1}.
\end{rem}

\begin{rem}
The introduction of the new variable \eqref{5.1} results in the linear symmetry of $\check{K}^{1} $ and $\check{K}^{2}$ equations, see \eqref{HSX88} and \eqref{HSX99}, which plays a key role in the treatment of nonlinear stability problems. Due to the structure of $\check{K}^{1}$ and $\check{K}^{2}$,  we can close the a priori assumptions \eqref{4..2} and \eqref{4..5}. In a sense, we think that the effect of $\check{K}^{1} $ and $\check{K}^{2}$ is consistent to $Q^{1}$ and $Q^{2}$ respectively, and it is perfect to bypass the handling of the linear coupling terms appeared in \eqref{2.10}.
\end{rem}

\begin{rem}
The estimate \eqref{4..8} for $Q_{\neq}^3$ loss $ \nu^{-\frac{1}{3}}$ on the right-hand side compared to the linearized estimate. This loss occurs when estimating the lift up term $ \partial_{Z Y}^{L} U_{\neq}^1$ in section \ref{sec6}. Note that the loss might not be unique. Indeed, if the estimate for $Q_{\neq}^{3}$ of \eqref{4..8} loss $\nu^{-\iota}$, where the positive number  $\iota$ is chosen later, then  by using $U_{\neq}^1= -|\nabla_{X, Z}|^{-1} |\nabla_{L}|^{-1} \check{K}_{\neq}^1$, we need to treat
\begin{align*}
\mathcal{LU}3 =& \int_{1}^{T}\int m  M {\left\langle D \right\rangle}^{N} Q_{\neq}^3  m M {\left\langle D \right\rangle}^{N}  \partial_{Z Y}^{L} U_{\neq}^1 dVdt \nonumber \\
\lesssim &  \left\| m Q_{\neq}^3   \right\|_{L^2 H^N}  \left\| M \check{K}_{\neq}^1 \right\|_{L^2 H^N}   \nonumber \\
\lesssim &  C_0 \varepsilon  \nu^{-\iota} \nu^{-\frac{1}{6}}   \varepsilon \nu^{-\frac{1}{6}} \nonumber \\
\lesssim & \left(C_0 \varepsilon \nu^{-\iota}  \right)^2 \left( \frac{1}{C_0} \nu^{\iota-\frac{1}{3}} \right)   \nonumber \\
\lesssim & \left(C_0 \varepsilon \nu^{-\iota}  \right)^2,
\end{align*}
with  $\iota \geqslant \frac{1}{3}$ being required. In fact,  similar to the proof of proposition \ref{pro4.1} in section \ref{sec5}--\ref{sec7}, we can obtain $\iota \in [\frac{1}{3}, \frac{1}{2}]$ which is sufficient to close  the bootstrap assumption. Here we think that $\iota=\frac{1}{3}$ is optimal. 
\end{rem}

\begin{rem}
	When the time  $ t \lesssim \nu^{-1}$, the lift-up effect on  $u_0^{2}$ and $u_0^3$ (see (\ref{1.13}) or (\ref{1.*15}) in Theorem \ref{1.1.}))  is an important factor that makes the fluid unstable, which experience a linear growth for this modes.
\end{rem}

\begin{rem}
	All of the a priori assumptions  above can be closed when the initial condition satisfies $\left\|u_{\mathrm{in}}\right\|_{H^{\sigma}}<\delta \nu^{2}$. Here, we use the closed estimate (\ref{4..5}) for example. Indeed, in order to estimate this term, we need to control the following nonlinear pressure term: 
	\begin{align}\label{4**2}
	\mathcal{NLP}\check{K}_{\neq}^2 (U_0, U_{\neq}) =& -\int_{1}^{T}\int M {\left\langle D \right\rangle}^{N} \check{K}_{\neq}^2 M {\left\langle D \right\rangle}^{N}|\partial_{X}| |\nabla_{L}|\partial_{Y}^{L} |\nabla_{L}|^{-2} \partial_{j}^{L}\left(\partial_{i} U_0^{j}  U_{\neq}^{i} \right)_{\neq}dVdt \cdot 1_{i \neq 1} \nonumber \\
	\leqslant & \left\|  \nabla_{L} M \check{K}_{\neq}^2  \right\|_{L^2 H^N}\left(\left\| \nabla U_{0}^1  \right\|_{L^{\infty} H^N}+ \left\| \nabla U_{0}^2  \right\|_{L^{\infty} H^N}+ \left\| \nabla U_{0}^3  \right\|_{L^{\infty} H^N}\right)   \nonumber \\
	\quad&  \times \left( \left\| M \check{K}_{\neq}^2  \right\|_{L^{2} H^N}  + \left\| m M  Q_{\neq}^{3}  \right\|_{L^{2} H^N} \right) \nonumber \\
	\lesssim & \,\varepsilon \nu^{-\frac{1}{2}} \left( \varepsilon + C_1 \varepsilon  \nu^{-1} + C_0 \varepsilon  \nu^{-1} \right) \left(\varepsilon \nu^{-\frac{1}{6}}+ C_0 \varepsilon  \nu^{-\frac{1}{3}} \nu^{-\frac{1}{6}} \right) \nonumber \\
	\lesssim &\, \varepsilon^2 \left(\varepsilon \nu^{-\frac{2}{3}}+C_0 \varepsilon  \nu^{-\frac{5}{3}}+C_0 \varepsilon \nu^{-1}+  C_0^2 \varepsilon  \nu^{-2}\right)  \nonumber \\
	\lesssim &\,\varepsilon^2,
	\end{align}
	Due to both $U_0^2$ and $U_0^3$ loss $\nu^{-1}$ from the lift-up effect, as well as the term $\partial_{X} U_{\neq}^3$ losses $\nu^{-\frac{1}{2}}$ in (\ref{4..12}), (\ref{4..13}) and (\ref{4..8}), respectively,  it ultimately leads to the need to assume that the initial value satisfy $\left\|u_{\mathrm{in}}\right\|_{H^{\sigma}}<\delta  \nu^{2}$ with $\delta$ small enough in order to close the a priori assumption under nonlinear interactions.
    \end{rem}
    
    \begin{rem}
    Compared with the stability results of 3D Navier-Stoke equations,  the threshold index $2$  occurs in the initial data in Theorem \ref{1.2.} cannot be replaced by other real numbers less than $2$, due to the lift-up effect in the direction of $u_0^{2, 3}$ caused by the rotation.  Of course,  it  is mathematically and physically interesting  to study the optimal stability threshold of solutions of \eqref{1....3}, which requires new  methods and estimates and is left for future  research.
\end{rem}
     
\subsection{Choice of constants}\label{4.3.}

There constants have not been specified yet:  $\delta \geqslant \varepsilon \nu^{-2}$, which appears in the statement of Theorem \ref{1.2.}, and $C_{0}$ and  $C_{1}$ which appear in the above bootstrap estimate. In the course of the proof, we choose them small such that
\begin{align*}
	\frac{1}{C_0}+\frac{1}{C_1}+\frac{C_0^2}{C_1} \delta+C_0^2 \delta < \frac{1}{\bar{C}},
\end{align*}
for a universal constant $\bar{C}=\bar{C}(\delta)$ that depends only on $\delta$. Indeed, we first fixes $C_1$, then $C_0$ dependent on $C_1$, which $C_{0}>C_{1}$ and finally choose $\delta$ small relative to both.

\begin{pro}\label{pro4.2}
	Under the bootstrap hypotheses, for $\varepsilon  \nu^{-2}$ sufficiently small, the following estimates hold:
	
	$(1)$ the zero frequency velocity $U_0^{i}$ $(i=1, 2, 3)$:
		\begin{align}
		\left\| U_0^1 \right\|_{L^{\infty} H^{N+2}}+ \nu^{\frac{1}{2}} \left\| \nabla U_0^1 \right\|_{L^{2} H^{N+2}} & \lesssim \varepsilon,     \label{4..16}\\
		\left\| U_0^2 \right\|_{L^{\infty} H^{N+2}}+ \nu^{\frac{1}{2}} \left\| \nabla U_0^2 \right\|_{L^{2} H^{N+2}} & \lesssim C_1 \varepsilon  \nu^{-1},     \label{4..17}\\
		\left\| U_0^3 \right\|_{L^{\infty} H^{N+2}}+ \nu^{\frac{1}{2}} \left\| \nabla U_0^3 \right\|_{L^{2} H^{N+2}} & \lesssim C_0 \varepsilon  \nu^{-1}.     \label{4..18} 
		\end{align}
	
	$(2)$ the non-zero frequency velocity $U_{\neq}^i$ $(i=1, 2, 3)$:
		\begin{align}
		\left\| U_{\neq}^1 \right\|_{L^{\infty} H^N}+ \nu^{\frac{1}{2}} \left\| \nabla_{L} U_{\neq}^1 \right\|_{L^{2} H^N} + \left\| \sqrt{-\dot{M}M}  U_{\neq}^1  \right\|_{L^{2} H^N} & \lesssim \varepsilon,   \label{4..22}\\
		\left\|  U_{\neq}^2 \right\|_{L^{\infty} H^N}+ \nu^{\frac{1}{2}} \left\|  \nabla_{L}  U_{\neq}^2 \right\|_{L^{2} H^N}+ \left\| \sqrt{-\dot{M}M}  U_{\neq}^2  \right\|_{L^{2} H^N} & \lesssim \varepsilon,      \label{4..23}\\
		\left\| m \Delta_{L} U_{\neq}^3 \right\|_{L^{\infty} H^N}+ \nu^{\frac{1}{2}} \left\| m \nabla_{L} \Delta_{L} U_{\neq}^3 \right\|_{L^{2} H^N}+ \left\| \sqrt{-\dot{M}M} m \Delta_{L} U_{\neq}^3  \right\|_{L^{2} H^N} & \lesssim C_0 \varepsilon  \nu^{-\frac{1}{3}},  \label{4..21}\\
		\left\|  U_{\neq}^3 \right\|_{L^{\infty} H^N}+ \nu^{\frac{1}{2}} \left\| \nabla_{L} U_{\neq}^3 \right\|_{L^{2} H^N}+ \left\| \sqrt{-\dot{M}M}  U_{\neq}^3  \right\|_{L^{2} H^N} & \lesssim C_0 \varepsilon  \nu^{-\frac{1}{3}}.  \label{4..24}
		\end{align}
	
\end{pro}

\begin{proof}
      For any $1<s \leqslant N$, it holds
	\begin{align}
	\left\|  U_0^1  \right\|_{H^{s+2}} \leqslant \left\|  \Delta U_0^1  \right\|_{H^{s}}+\left\| U_0^1   \right\|_{L^2}=\left\|  Q_0^1  \right\|_{H^{s}}+\left\| U_0^1   \right\|_{L^2}.\nonumber
	\end{align}
	Therefore, by the bootstrap hypotheses \eqref{4..1} and \eqref{4..11}, the estimate \eqref{4..16} can be obtained. Similarly, other estimates on the zero frequency velocity follow from the bootstrap hypotheses.

From \eqref{HSX110}, we can immediately  get \eqref{4..22}--\eqref{4..21} by using the bootstrap hypotheses \eqref{4..2}, \eqref{4..5} and \eqref{4..8}. We have  $1  \lesssim  |m^{\frac{1}{2}}  \Delta_{L}|$ due to (\ref{lem4.1.11}). Then, the estimate \eqref{4..21} implies \eqref{4..24}. 
\end{proof}

The next proposition details  the enhanced dissipation of solutions.
\begin{pro}\label{pro4.3}
	Under the bootstrap hypotheses, the following additional estimates hold:
	
	$(1)$ the enhanced dissipation of the non-zero frequency $\check{K}_{\neq}^{1, 2}$ and $Q_{\neq}^3$:
		\begin{align}
		\left\|  \check{K}_{\neq}^{1, 2} \right\|_{L^{2} H^N} & \lesssim \varepsilon \nu^{-\frac{1}{6}},   \label{4..25}\\
		\left\| m Q_{\neq}^3 \right\|_{L^{2} H^N} & \lesssim C_0 \varepsilon  \nu^{-\frac{1}{2}};  \label{4..27}
		\end{align}
	
	$(2)$ the enhanced dissipation of the non-zero frequency $U_{\neq}$:
		\begin{align}
		\left\| \nabla_{L} U_{\neq}^{1, 2} \right\|_{L^{2} H^N} & \lesssim \varepsilon \nu^{-\frac{1}{6}},   \label{4..28}\\
		\left\| \Delta_{X, Z} U_{\neq}^3 \right\|_{L^{2} H^N} & \lesssim C_0\varepsilon  \nu^{-\frac{1}{2}};   \label{4..30}
		\end{align}
		
	$(3)$ the enhanced dissipation of  $tU_{\neq}$:
		\begin{align}
		\left\| t \partial_{X} U_{\neq}^{1, 2} \right\|_{L^{2} H^{N-1}} & \lesssim \varepsilon \nu^{-\frac{1}{6}},   \label{4..31}\\
		\left\| t \partial_{X} U_{\neq}^3 \right\|_{L^{2} H^{N-1}} & \lesssim C_0\varepsilon  \nu^{-\frac{5}{6}}.  \label{4..33} 
		\end{align}
\end{pro}

\begin{proof}
First,	by Corollary \ref{cor4.1},   \eqref{4..2}--\eqref{4..5}, we have
	\begin{align*}
	\left\| \check{K}_{\neq}^{1, 2} \right\|_{L^{2} H^N} \lesssim \nu^{-\frac{1}{6}} \left( \nu^{\frac{1}{2}} \left\|  \nabla_{L} \check{K}_{\neq}^{1, 2} \right\|_{L^2 H^{N}}+\left\|  \sqrt{-\dot{M} M} \check{K}_{\neq}^{1, 2} \right\|_{L^2 H^{N}} \right) \lesssim  \varepsilon \nu^{-\frac{1}{6}},
	\end{align*}
	and similarly for  $m Q_{\neq}^3$.
	
Next,  for any $k\neq 0$, by the fact $U_{\neq}^1=-|\nabla_{X, Z}|^{-1} |\nabla_{L}|^{-1} \check{K}_{\neq}^{1}$ and $U_{\neq}^2=-|\partial_{X}|^{-1} |\nabla_{L}|^{-1} \check{K}_{\neq}^{2}$, we have
	\begin{align}
	\left\| \nabla_{L} U_{\neq}^{1, 2} \right\|_{L^{2} H^N}  \lesssim &\left\|  M \check{K}_{\neq}^{1, 2} \right\|_{L^2 H^N}  \nonumber \\
	\lesssim & \nu^{-\frac{1}{6}} \left(\nu^{\frac{1}{2}} \left\| M \nabla_{L} \check{K}_{\neq}^{1, 2} \right\|_{L^{2} H^N} + \left\| \sqrt{-\dot{M}M} \check{K}_{\neq}^{1, 2}  \right\|_{L^{2} H^N} \right)  \nonumber \\
	\lesssim & \varepsilon \nu^{-\frac{1}{6}}, \nonumber
	\end{align}
which gives (\ref{4..28}).	
	
	Now we continue to estimate \eqref{4..30}. Notice that for any $k \neq 0$, we have
	\begin{align*}
	|k, l|^2 \leqslant \frac{k^2+l^2}{k^2+(\eta- kt)^2+l^2} |k, \eta- kt, l|^2,
	\end{align*}
	therefore, it is easy to deduce that
	\begin{align}
	\left\| \Delta_{X, Z} U_{\neq}^3 \right\|_{L^{2} H^N} \lesssim \left\| m \Delta_{L} U_{\neq}^3 \right\|_{L^{2} H^N} \lesssim C_0\varepsilon  \nu^{-\frac{1}{2}},\nonumber
	\end{align}
	by \eqref{4..21} and Corollary \ref{cor4.1}.
	
	In order to estimate \eqref{4..31}, we can use $|kt| \lesssim {\left\langle \eta- kt \right\rangle} {\left\langle \eta \right\rangle}$ and \eqref{4..28} to derive
	\begin{align*}
	\left\| t \partial_{X} U_{\neq}^{1, 2} \right\|_{L^{2} H^{N-1}} \lesssim \left\| \nabla_{L} U_{\neq}^{1, 2} \right\|_{L^2 H^{N}} \lesssim \varepsilon \nu^{-\frac{1}{6}}.
	\end{align*}
	Similarly, by   \eqref{lem4.1.11},  it holds
	\begin{align}
	\left\| t \partial_{X} U_{\neq}^3 \right\|_{L^{2} H^{N-1}} \lesssim \left\| \nabla_{L} m \Delta_{L} U_{\neq}^3   \right\|_{L^2 H^N} \lesssim C_0\varepsilon  \nu^{-\frac{5}{6}}.\nonumber
	\end{align}
\end{proof}
In fact, Theorem \ref{1.2.} can be directly derived from Propositions \ref{pro4.2} and \ref{pro4.3}. Therefore, we only need to prove Proposition \ref{4.1} and its proof will be given in the following sections.

\section{Energy estimates on $\check{K}_{\neq}^{1, 2}$ and $Q_{0}^{1, 2}$ }\label{sec5}
In this section, we want to prove that under the bootstrap hypotheses of Proposition \ref{pro4.1}, the estimates on $\check{K}_{\neq}^{1, 2}$ and $Q_0^{1, 2}$  hold (i.e., \eqref{4..2}--\eqref{4..5} and \eqref{4..1}--\eqref{4..4}), with 8 replaced by 4 on the right-hand side.

\subsection{$H^{N}$ estimates on $\check{K}_{\neq}^1$ and $\check{K}_{\neq}^2$}

First, we recall the equations of $\check{K}^1$ and $\check{K}^2$.
From \eqref{HSX88} and \eqref{HSX99}, we obtain the equations of  $\check{K}_{\neq}^1$ and $\check{K}_{\neq}^2$:
\begin{align}\label{5.2}
\left\{\begin{array}{l}
\partial_{t} \check{K}_{\neq}^1-\nu \Delta_{L} \check{K}_{\neq}^1+  |\partial_{X}| |\nabla_{X, Z}| |\nabla_{L}|^{-2} \check{K}_{\neq}^2\\
\quad  =|\nabla_{X, Z}| |\nabla_{L}|\left( U \cdot \nabla_{L} U^{1} \right)_{\neq}-|\nabla_{X, Z}| |\nabla_{L}|\partial_{X} |\nabla_{L}|^{-2}\left(\partial_{i}^{L} U^{j} \partial_{j}^{L} U^{i}\right)_{\neq}, \\
\partial_{t} \check{K}_{\neq}^2-\nu \Delta_{L} \check{K}_{\neq}^2-  |\partial_{X}| |\nabla_{X, Z}| |\nabla_{L}|^{-2} \check{K}_{\neq}^1\\
\quad  =|\partial_{X}| |\nabla_{L}|\left( U \cdot \nabla_{L} U^{2} \right)_{\neq}-|\partial_{X}| |\nabla_{L}|\partial_{Y}^{L} |\nabla_{L}|^{-2}\left(\partial_{i}^{L} U^{j} \partial_{j}^{L} U^{i}\right)_{\neq}.
\end{array}\right.
\end{align}
An energy estimate yields
\begin{align}\label{4..66}
&\frac{1}{2}\frac{d}{dt} \left( \left\|  M \check{K}_{\neq}^1  \right\|_{H^N}^2 +\left\|  M \check{K}_{\neq}^2  \right\|_{H^N}^2 \right) + \left\|  \sqrt{-\dot{M}M}  \check{K}_{\neq}^1  \right\|_{H^N}^2+ \left\|  \sqrt{-\dot{M}M}  \check{K}_{\neq}^2  \right\|_{H^N}^2 \nonumber \\
&= \int M {\left\langle D \right\rangle}^{N} \check{K}_{\neq}^1 M {\left\langle D \right\rangle}^{N} \Big[  \nu \Delta_{L} \check{K}_{\neq}^1-  |\partial_{X}| |\nabla_{X, Z}| |\nabla_{L}|^{-2} \check{K}_{\neq}^2+ |\nabla_{X, Z}| |\nabla_{L}|\left( U \cdot \nabla_{L} U^{1} \right)_{\neq} \nonumber \\
& \quad -|\nabla_{X, Z}| |\nabla_{L}|\partial_{X} |\nabla_{L}|^{-2}\left(\partial_{i}^{L} U^{j} \partial_{j}^{L} U^{i}\right)_{\neq}  \Big] dV \nonumber \\
& \quad + \int M {\left\langle D \right\rangle}^{N} \check{K}_{\neq}^2 M {\left\langle D \right\rangle}^{N} \Big[ \nu \Delta_{L} \check{K}_{\neq}^2+  |\partial_{X}| |\nabla_{X, Z}| |\nabla_{L}|^{-2} \check{K}_{\neq}^1+|\partial_{X}| |\nabla_{L}|\left( U \cdot \nabla_{L} U^{2} \right)_{\neq} \nonumber \\
& \quad -|\partial_{X}| |\nabla_{L}|\partial_{Y}^{L} |\nabla_{L}|^{-2}\left(\partial_{i}^{L} U^{j} \partial_{j}^{L} U^{i}\right)_{\neq} \Big] dV \nonumber \\
&=-\nu \left\|  \nabla_{L} M \check{K}_{\neq}^1  \right\|_{H^N}^2 -\nu \left\|  \nabla_{L} M \check{K}_{\neq}^2  \right\|_{H^N}^2+ \int M {\left\langle D \right\rangle}^{N} \check{K}_{\neq}^1 M {\left\langle D \right\rangle}^{N} \Big[ |\nabla_{X, Z}| |\nabla_{L}|\left( U \cdot \nabla_{L} U^{1} \right)_{\neq} \nonumber \\
&\quad- |\nabla_{X, Z}| |\nabla_{L}|\partial_{X} |\nabla_{L}|^{-2}\left(\partial_{i}^{L} U^{j} \partial_{j}^{L} U^{i}\right)_{\neq} \Big] dV + \int M {\left\langle D \right\rangle}^{N} \check{K}_{\neq}^2 M {\left\langle D \right\rangle}^{N} \Big[ |\partial_{X}| |\nabla_{L}|\left( U \cdot \nabla_{L} U^{2} \right)_{\neq} \nonumber \\
& \quad- |\partial_{X}| |\nabla_{L}|\partial_{Y}^{L} |\nabla_{L}|^{-2}\left(\partial_{i}^{L} U^{j} \partial_{j}^{L} U^{i}\right)_{\neq} \Big] dV
\end{align}
where we have used 
\begin{align*}
&\int  M {\left\langle D \right\rangle}^{N} \check{K}_{\neq}^1  M {\left\langle D \right\rangle}^{N} \partial_t \check{K}_{\neq}^1 dV+ \int  M {\left\langle D \right\rangle}^{N} \check{K}_{\neq}^2  M {\left\langle D \right\rangle}^{N} \partial_t \check{K}_{\neq}^2 dV    \\
&\quad=\frac{1}{2}\frac{d}{dt} \left\| M\check{K}_{\neq}^{1, 2}  \right\|_{H^N}^2+\left\| \sqrt{-\dot{M}M} \check{K}_{\neq}^{1, 2}  \right\|_{H^N}^2.
\end{align*}

Now, integrating \eqref{4..66} with respect  to $t$, it  yields
\begin{align}\label{5.4}
&\frac{1}{2} \left\|  M \check{K}_{\neq}^{1, 2} (T) \right\|_{H^N}^2 + \left\|  \sqrt{-\dot{M}M}  \check{K}_{\neq}^{1, 2} \right\|_{L^2 H^N}^2+\nu \left\|  \nabla_{L} M \check{K}_{\neq}^{1, 2}  \right\|_{L^2 H^N}^2 \nonumber \\
&=\frac{1}{2} \left\|  M \check{K}_{\neq}^{1, 2} (1) \right\|_{H^N}^2 + \int_{1}^{T}\int M {\left\langle D \right\rangle}^{N} \check{K}_{\neq}^1 M {\left\langle D \right\rangle}^{N}  |\nabla_{X, Z}| |\nabla_{L}|\left( U \cdot \nabla_{L} U^{1} \right)_{\neq}dVdt \nonumber \\
&\quad- \int_{1}^{T}\int M {\left\langle D \right\rangle}^{N} \check{K}_{\neq}^1 M {\left\langle D \right\rangle}^{N}|\nabla_{X, Z}| |\nabla_{L}|\partial_{X} |\nabla_{L}|^{-2}\left(\partial_{i}^{L} U^{j} \partial_{j}^{L} U^{i}\right)_{\neq} dVdt \nonumber \\
& \quad + \int_{1}^{T}\int M {\left\langle D \right\rangle}^{N} \check{K}_{\neq}^2 M {\left\langle D \right\rangle}^{N} |\partial_{X}| |\nabla_{L}|\left( U \cdot \nabla_{L} U^{2} \right)_{\neq} dVdt\nonumber \\
& \quad- \int_{1}^{T}\int M {\left\langle D \right\rangle}^{N} \check{K}_{\neq}^2 M {\left\langle D \right\rangle}^{N} |\partial_{X}| |\nabla_{L}|\partial_{Y}^{L} |\nabla_{L}|^{-2}\left(\partial_{i}^{L} U^{j} \partial_{j}^{L} U^{i}\right)_{\neq} dVdt \nonumber \\
&\overset{\Delta}{=}\frac{1}{2} \left\|  M \check{K}_{\neq}^{1, 2} (1) \right\|_{H^N}^2 + \mathcal{T}\check{K}_{\neq}^1+\mathcal{NLP}\check{K}_{\neq}^1+\mathcal{T}\check{K}_{\neq}^2+\mathcal{NLP}\check{K}_{\neq}^2,
\end{align}
where the right-hand side will be estimated term by term as follows.   For the estimate of  the transport  term $\mathcal{T}\check{K}_{\neq}^1$. We decompose $\mathcal{T}\check{K}_{\neq}^1$ into three parts
\begin{align}\label{5.5}
\mathcal{T}\check{K}_{\neq}^1=&\int_{1}^{T}\int M {\left\langle D \right\rangle}^{N} \check{K}_{\neq}^1 M {\left\langle D \right\rangle}^{N}  |\nabla_{X, Z}| |\nabla_{L}|\left( U \cdot \nabla_{L} U^{1} \right)_{\neq}dVdt  \nonumber \\
=& \int_{1}^{T}\int M {\left\langle D \right\rangle}^{N} \check{K}_{\neq}^1 M {\left\langle D \right\rangle}^{N} |\nabla_{X, Z}| |\nabla_{L}| \Big[ \left(U_0 \cdot \nabla_{L} U_{\neq}^{1} \right)_{\neq} \nonumber \\
\quad& + \left( U_{\neq} \cdot \nabla U_0^{1} \right)_{\neq}+ \left( U_{\neq} \cdot \nabla_{L} U_{\neq}^{1} \right)_{\neq}\Big] dVdt  \nonumber \\
\overset{\Delta}{=}& \mathcal{T}\check{K}_{\neq}^1(U_0, U_{\neq}^1)+\mathcal{T}\check{K}_{\neq}^1(U_{\neq}, U_0^1)+\mathcal{T}\check{K}_{\neq}^1(U_{\neq}, U_{\neq}^1).
\end{align}
For $\mathcal{T}\check{K}_{\neq}^1(U_0, U_{\neq}^1)$, we have
\begin{align}\label{5.6}
\mathcal{T}\check{K}_{\neq}^1(U_0, U_{\neq}^1) = & \int_{1}^{T}\int M {\left\langle D \right\rangle}^{N} \check{K}_{\neq}^1 M {\left\langle D \right\rangle}^{N} |\nabla_{X, Z}| |\nabla_{L}| \Big[ \left(U_0^1 \partial_{X} U_{\neq}^{1} \right)_{\neq} \nonumber \\
&\quad + \left(U_0^2  \partial_{Y}^{L} U_{\neq}^{1} \right)_{\neq}+\left(U_0^3  \partial_{Z} U_{\neq}^{1} \right)_{\neq} \Big] dVdt\nonumber \\
\overset{\Delta}{=} & \mathcal{T}\check{K}_{\neq}^1(U_0^1, U_{\neq}^1)+\mathcal{T}\check{K}_{\neq}^1(U_0^2, U_{\neq}^1)+\mathcal{T}\check{K}_{\neq}^1(U_0^3, U_{\neq}^1),
\end{align}
We use \eqref{2.3}, \eqref{4.13}, \eqref{4..2}, \eqref{4..16} and Corollary \ref{cor4.1} to deduce
\begin{align}\label{5.7}
\mathcal{T}\check{K}_{\neq}^1(U_0^1, U_{\neq}^1)&=-\int_{1}^{T}\int M {\left\langle D \right\rangle}^{N} \check{K}_{\neq}^1 M {\left\langle D \right\rangle}^{N} |\nabla_{X, Z}| |\nabla_{L}| \left(U_0^1 \partial_{X} |\nabla_{X, Z}|^{-1} |\nabla_{L}|^{-1}\check{K}_{\neq}^{1} \right)_{\neq}  dVdt \nonumber \\
& \leqslant \left\| \nabla_{L} M \check{K}_{\neq}^1 \right\|_{L^2 H^N} \left( \left\| U_0^1 \right\|_{L^{\infty} H^N} +\left\| \nabla U_0^1 \right\|_{L^{\infty} H^N} \right)\left\| M \check{K}_{\neq}^1 \right\|_{L^2 H^N} \nonumber \\
& \lesssim \varepsilon \nu^{-\frac{1}{2}} \varepsilon \varepsilon \nu^{-\frac{1}{6}} \nonumber \\
& \lesssim \varepsilon^2 \left( \varepsilon \nu^{-\frac{2}{3}}\right),
\end{align}
where $\varepsilon \nu^{-\frac{2}{3}}$ is sufficiently small.
For $\mathcal{T}\check{K}_{\neq}^1(U_0^2, U_{\neq}^1)$, using \eqref{4..2}, \eqref{4..17} and Corollary \ref{cor4.1}, it holds
\begin{align}\label{5.8}
\mathcal{T}\check{K}_{\neq}^1(U_0^2, U_{\neq}^1) = & -\int_{1}^{T}\int M {\left\langle D \right\rangle}^{N} \check{K}_{\neq}^1 M {\left\langle D \right\rangle}^{N} |\nabla_{X, Z}| |\nabla_{L}| \left(U_0^2  \partial_{Y}^{L} |\nabla_{X, Z}|^{-1} |\nabla_{L}|^{-1}\check{K}_{\neq}^{1} \right)_{\neq}  dVdt \nonumber \\
\leqslant & \left\| \nabla_{L} M \check{K}_{\neq}^1 \right\|_{L^2 H^N} \left( \left\| U_0^2 \right\|_{L^{\infty} H^N} +\left\| \nabla U_0^2 \right\|_{L^{\infty} H^N} \right)\left\| M \check{K}_{\neq}^1 \right\|_{L^2 H^N} \nonumber \\
\lesssim & \varepsilon \nu^{-\frac{1}{2}} C_1 \varepsilon  \nu^{-1} \varepsilon \nu^{-\frac{1}{6}} \nonumber \\
\lesssim & \varepsilon^2 \left( C_1 \varepsilon  \nu^{-\frac{5}{3}}\right),
\end{align}
where $C_1 \varepsilon  \nu^{-\frac{5}{3}}$ is sufficiently small.
Finally, we use \eqref{4..2}, \eqref{4..18} and Corollary \ref{cor4.1} to deduce
\begin{align}\label{5.9}
\mathcal{T}\check{K}_{\neq}^1(U_0^3, U_{\neq}^1)= & -\int_{1}^{T}\int M {\left\langle D \right\rangle}^{N} \check{K}_{\neq}^1 M {\left\langle D \right\rangle}^{N} |\nabla_{X, Z}| |\nabla_{L}| \left(U_0^3 \partial_{Z} |\nabla_{X, Z}|^{-1} |\nabla_{L}|^{-1}\check{K}_{\neq}^{1} \right)_{\neq}  dVdt\nonumber \\
\leqslant& \left\| \nabla_{L} M \check{K}_{\neq}^1 \right\|_{L^2 H^N}  \left\| U_0^3 \right\|_{L^{\infty} H^{N+1}} \left\| M \check{K}_{\neq}^1 \right\|_{L^2 H^N} \nonumber \\
\lesssim & \varepsilon \nu^{-\frac{1}{2}} C_0 \varepsilon  \nu^{-1} \varepsilon \nu^{-\frac{1}{6}} \nonumber \\
\lesssim & \varepsilon^2 \left( C_0 \varepsilon  \nu^{-\frac{5}{3}} \right).
\end{align}
Next, substituting \eqref{5.7}--\eqref{5.9} into \eqref{5.6}, it yields
\begin{align*}
\mathcal{T}\check{K}_{\neq}^1(U_0, U_{\neq}^1) \lesssim \varepsilon^2.
\end{align*}

To estimate the term $\mathcal{T}\check{K}_{\neq}^1(U_{\neq}, U_0^1)$, we decompose it into two parts
\begin{align}\label{5.10}
\mathcal{T}\check{K}_{\neq}^1(U_{\neq}, U_0^1) = & \int_{1}^{T}\int  M {\left\langle D \right\rangle}^{N} \check{K}_{\neq}^1  M {\left\langle D \right\rangle}^{N} \left( U_{\neq}^2 \partial_{Y} U_0^{1}+ U_{\neq}^3 \partial_{Z} U_0^{1} \right) dVdt  \nonumber \\
\overset{\Delta}{=}& \mathcal{T}\check{K}_{\neq}^1(U_{\neq}^2, U_0^1)+\mathcal{T}\check{K}_{\neq}^1(U_{\neq}^3, U_0^1),
\end{align}
where the fact $\partial_{X} U_0^1=0$ is used.
For $\mathcal{T}\check{K}_{\neq}^1(U_{\neq}^2, U_0^1)$, using \eqref{4..2}, \eqref{4..5}, \eqref{4..16}, \eqref{5.1} and Corollary \ref{cor4.1}, it holds
\begin{align}\label{5..11}
\mathcal{T}\check{K}_{\neq}^1(U_{\neq}^2, U_0^1)=&-\int_{1}^{T}\int M {\left\langle D \right\rangle}^{N} \check{K}_{\neq}^1 M {\left\langle D \right\rangle}^{N} |\nabla_{X, Z}| |\nabla_{L}| \left(|\partial_{X}|^{-1} |\nabla_{L}|^{-1}\check{K}_{\neq}^{2} \partial_{Y} U_0^1 \right)_{\neq}  dVdt\nonumber \\
\leqslant& \left\| \nabla_{L} M \check{K}_{\neq}^1 \right\|_{L^2 H^N} \left( \left\| \nabla U_0^1 \right\|_{L^{\infty} H^N}+ \left\| U_0^1 \right\|_{L^{\infty} H^{N+2}}\right) \left\| M \check{K}_{\neq}^2 \right\|_{L^2 H^N} \nonumber \\
\lesssim & \varepsilon \nu^{-\frac{1}{2}} \varepsilon \varepsilon \nu^{-\frac{1}{6}} \nonumber \\
\lesssim & \varepsilon^2 \left( \varepsilon \nu^{-\frac{2}{3}}\right),
\end{align}
which suffices for $\varepsilon \nu^{-\frac{2}{3}}$  sufficiently small. For $\mathcal{T}\check{K}_{\neq}^1(U_{\neq}^3, U_0^1)$, we use    to obtain
\begin{align}\label{5.12}
\mathcal{T}\check{K}_{\neq}^1(U_{\neq}^3, U_0^1)=& \int_{1}^{T}\int M {\left\langle D \right\rangle}^{N} \check{K}_{\neq}^1 M {\left\langle D \right\rangle}^{N} |\nabla_{X, Z}| |\nabla_{L}| \left(U_{\neq}^{3} \partial_{Z} U_0^1 \right)_{\neq}  dVdt\nonumber \\
\leqslant& \left\| \nabla_{L} M \check{K}_{\neq}^1 \right\|_{L^2 H^N}  \left\|  U_0^1 \right\|_{L^{\infty} H^{N+1}} \left\| M \nabla_{L} U_{\neq}^3 \right\|_{L^2 H^N} \nonumber \\
\quad & +  \left\| \nabla_{L} M \check{K}_{\neq}^1 \right\|_{L^2 H^N}  \left\|  U_0^1 \right\|_{L^{\infty} H^{N+2}} \left\| M U_{\neq}^3 \right\|_{L^2 H^N}  \nonumber \\
\lesssim & \varepsilon \nu^{-\frac{1}{2}} C_0 \varepsilon  \nu^{-\frac{1}{3}} \nu^{-\frac{1}{2}} \varepsilon+ \varepsilon \nu^{-\frac{1}{2}} C_0 \varepsilon  \nu^{-\frac{1}{3}} \nu^{-\frac{1}{6}} \varepsilon \nonumber \\
\lesssim & \varepsilon^2 \left( C_0 \varepsilon  \nu^{-\frac{4}{3}} + C_0 \varepsilon  \nu^{-1}\right) \nonumber \\
\lesssim & \varepsilon^2,
\end{align}
which $C_0 \varepsilon  \nu^{-\frac{4}{3}}$ is sufficiently small. Hence, we substitute \eqref{5..11}--\eqref{5.12} into \eqref{5.10} and obtain
\begin{align*}
\mathcal{T}\check{K}_{\neq}^1(U_{\neq}, U_{0}^1) \lesssim \varepsilon^2.
\end{align*}

When we estimate   $\mathcal{T}\check{K}_{\neq}^1(U_{\neq}, U_{\neq}^1)$, it can be divided into three parts
\begin{align}\label{5.13}
\mathcal{T}\check{K}_{\neq}^1(U_{\neq}, U_{\neq}^1)= &\int_{1}^{T}\int M {\left\langle D \right\rangle}^{N} \check{K}_{\neq}^1 M {\left\langle D \right\rangle}^{N} |\nabla_{X, Z}| |\nabla_{L}|  \left( U_{\neq} \cdot \nabla_{L} U_{\neq}^{1} \right)_{\neq} dVdt \nonumber \\
=& \int_{1}^{T}\int M {\left\langle D \right\rangle}^{N} \check{K}_{\neq}^1 M {\left\langle D \right\rangle}^{N} |\nabla_{X, Z}| |\nabla_{L}|  \left(U_{\neq}^1 \partial_{X} U_{\neq}^{1}+U_{\neq}^2 \partial_{Y}^{L} U_{\neq}^{1}+ U_{\neq}^3 \partial_{Z} U_{\neq}^{1} \right)_{\neq} dVdt  \nonumber \\
\overset{\Delta}{=}& \mathcal{T}\check{K}_{\neq}^1(U_{\neq}^1, U_{\neq}^1)+\mathcal{T}\check{K}_{\neq}^1(U_{\neq}^2, U_{\neq}^1)+\mathcal{T}\check{K}_{\neq}^1(U_{\neq}^3, U_{\neq}^1).
\end{align}
For $\mathcal{T}\check{K}_{\neq}^1(U_{\neq}^1, U_{\neq}^1)$, by virtue of   \eqref{4..2}, \eqref{5.1} and Corollary \ref{cor4.1}, we have   
\begin{align}\label{5.14}
\mathcal{T}\check{K}_{\neq}^1(U_{\neq}^1, U_{\neq}^1) =& \int_{1}^{T}\int M {\left\langle D \right\rangle}^{N} \check{K}_{\neq}^1 M {\left\langle D \right\rangle}^{N} |\nabla_{X, Z}| |\nabla_{L}|  \Big( |\nabla_{X, Z}|^{-1} |\nabla_{L}|^{-1} \check{K}_{\neq}^1 \nonumber \\
& \quad \partial_{X} |\nabla_{X, Z}|^{-1} |\nabla_{L}|^{-1} \check{K}_{\neq}^1 \Big)_{\neq} dVdt \nonumber \\
\leqslant & \left\|  \nabla_{L} M \check{K}_{\neq}^1  \right\|_{L^2 H^N} \left\| M \check{K}_{\neq}^{1}  \right\|_{L^{\infty} H^N} \left\| M \check{K}_{\neq}^1  \right\|_{L^2 H^N}  \nonumber \\
\lesssim & \varepsilon \nu^{-\frac{1}{2}} \varepsilon \varepsilon \nu^{-\frac{1}{6}} \nonumber \\
\lesssim & \varepsilon^2 \left( \varepsilon \nu^{-\frac{2}{3}}\right).
\end{align}
Then we consider $\mathcal{T}\check{K}_{\neq}^1(U_{\neq}^2, U_{\neq}^1)$. By using \eqref{4..2}, \eqref{4..5}, \eqref{5.1} and Corollary \ref{cor4.1}, one obtains
\begin{align}\label{5..15}
\mathcal{T}\check{K}_{\neq}^1(U_{\neq}^2, U_{\neq}^1) =& \int_{1}^{T}\int M {\left\langle D \right\rangle}^{N} \check{K}_{\neq}^1 M {\left\langle D \right\rangle}^{N} |\nabla_{X, Z}| |\nabla_{L}|  \Big( |\partial_{X}|^{-1} |\nabla_{L}|^{-1} \check{K}_{\neq}^2 \nonumber \\
& \quad \partial_{Y}^{L} |\nabla_{X, Z}|^{-1} |\nabla_{L}|^{-1} \check{K}_{\neq}^1 \Big)_{\neq} dVdt \nonumber \\
\leqslant & \left\|  \nabla_{L} M \check{K}_{\neq}^1  \right\|_{L^2 H^N} \left\| M \check{K}_{\neq}^{2}  \right\|_{L^{\infty} H^N} \left\| M \check{K}_{\neq}^1  \right\|_{L^2 H^N}  \nonumber \\
\lesssim & \varepsilon \nu^{-\frac{1}{2}} \varepsilon \varepsilon \nu^{-\frac{1}{6}} \nonumber \\
\lesssim & \varepsilon^2 \left( \varepsilon \nu^{-\frac{2}{3}}\right).
\end{align}
Similarly, by using \eqref{4..2}, \eqref{4..24}, Corollary \ref{cor4.1} and the upper bound of $M$, the term $\mathcal{T}\check{K}_{\neq}^1(U_{\neq}^3, U_{\neq}^1)$ can be treated as
\begin{align}\label{5.16}
\mathcal{T}\check{K}_{\neq}^1(U_{\neq}^3, U_{\neq}^1) =& - \int_{1}^{T}\int M {\left\langle D \right\rangle}^{N} \check{K}_{\neq}^1 M {\left\langle D \right\rangle}^{N} |\nabla_{X, Z}| |\nabla_{L}|  \Big( U_{\neq}^3 \partial_{Z} |\nabla_{X, Z}|^{-1} |\nabla_{L}|^{-1} \check{K}_{\neq}^1 \Big)_{\neq} dVdt \nonumber \\
\leqslant & \left\|  \nabla_{L} M \check{K}_{\neq}^1  \right\|_{L^2 H^N} \left\| \nabla_{L} M U_{\neq}^{3}  \right\|_{L^{2} H^N} \left\| M \check{K}_{\neq}^1  \right\|_{L^{\infty} H^N} \nonumber \\
\quad& +\left\|  \nabla_{L} M \check{K}_{\neq}^1  \right\|_{L^2 H^N} \left\|  M U_{\neq}^{3}  \right\|_{L^{2} H^N} \left\| M \check{K}_{\neq}^1  \right\|_{L^{\infty} H^N}  \nonumber \\
\lesssim & \varepsilon \nu^{-\frac{1}{2}} C_0 \varepsilon  \nu^{-\frac{1}{3}} \nu^{-\frac{1}{2}} \varepsilon + \varepsilon \nu^{-\frac{1}{2}} C_0 \varepsilon  \nu^{-\frac{1}{3}} \nu^{-\frac{1}{6}} \varepsilon \nonumber \\
\lesssim & \varepsilon^2 \left( C_0 \varepsilon  \nu^{-\frac{4}{3}}+ C_0 \varepsilon  \nu^{-1}\right),
\end{align}
which suffices for $C_0 \varepsilon  \nu^{-\frac{4}{3}}$ sufficiently small. Hence, combining with \eqref{5.13}--\eqref{5.16}, we immediately have
\begin{align*}
\mathcal{T}\check{K}_{\neq}^1(U_{\neq}, U_{\neq}^1) \lesssim \varepsilon^2.
\end{align*}

We estimate  the nonlinear pressure term $\mathcal{NLP}\check{K}_{\neq}^1$ which divides into
\begin{align}\label{5.17}
\mathcal{NLP}\check{K}_{\neq}^1 =& - \int_{1}^{T}\int M {\left\langle D \right\rangle}^{N} \check{K}_{\neq}^1 M {\left\langle D \right\rangle}^{N}|\nabla_{X, Z}| |\nabla_{L}|\partial_{X} |\nabla_{L}|^{-2}\left(\partial_{i}^{L} U^{j} \partial_{j}^{L} U^{i}\right)_{\neq} dVdt \nonumber \\
=&- \int_{1}^{T}\int M {\left\langle D \right\rangle}^{N} \check{K}_{\neq}^1 M {\left\langle D \right\rangle}^{N}|\nabla_{X, Z}| |\nabla_{L}|\partial_{X} |\nabla_{L}|^{-2}\Big[\partial_{i} U_0^{j} \partial_{j}^{L} U_{\neq}^{i} \nonumber \\
& +  \partial_{i}^{L} U_{\neq}^{j} \partial_{j} U_0^{i}+\partial_{i}^{L} U_{\neq}^{j} \partial_{j}^{L} U_{\neq}^{i}\Big]_{\neq} dVdt \nonumber \\
\overset{\Delta}{=}& \mathcal{NLP}\check{K}_{\neq}^1 (U_0, U_{\neq})+\mathcal{NLP}\check{K}_{\neq}^1 (U_{\neq}, U_{0})+\mathcal{NLP}\check{K}_{\neq}^1 (U_{\neq}, U_{\neq}).
\end{align}
First, we divide $\mathcal{NLP}\check{K}_{\neq}^1 (U_0, U_{\neq})$ into the following four terms
\begin{align}\label{5.18}
\mathcal{NLP}\check{K}_{\neq}^1 (U_0, U_{\neq})=&  \int_{1}^{T}\int M {\left\langle D \right\rangle}^{N}|\nabla_{L}| \check{K}_{\neq}^1 M {\left\langle D \right\rangle}^{N}|\nabla_{X, Z}| \partial_{X} |\nabla_{L}|^{-2}\left(\partial_{i} U_0^{j} \partial_{j}^{L} U_{\neq}^{i} \right)_{\neq} dVdt \cdot 1_{i \neq 1} \nonumber \\
=& \int_{1}^{T}\int M {\left\langle D \right\rangle}^{N}|\nabla_{L}| \check{K}_{\neq}^1 M {\left\langle D \right\rangle}^{N}|\nabla_{X, Z}| \partial_{X} |\nabla_{L}|^{-2}\Big[\partial_{Y} U_0^{1} \partial_{X} U_{\neq}^{2}  \nonumber \\
\quad &+\partial_{Y} U_0^{2} \partial_{Y}^{L} U_{\neq}^{2}+\partial_{Y} U_0^{3} \partial_{Z} U_{\neq}^{2}+\partial_{Z} U_0^{j} \partial_{j}^{L} U_{\neq}^{3} \Big]_{\neq} dVdt  \nonumber \\
\overset{\Delta}{=}& \mathcal{NLP}\check{K}_{\neq}^1 (U_0^1, U_{\neq}^2)+\mathcal{NLP}\check{K}_{\neq}^1 (U_0^2, U_{\neq}^2)+\mathcal{NLP}\check{K}_{\neq}^1 (U_0^3, U_{\neq}^2) \nonumber \\
\quad &+\mathcal{NLP}\check{K}_{\neq}^1 (U_0, U_{\neq}^3).
\end{align}
 
In order to estimate $\mathcal{NLP}\check{K}_{\neq}^1 (U_0^1, U_{\neq}^2)$, we use  \eqref{4..2}--\eqref{4..5}, \eqref{4..16} and Corollary \ref{cor4.1} to obtain
\begin{align}\label{5.19}
\mathcal{NLP}\check{K}_{\neq}^1 (U_0^1, U_{\neq}^2) =& \int_{1}^{T}\int M {\left\langle D \right\rangle}^{N}|\nabla_{L}| \check{K}_{\neq}^1 M {\left\langle D \right\rangle}^{N}|\nabla_{X, Z}| \partial_{X} |\nabla_{L}|^{-2}\Big(\partial_{Y} U_0^{1} \nonumber \\
& \quad \partial_{X} |\partial_{X}|^{-1} |\nabla_{L}|^{-1}\check{K}_{\neq}^{2} \Big)_{\neq}   dVdt \nonumber \\
\leqslant & \left\|  M \nabla_{L} \check{K}_{\neq}^1 \right\|_{L^2 H^N} \left\| \partial_{Y} U_{0}^1 \right\|_{L^{\infty} H^N} \left\| M \check{K}_{\neq}^2 \right\|_{L^2 H^N}  \nonumber \\
\lesssim & \varepsilon \nu^{-\frac{1}{2}} \varepsilon  \varepsilon \nu^{-\frac{1}{6}} \nonumber \\
\lesssim & \varepsilon^2\left( \varepsilon \nu^{-\frac{2}{3}}\right) \nonumber \\
\lesssim & \varepsilon^2,
\end{align}
where  $\varepsilon  \nu^{-\frac{2}{3}}$ is sufficiently small.

For $\mathcal{NLP}\check{K}_{\neq}^1 (U_0^2, U_{\neq}^2)$, by using \eqref{4..2},  \eqref{4..5},  \eqref{4..17}, \eqref{5.1} and Corollary \ref{cor4.1},  we get
\begin{align}\label{5.20}
\mathcal{NLP}\check{K}_{\neq}^1 (U_0^2, U_{\neq}^2) =& -\int_{1}^{T}\int M {\left\langle D \right\rangle}^{N}|\nabla_{L}| \check{K}_{\neq}^1 M {\left\langle D \right\rangle}^{N}|\nabla_{X, Z}| \partial_{X} |\nabla_{L}|^{-2}\Big(\partial_{Y} U_0^{2} \nonumber \\
& \quad \partial_{Y}^{L} |\partial_{X}|^{-1} |\nabla_{L}|^{-1}\check{K}_{\neq}^{2} \Big)_{\neq}   dVdt \nonumber \\
\leqslant & \left\|  M \nabla_{L} \check{K}_{\neq}^1 \right\|_{L^2 H^N} \left\| \partial_{Y} U_{0}^2 \right\|_{L^{\infty} H^N} \left\| M \check{K}_{\neq}^2 \right\|_{L^2 H^N}  \nonumber \\
\lesssim & \varepsilon \nu^{-\frac{1}{2}} C_1 \varepsilon  \nu^{-1} \varepsilon \nu^{-\frac{1}{6}} \nonumber \\
\lesssim & \varepsilon^2\left( C_1 \varepsilon  \nu^{-\frac{5}{3}}\right) \nonumber \\
\lesssim & \varepsilon^2,
\end{align}
which suffices for   $ C_1 \varepsilon  \nu^{-\frac{5}{3}}$ sufficiently small. To estimate $\mathcal{NLP}\check{K}_{\neq}^1 (U_0^3, U_{\neq}^2)$, using \eqref{4..2}, \eqref{4..5}, \eqref{4..18}, \eqref{5.1} and  Corollary \ref{cor4.1} yields
\begin{align}\label{5.21}
\mathcal{NLP}\check{K}_{\neq}^1 (U_0^3, U_{\neq}^2) & -\int_{1}^{T}\int M {\left\langle D \right\rangle}^{N}|\nabla_{L}| \check{K}_{\neq}^1 M {\left\langle D \right\rangle}^{N}|\nabla_{X, Z}| \partial_{X} |\nabla_{L}|^{-2}\Big(\partial_{Y} U_0^{3} \nonumber \\
& \quad \partial_{Z} |\partial_{X}|^{-1} |\nabla_{L}|^{-1}\check{K}_{\neq}^{2} \Big)_{\neq}   dVdt \nonumber \\
\leqslant & \left\|  M \nabla_{L} \check{K}_{\neq}^1 \right\|_{L^2 H^N} \left\| \partial_{Y} U_{0}^3 \right\|_{L^{\infty} H^N} \left\| M \check{K}_{\neq}^2 \right\|_{L^2 H^N}  \nonumber \\
\lesssim & \varepsilon \nu^{-\frac{1}{2}} C_0 \varepsilon  \nu^{-1} \varepsilon \nu^{-\frac{1}{6}} \nonumber \\
\lesssim & \varepsilon^2\left( C_0 \varepsilon  \nu^{-\frac{5}{3}}\right) \nonumber \\
\lesssim & \varepsilon^2,
\end{align}
where  $C_0 \varepsilon  \nu^{-\frac{5}{3}}$ is sufficiently small.
Using Corollary \ref{cor4.1}, \eqref{4..2}, \eqref{4..8}, \eqref{4..16}--\eqref{4..18}, $\partial_{X}U_0^j=0$, the fact $\big| \widehat{\partial_{X} U_{\neq}^3 }\big| \lesssim \big| \widehat{m M Q_{\neq}^3} \big|$ and the upper bound \eqref{4.13} of $M$ yields
\begin{align}\label{5..22}
\mathcal{NLP}\check{K}_{\neq}^1 (U_0, U_{\neq}^3)=& \int_{1}^{T}\int M {\left\langle D \right\rangle}^{N}|\nabla_{L}| \check{K}_{\neq}^1 M {\left\langle D \right\rangle}^{N}|\nabla_{X, Z}| \partial_{X} |\nabla_{L}|^{-2}\left(\partial_{Z} U_0^{j} \partial_{j}^{L} U_{\neq}^{3} \right)_{\neq}   dVdt \nonumber \\
=& \int_{1}^{T}\int M {\left\langle D \right\rangle}^{N}|\nabla_{L}| \check{K}_{\neq}^1 M {\left\langle D \right\rangle}^{N}|\nabla_{X, Z}| |\nabla_{L}|^{-2} \partial_{j}^{L}\left(\partial_{Z} U_0^{j}  \partial_{X} U_{\neq}^{3} \right)_{\neq}   dVdt \nonumber \\
\leqslant & \left\|  M \nabla_{L} \check{K}_{\neq}^1 \right\|_{L^2 H^N} \left\|  U_{0}^{j} \right\|_{L^{\infty} H^{N+1}} \left\| m M Q_{\neq}^3 \right\|_{L^2 H^N}  \nonumber \\
\lesssim & \varepsilon \nu^{-\frac{1}{2}} \left( \varepsilon+C_0 \varepsilon  \nu^{-1} \right) C_0 \varepsilon  \nu^{-\frac{1}{3}}  \nu^{-\frac{1}{6}} \nonumber \\
\lesssim & \varepsilon^2\left( C_0^2 \varepsilon  \nu^{-2}\right) \nonumber \\
\lesssim & \varepsilon^2.
\end{align}
Therefore, we substitute \eqref{5.19}--\eqref{5..22} into \eqref{5.18} to obtain
\begin{align}\label{5.25}
\mathcal{NLP}\check{K}_{\neq}^1 (U_0, U_{\neq}) \lesssim \varepsilon^2.
\end{align}
In  the same way, we have 
\begin{align}\label{5..26}
\mathcal{NLP}\check{K}_{\neq}^1 (U_{\neq}, U_{0}) \lesssim \varepsilon^2.
\end{align}

For $\mathcal{NLP}\check{K}_{\neq}^1 (U_{\neq}, U_{\neq})$, note that the fact $\big| \widehat{\Delta_{X, Z} U_{\neq}^3 }\big| \lesssim \big| \widehat{m M Q_{\neq}^3} \big|$. And we use \eqref{HSX110}, \eqref{4..2}--\eqref{4..8},  \eqref{4..22}--\eqref{4..24} and Corollary \ref{cor4.1} to get
\begin{align}\label{5.27}
\mathcal{NLP}\check{K}_{\neq}^1 (U_{\neq}, U_{\neq})=& \int_{1}^{T}\int M {\left\langle D \right\rangle}^{N} |\nabla_{L}|\check{K}_{\neq}^1 M {\left\langle D \right\rangle}^{N}|\nabla_{X, Z}| \partial_{X} |\nabla_{L}|^{-2} \partial_{i}^{L} \partial_{j}^{L}\left( U_{\neq}^{j}  U_{\neq}^{i}\right)_{\neq} dVdt\nonumber \\
\leqslant & \, 2 \left\|  M \nabla_{L} \check{K}_{\neq}^1 \right\|_{L^{2} H^N} \big( \left\|  |\nabla_{X, Z}| \partial_{X} U_{\neq}^{j} \right\|_{L^2 H^N}  \left\|   U_{\neq}^i \right\|_{L^{\infty} H^N} \nonumber\\
\quad &+  \left\| \partial_{X} U_{\neq}^{j} \right\|_{L^{\infty} H^N}  \left\| |\nabla_{X, Z}|  U_{\neq}^i \right\|_{L^2 H^N} \big) \nonumber \\
\leqslant & \, 2 \left\|  M \nabla_{L} \check{K}_{\neq}^1 \right\|_{L^{2} H^N} \Big[ \Big( \left\| -|\nabla_{X, Z}| \partial_{X} |\nabla_{X, Z}|^{-1} |\nabla_{L}|^{-1} \check{K}_{\neq}^{1} \right\|_{L^{2} H^N}  \nonumber \\
\quad &+\left\| - |\nabla_{X, Z}|\partial_{X} |\partial_{X}|^{-1} |\nabla_{L}|^{-1} \check{K}_{\neq}^{2} \right\|_{L^{2} H^N}+ \left\| m M Q_{\neq}^{3} \right\|_{L^{2} H^N} \Big) \nonumber \\
\quad  & \times\left(\left\|  U_{\neq}^{1, 2} \right\|_{L^{\infty} H^N}+ \left\|    U_{\neq}^{3} \right\|_{L^{\infty} H^N} \right) \nonumber \\
\quad& + \Big( \left\| -|\nabla_{X, Z}| |\nabla_{X, Z}|^{-1} |\nabla_{L}|^{-1} \check{K}_{\neq}^{1} \right\|_{L^{2} H^N}  \nonumber \\
\quad &+\left\| - |\nabla_{X, Z}| |\partial_{X}|^{-1} |\nabla_{L}|^{-1} \check{K}_{\neq}^{2} \right\|_{L^{2} H^N}+ \left\| m M Q_{\neq}^{3} \right\|_{L^{2} H^N} \Big) \nonumber \\
\quad  & \times\left(\left\| M\check{K}_{\neq}^{1, 2} \right\|_{L^{\infty} H^N}+ \left\|    m M Q_{\neq}^{3} \right\|_{L^{\infty} H^N} \right) \Big] \nonumber \\
\leqslant & \, 2 \left\|  M \nabla_{L} \check{K}_{\neq}^1 \right\|_{L^{2} H^N} \Big( \left\| M \check{K}_{\neq}^{1, 2} \right\|_{L^2 H^N} + \left\| m M Q_{\neq}^{3} \right\|_{L^2 H^N} \Big) \nonumber \\
\quad&  \times \left(\left\|    U_{\neq}^{1, 2} \right\|_{L^{\infty} H^N}+ \left\|  U_{\neq}^{3} \right\|_{L^{\infty} H^N} + \left\| M\check{K}_{\neq}^{1, 2} \right\|_{L^{\infty} H^N} + \left\| m M Q_{\neq}^3 \right\|_{L^{\infty} H^N}\right) \nonumber \\
\lesssim & \varepsilon \nu^{-\frac{1}{2}} \left( \varepsilon \nu^{-\frac{1}{6}} +C_0 \varepsilon  \nu^{-\frac{1}{3}} \nu^{-\frac{1}{6}}\right) \left( \varepsilon +C_0 \varepsilon  \nu^{-\frac{1}{3}} \right)  \nonumber \\
\lesssim & \varepsilon^2 \left(\varepsilon \nu^{-\frac{2}{3}} +C_0 \varepsilon  \nu^{-1}+ +C_0^2 \varepsilon  \nu^{-\frac{4}{3}}\right)  \nonumber \\
\lesssim &\varepsilon^2,
\end{align}
which suffices for $C_0^2 \varepsilon  \nu^{-\frac{4}{3}}$ sufficiently small.

We then  estimate  the transport term $\mathcal{T}\check{K}_{\neq}^2$. Note that
\begin{align}\label{5.28}
\mathcal{T}\check{K}_{\neq}^2=&\int_{1}^{T}\int M {\left\langle D \right\rangle}^{N} \check{K}_{\neq}^2 M {\left\langle D \right\rangle}^{N}  |\partial_{X}| |\nabla_{L}|\left( U \cdot \nabla_{L} U^{2} \right)_{\neq}dVdt  \nonumber \\
=& \int_{1}^{T}\int M {\left\langle D \right\rangle}^{N} \check{K}_{\neq}^2 M {\left\langle D \right\rangle}^{N} |\partial_{X}| |\nabla_{L}| \Big[ \left(U_0 \cdot \nabla_{L} U_{\neq}^{2} \right)_{\neq} \nonumber \\
\quad& + \left( U_{\neq} \cdot \nabla U_0^{2} \right)_{\neq}+ \left( U_{\neq} \cdot \nabla_{L} U_{\neq}^{2} \right)_{\neq}\Big] dVdt  \nonumber \\
\overset{\Delta}{=}& \mathcal{T}\check{K}_{\neq}^2(U_0, U_{\neq}^2)+\mathcal{T}\check{K}_{\neq}^2(U_{\neq}, U_0^2)+\mathcal{T}\check{K}_{\neq}^2(U_{\neq}, U_{\neq}^2).
\end{align}
Considering  $\mathcal{T}\check{K}_{\neq}^2(U_0, U_{\neq}^2)$, we decompose it into 
\begin{align}\label{5.29}
\mathcal{T}\check{K}_{\neq}^2(U_0, U_{\neq}^2)= &\int_{1}^{T}\int M {\left\langle D \right\rangle}^{N} \check{K}_{\neq}^2 M {\left\langle D \right\rangle}^{N} |\partial_{X}| |\nabla_{L}|  \left(U_0^1 \partial_{X} U_{\neq}^{2} +U_0^2 \partial_{Y}^{L} U_{\neq}^{2}+U_0^3 \partial_{Z} U_{\neq}^{2}\right)_{\neq} dVdt \nonumber \\
\overset{\Delta}{=}& \mathcal{T}\check{K}_{\neq}^2(U_0^1, U_{\neq}^2)+\mathcal{T}\check{K}_{\neq}^2(U_0^2, U_{\neq}^2)+\mathcal{T}\check{K}_{\neq}^2(U_0^3, U_{\neq}^2),
\end{align}
where the term $\mathcal{T}\check{K}_{\neq}^2(U_0^1, U_{\neq}^2)$ is estimated as 
\begin{align}\label{5.30}
\mathcal{T}\check{K}_{\neq}^2(U_0^1, U_{\neq}^2)=& -\int_{1}^{T}\int M {\left\langle D \right\rangle}^{N} \check{K}_{\neq}^2 M {\left\langle D \right\rangle}^{N} |\partial_{X}| |\nabla_{L}|  \left(U_0^1 \partial_{X} |\partial_{X}|^{-1} |\nabla_{L}|^{-1}\check{K}_{\neq}^{2} \right)_{\neq} dVdt \nonumber \\
\leqslant & \left\|  M \nabla_{L} \check{K}_{\neq}^2 \right\|_{L^2 H^N} \left\|  U_{0}^1 \right\|_{L^{\infty} H^{N}} \left\| M \check{K}_{\neq}^2 \right\|_{L^2 H^N}  \nonumber \\
\lesssim & \varepsilon \nu^{-\frac{1}{2}}  \varepsilon \varepsilon \nu^{-\frac{1}{6}} \nonumber \\
\lesssim & \varepsilon^2\left(  \varepsilon \nu^{-\frac{2}{3}}\right) \nonumber \\
\lesssim & \varepsilon^2,
\end{align}
by using the fact $|\partial_{X}|U_0^1=0$, \eqref{4..5}, \eqref{4..16}, \eqref{5.1} and Corollary \ref{cor4.1}.
Similar to \eqref{5.30}, one has
\begin{align}\label{5.31}
\mathcal{T}\check{K}_{\neq}^2(U_0^2, U_{\neq}^2)=& -\int_{1}^{T}\int M {\left\langle D \right\rangle}^{N} \check{K}_{\neq}^2 M {\left\langle D \right\rangle}^{N} |\partial_{X}| |\nabla_{L}|  \left(U_0^2 \partial_{Y}^{L} |\partial_{X}|^{-1} |\nabla_{L}|^{-1}\check{K}_{\neq}^{2} \right)_{\neq} dVdt \nonumber \\
\leqslant & \left\|  M \nabla_{L} \check{K}_{\neq}^2 \right\|_{L^2 H^N} \left\|  U_{0}^2 \right\|_{L^{\infty} H^{N}} \left\| M \check{K}_{\neq}^2 \right\|_{L^2 H^N}  \nonumber \\
\lesssim & \varepsilon \nu^{-\frac{1}{2}}  C_1 \varepsilon  \nu^{-1} \varepsilon \nu^{-\frac{1}{6}} \nonumber \\
\lesssim & \varepsilon^2\left(  C_1 \varepsilon  \nu^{-\frac{5}{3}}\right) \nonumber \\
\lesssim & \varepsilon^2,
\end{align}
which suffices for  $C_1 \varepsilon  \nu^{-\frac{5}{3}}$ sufficiently small. Notice that $\partial_{X} U_0^3=0$. Hence for $\mathcal{T}\check{K}_{\neq}^2(U_0^3, U_{\neq}^2)$, we use \eqref{4..5}, \eqref{4..18}, \eqref{5.1} and Corollary \ref{cor4.1} to get
\begin{align}\label{5.32}
\mathcal{T}\check{K}_{\neq}^2(U_0^3, U_{\neq}^2)=& -\int_{1}^{T}\int M {\left\langle D \right\rangle}^{N} \check{K}_{\neq}^2 M {\left\langle D \right\rangle}^{N} |\partial_{X}| |\nabla_{L}|  \left(U_0^3 \partial_{Z} |\partial_{X}|^{-1} |\nabla_{L}|^{-1}\check{K}_{\neq}^{2} \right)_{\neq} dVdt \nonumber \\
\leqslant & \left\|  M \nabla_{L} \check{K}_{\neq}^2 \right\|_{L^2 H^N} \left\|  U_{0}^3 \right\|_{L^{\infty} H^{N}} \left\| M \check{K}_{\neq}^2 \right\|_{L^2 H^N}  \nonumber \\
\lesssim & \varepsilon \nu^{-\frac{1}{2}}  C_0 \varepsilon  \nu^{-1} \varepsilon \nu^{-\frac{1}{6}} \nonumber \\
\lesssim & \varepsilon^2\left(  C_0 \varepsilon  \nu^{-\frac{5}{3}}\right) \nonumber \\
\lesssim & \varepsilon^2,
\end{align}
where  $C_0 \varepsilon  \nu^{-\frac{5}{3}}$ is sufficiently small. Therefore, we substitute \eqref{5.30}--\eqref{5.32} into \eqref{5.29} to yield
\begin{align}\label{5.33}
\mathcal{T}\check{K}_{\neq}^2(U_0, U_{\neq}^2) \lesssim \varepsilon^2.
\end{align}

For $\mathcal{T}\check{K}_{\neq}^2(U_{\neq}, U_0^2)$,
we first divide $\mathcal{T}\check{K}_{\neq}^2(U_{\neq}, U_0^2)$ into 
\begin{align}\label{5..34}
\mathcal{T}\check{K}_{\neq}^2(U_{\neq}, U_0^2)=& \int_{1}^{T}\int M {\left\langle D \right\rangle}^{N} \check{K}_{\neq}^2 M {\left\langle D \right\rangle}^{N} |\partial_{X}| |\nabla_{L}| \left( U_{\neq}^2  \partial_{Y} U_0^{2}+ U_{\neq}^3  \partial_{Z} U_0^{2}\right)_{\neq}  dVdt \nonumber \\
\overset{\Delta}{=}& \mathcal{T}\check{K}_{\neq}^2(U_{\neq}^2, U_0^2)+\mathcal{T}\check{K}_{\neq}^2(U_{\neq}^3, U_0^2).
\end{align}
By the fact $\partial_{X} U_0^2=0$, Corollary \ref{cor4.1}, \eqref{4..5} and \eqref{4..17}, we can obtain
\begin{align}\label{5.35}
\mathcal{T}\check{K}_{\neq}^2(U_{\neq}^2, U_0^2)= & -\int_{1}^{T}\int M {\left\langle D \right\rangle}^{N} \check{K}_{\neq}^2 M {\left\langle D \right\rangle}^{N} |\partial_{X}| |\nabla_{L}| \left( |\partial_{X}|^{-1} |\nabla_{L}|^{-1}\check{K}_{\neq}^2  \partial_{Y} U_0^{2}\right)_{\neq}  dVdt   \nonumber \\
\leqslant &  \left\|  M \nabla_{L} \check{K}_{\neq}^2 \right\|_{L^2 H^N} \left\|  \nabla U_{0}^2 \right\|_{L^{2} H^{N}} \left\| M \check{K}_{\neq}^2 \right\|_{L^{\infty} H^N}    \nonumber \\
\lesssim &  \varepsilon \nu^{-\frac{1}{2}} C_1  \varepsilon  \nu^{-1}  \nu^{-\frac{1}{2}}\varepsilon    \nonumber \\
\lesssim & \varepsilon^2 \left( C_1 \varepsilon  \nu^{-2} \right)  \nonumber \\
\lesssim & \varepsilon^2.
\end{align}
For $\mathcal{T}\check{K}_{\neq}^2(U_{\neq}^3, U_0^2)$, notice that $\partial_{X} U_0^2=0$ and the relation $\big| \widehat{\partial_{X} U_{\neq}^3 }\big| \lesssim \big| \widehat{m M Q_{\neq}^3} \big|$. Using \eqref{4..5}, \eqref{4..8} and \eqref{4..24} to obtain
\begin{align}\label{5.36}
\mathcal{T}\check{K}_{\neq}^2(U_{\neq}^3, U_0^2) \leqslant & \left\|  M \nabla_{L} \check{K}_{\neq}^2 \right\|_{L^2 H^N} \left\|  \partial_{X} U_{\neq}^3 \right\|_{L^{2} H^{N}} \left\| U_{0}^2 \right\|_{L^{\infty} H^{N+1}}    \nonumber \\
\lesssim &  \varepsilon \nu^{-\frac{1}{2}}  C_0  \varepsilon  \nu^{-\frac{1}{3}}  \nu^{-\frac{1}{6}} C_1 \varepsilon  \nu^{-1}  \nonumber \\
\lesssim & \varepsilon^2 \left( C_0 C_1 \varepsilon  \nu^{-2} \right)  \nonumber \\
\lesssim & \varepsilon^2,
\end{align}
which suffices for $C_0 C_1 \varepsilon  \nu^{-2}$ sufficiently small. Combining with \eqref{5.35}--\eqref{5.36}, we have
\begin{align}\label{5.37}
\mathcal{T}\check{K}_{\neq}^2(U_{\neq}, U_0^2) \lesssim \varepsilon^2.
\end{align}

Similar to  $\mathcal{T}\check{K}_{\neq}^1(U_{\neq}, U_{\neq}^1)$, the nonlinear term $\mathcal{T}\check{K}_{\neq}^2(U_{\neq}, U_{\neq}^2)$ is rewritten as
\begin{align}\label{5.38}
\mathcal{T}\check{K}_{\neq}^2(U_{\neq}, U_{\neq}^2)= & \int_{1}^{T}\int M {\left\langle D \right\rangle}^{N} \check{K}_{\neq}^2 M {\left\langle D \right\rangle}^{N} |\partial_{X}| |\nabla_{L}|  \left(U_{\neq}^1 \partial_{X} U_{\neq}^{2}+U_{\neq}^2 \partial_{Y}^{L} U_{\neq}^{2}+ U_{\neq}^3 \partial_{Z} U_{\neq}^{2} \right)_{\neq} dVdt  \nonumber \\
\overset{\Delta}{=}& \mathcal{T}\check{K}_{\neq}^2(U_{\neq}^1, U_{\neq}^2)+\mathcal{T}\check{K}_{\neq}^2(U_{\neq}^2, U_{\neq}^2)+\mathcal{T}\check{K}_{\neq}^2(U_{\neq}^3, U_{\neq}^2).
\end{align}
For $\mathcal{T}\check{K}_{\neq}^2(U_{\neq}^1, U_{\neq}^2)$, we use \eqref{4..2}, \eqref{4..5}, \eqref{5.1} and Corollary \ref{cor4.1} to obtain 
\begin{align}\label{5.39}
\mathcal{T}\check{K}_{\neq}^2(U_{\neq}^1, U_{\neq}^2)= & \int_{1}^{T}\int M {\left\langle D \right\rangle}^{N} \check{K}_{\neq}^2 M {\left\langle D \right\rangle}^{N} |\partial_{X}| |\nabla_{L}|  \Big( |\nabla_{X, Z}|^{-1} |\nabla_{L}|^{-1} \check{K}_{\neq}^1 \nonumber \\
& \quad \partial_{X} |\partial_{X}|^{-1} |\nabla_{L}|^{-1} \check{K}_{\neq}^{2} \Big)_{\neq} dVdt  \nonumber \\
\leqslant&  \left\|  M \nabla_{L} \check{K}_{\neq}^2 \right\|_{L^2 H^N} \left\|  M \check{K}_{\neq}^1 \right\|_{L^{\infty} H^{N}} \left\| M \check{K}_{\neq}^2 \right\|_{L^2 H^N}  \nonumber \\
\lesssim & \varepsilon \nu^{-\frac{1}{2}}  \varepsilon \varepsilon \nu^{-\frac{1}{6}} \nonumber \\
\lesssim & \varepsilon^2\left(  \varepsilon \nu^{-\frac{2}{3}}\right).
\end{align}
Then, by \eqref{4..2}, \eqref{4..5}, \eqref{5.1} and Corollary \ref{cor4.1}, we can estimate $\mathcal{T}\check{K}_{\neq}^2(U_{\neq}^2, U_{\neq}^2)$ as follows
\begin{align}\label{5..40}
\mathcal{T}\check{K}_{\neq}^2(U_{\neq}^2, U_{\neq}^2) = & \int_{1}^{T}\int M {\left\langle D \right\rangle}^{N} \check{K}_{\neq}^2 M {\left\langle D \right\rangle}^{N} |\partial_{X}| |\nabla_{L}|  \Big( |\partial_{X}|^{-1} |\nabla_{L}|^{-1} \check{K}_{\neq}^2 \nonumber \\
& \quad \partial_{Y}^{L} |\partial_{X}|^{-1} |\nabla_{L}|^{-1} \check{K}_{\neq}^{2} \Big)_{\neq} dVdt  \nonumber \\
\leqslant&  \left\|  M \nabla_{L} \check{K}_{\neq}^2 \right\|_{L^2 H^N} \left\|  M \check{K}_{\neq}^2 \right\|_{L^{\infty} H^{N}} \left\| M \check{K}_{\neq}^2 \right\|_{L^2 H^N}  \nonumber \\
\lesssim & \varepsilon \nu^{-\frac{1}{2}}  \varepsilon \varepsilon \nu^{-\frac{1}{6}} \nonumber \\
\lesssim & \varepsilon^2\left(  \varepsilon \nu^{-\frac{2}{3}}\right) \nonumber \\
\lesssim &\varepsilon^2.
\end{align}
Similarly, by  \eqref{4..5}, \eqref{4..24}, \eqref{5.1} and Corollary \ref{cor4.1}, it holds
\begin{align}\label{5..41}
\mathcal{T}\check{K}_{\neq}^2(U_{\neq}^3, U_{\neq}^2) =& - \int_{1}^{T}\int M {\left\langle D \right\rangle}^{N} \check{K}_{\neq}^2 M {\left\langle D \right\rangle}^{N} |\partial_{X}| |\nabla_{L}|  \Big( U_{\neq}^3 \partial_{Z} |\partial_{X}|^{-1} |\nabla_{L}|^{-1} \check{K}_{\neq}^2 \Big)_{\neq} dVdt \nonumber \\
\leqslant & \left\|  \nabla_{L} M \check{K}_{\neq}^2  \right\|_{L^2 H^N} \left\| \nabla_{L} U_{\neq}^{3}  \right\|_{L^{2} H^N} \left\| M \check{K}_{\neq}^2  \right\|_{L^{\infty} H^N} \nonumber \\
\quad& +\left\|  \nabla_{L} M \check{K}_{\neq}^2  \right\|_{L^2 H^N} \left\|   U_{\neq}^{3}  \right\|_{L^{2} H^N} \left\| M \check{K}_{\neq}^2  \right\|_{L^{\infty} H^N}  \nonumber \\
\lesssim & \varepsilon \nu^{-\frac{1}{2}} \left( C_0 \varepsilon  \nu^{-\frac{1}{3}} \nu^{-\frac{1}{6}} \varepsilon + C_0 \varepsilon  \nu^{-\frac{1}{3}} \nu^{-\frac{1}{2}} \varepsilon \right) \nonumber \\
\lesssim & \varepsilon^2 \left( C_0 \varepsilon  \nu^{-1}+ C_0 \varepsilon  \nu^{-\frac{4}{3}}\right) \nonumber \\
\lesssim &\varepsilon^2.
\end{align}

For the nonlinear  pressure term $\mathcal{NLP}\check{K}_{\neq}^2$. We divide $\mathcal{NLP}\check{K}_{\neq}^2$ into
\begin{align}\label{5.34}
\mathcal{NLP}\check{K}_{\neq}^2=& - \int_{1}^{T}\int M {\left\langle D \right\rangle}^{N} \check{K}_{\neq}^2 M {\left\langle D \right\rangle}^{N}|\partial_{X}| |\nabla_{L}|\partial_{Y}^{L} |\nabla_{L}|^{-2}\left(\partial_{i}^{L} U^{j} \partial_{j}^{L} U^{i}\right)_{\neq} dVdt \nonumber \\
=&- \int_{1}^{T}\int M {\left\langle D \right\rangle}^{N} \check{K}_{\neq}^2 M {\left\langle D \right\rangle}^{N}|\partial_{X}| |\nabla_{L}|\partial_{Y}^{L} |\nabla_{L}|^{-2}\Big[\partial_{i} U_0^{j} \partial_{j}^{L} U_{\neq}^{i} \nonumber \\
&+ \partial_{i}^{L} U_{\neq}^{j} \partial_{j} U_0^{i}+\partial_{i}^{L} U_{\neq}^{j} \partial_{j}^{L} U_{\neq}^{i}\Big]_{\neq} dVdt \nonumber \\
\overset{\Delta}{=}& \mathcal{NLP}\check{K}_{\neq}^2 (U_0, U_{\neq})+\mathcal{NLP}\check{K}_{\neq}^2 (U_{\neq}, U_{0})+\mathcal{NLP}\check{K}_{\neq}^2 (U_{\neq}, U_{\neq}).
\end{align}
We consider $\mathcal{NLP}\check{K}_{\neq}^2 (U_0, U_{\neq})$. By  \eqref{4..5}, \eqref{4..8}, \eqref{4..16}--\eqref{4..18},  \eqref{5.1}, the fact $\big| \widehat{\partial_{X} U_{\neq}^3 }\big| \lesssim \big| \widehat{m M Q_{\neq}^3} \big|$ and Corollary \ref{cor4.1}, it yields
\begin{align}\label{5.43}
\mathcal{NLP}\check{K}_{\neq}^2 (U_0, U_{\neq}) =& -\int_{1}^{T}\int M {\left\langle D \right\rangle}^{N} \check{K}_{\neq}^2 M {\left\langle D \right\rangle}^{N}|\partial_{X}| |\nabla_{L}|\partial_{Y}^{L} |\nabla_{L}|^{-2} \partial_{j}^{L}\left(\partial_{i} U_0^{j}  U_{\neq}^{i} \right)_{\neq}dVdt \cdot 1_{i \neq 1} \nonumber \\
= & \int_{1}^{T}\int M {\left\langle D \right\rangle}^{N} |\nabla_{L}| \check{K}_{\neq}^2 M {\left\langle D \right\rangle}^{N} \partial_{Y}^{L} |\nabla_{L}|^{-2} \partial_{j}^{L}\Big(\partial_{Y} U_0^{1} |\partial_{X}| U_{\neq}^{2} \nonumber \\
& \quad + \partial_{Y} U_0^{2} |\partial_{X}| U_{\neq}^{2}+\partial_{Y} U_0^{3} |\partial_{X}| U_{\neq}^{2}+\partial_{Z} U_0^{1} |\partial_{X}| U_{\neq}^{3} \nonumber \\
& \quad +\partial_{Z} U_0^{2} |\partial_{X}| U_{\neq}^{3}+\partial_{Z} U_0^{3} |\partial_{X}| U_{\neq}^{3}\Big)_{\neq}dVdt   \nonumber \\
\leqslant & \left\|  \nabla_{L} M \check{K}_{\neq}^2  \right\|_{L^2 H^N}\left(\left\| \nabla U_{0}^1  \right\|_{L^{\infty} H^N}+ \left\| \nabla U_{0}^2  \right\|_{L^{\infty} H^N}+ \left\| \nabla U_{0}^3  \right\|_{L^{\infty} H^N}\right)   \nonumber \\
\quad&  \times \left( \left\| M \check{K}_{\neq}^2  \right\|_{L^{2} H^N}  + \left\| m M  Q_{\neq}^{3}  \right\|_{L^{2} H^N} \right) \nonumber \\
\lesssim & \varepsilon \nu^{-\frac{1}{2}} \left( \varepsilon + C_1 \varepsilon \nu^{-1} + C_0 \varepsilon \nu^{-1} \right) \left(\varepsilon \nu^{-\frac{1}{6}}+ C_0 \varepsilon \nu^{-\frac{1}{3}} \nu^{-\frac{1}{6}} \right) \nonumber \\
\lesssim & \varepsilon^2 \left(\varepsilon \nu^{-\frac{2}{3}}+C_0 \varepsilon \nu^{-\frac{5}{3}}+C_0 \varepsilon \nu^{-1}+  C_0^2 \varepsilon   \nu^{-2}\right)  \nonumber \\
\lesssim &\varepsilon^2,
\end{align}
which suffices for $C_0^2 \varepsilon  \nu^{-2}$ sufficiently small. 
Similarly, we can obtain 
\begin{align}\label{5..44}
\mathcal{NLP}\check{K}_{\neq}^2 (U_{\neq}, U_{0}) \lesssim \varepsilon^2.
\end{align}
For $\mathcal{NLP}\check{K}_{\neq}^2 (U_{\neq}, U_{\neq})$, we use the fact $\big| \widehat{\partial_{X} U_{\neq}^3 }\big| \lesssim \big| \widehat{m M Q_{\neq}^3} \big|$, \eqref{HSX110}, \eqref{4..2}--\eqref{4..8} and \eqref{4..22}--\eqref{4..24} to get
\begin{align}\label{5.45}
\mathcal{NLP}\check{K}_{\neq}^2 (U_{\neq}, U_{\neq})=& \int_{1}^{T}\int M {\left\langle D \right\rangle}^{N} |\nabla_{L}| \check{K}_{\neq}^2 M {\left\langle D \right\rangle}^{N}|\partial_{X}| \partial_{Y}^{L} |\nabla_{L}|^{-2} \partial_{i}^{L} \partial_{j}^{L}\left( U_{\neq}^{j}  U_{\neq}^{i}\right)_{\neq} dVdt  \nonumber \\
\leqslant & \, 2 \left\|  M \nabla_{L} \check{K}_{\neq}^2 \right\|_{L^{2} H^N} \big( \left\|  |\partial_{X}| \partial_{Y}^{L} U_{\neq}^{j} \right\|_{L^2 H^N}  \left\|   U_{\neq}^i \right\|_{L^{\infty} H^N} \nonumber\\
\quad &+  \left\| |\partial_{X}|  U_{\neq}^{j} \right\|_{L^{\infty} H^N}  \left\| \partial_{Y}^{L}  U_{\neq}^i \right\|_{L^2 H^N} \big) \nonumber \\
\leqslant & \, 2 \left\|  M \nabla_{L} \check{K}_{\neq}^1 \right\|_{L^{2} H^N} \Big[ \Big( \left\| -|\partial_{X}| \partial_{Y}^{L} |\nabla_{X, Z}|^{-1} |\nabla_{L}|^{-1} \check{K}_{\neq}^{1} \right\|_{L^{2} H^N}  \nonumber \\
\quad &+\left\| - |\partial_{X}| \partial_{Y}^{L} |\partial_{X}|^{-1} |\nabla_{L}|^{-1} \check{K}_{\neq}^{2} \right\|_{L^{2} H^N}+ \left\| m M \nabla_{L} Q_{\neq}^{3} \right\|_{L^{2} H^N} \Big) \nonumber \\
\quad  & \times\left(\left\|  U_{\neq}^{1, 2} \right\|_{L^{\infty} H^N}+ \left\|    U_{\neq}^{3} \right\|_{L^{\infty} H^N} \right) \nonumber \\
\quad& + \Big( \left\| -\partial_{Y}^{L} |\nabla_{X, Z}|^{-1} |\nabla_{L}|^{-1} \check{K}_{\neq}^{1} \right\|_{L^{2} H^N}  \nonumber \\
\quad &+\left\| - \partial_{Y}^{L} |\partial_{X}|^{-1} |\nabla_{L}|^{-1} \check{K}_{\neq}^{2} \right\|_{L^{2} H^N}+ \left\| \nabla_{L} U_{\neq}^{3} \right\|_{L^{2} H^N} \Big) \nonumber \\
\quad  & \times\left(\left\| M\check{K}_{\neq}^{1, 2} \right\|_{L^{\infty} H^N}+ \left\|    m M Q_{\neq}^{3} \right\|_{L^{\infty} H^N} \right) \Big] \nonumber \\
\leqslant & \, 2 \left\|  M \nabla_{L} \check{K}_{\neq}^1 \right\|_{L^{2} H^N} \Big( \left\| M \check{K}_{\neq}^{1, 2} \right\|_{L^2 H^N} + \left\| m M \nabla_{L} Q_{\neq}^{3} \right\|_{L^2 H^N} + \left\| \nabla_{L} U_{\neq}^{3} \right\|_{L^2 H^N} \Big) \nonumber \\
\quad&  \times \left(\left\|    U_{\neq}^{1, 2} \right\|_{L^{\infty} H^N}+ \left\|  U_{\neq}^{3} \right\|_{L^{\infty} H^N} + \left\| M\check{K}_{\neq}^{1, 2} \right\|_{L^{\infty} H^N} + \left\| m M Q_{\neq}^3 \right\|_{L^{\infty} H^N}\right) \nonumber \\
\lesssim & \varepsilon \nu^{-\frac{1}{2}} \left( \varepsilon \nu^{-\frac{1}{6}} +C_0 \varepsilon  \nu^{-\frac{1}{3}} \nu^{-\frac{1}{2}}\right) \left( \varepsilon +C_0 \varepsilon  \nu^{-\frac{1}{3}} \right)  \nonumber \\
\lesssim & \varepsilon^2 \left(\varepsilon \nu^{-\frac{2}{3}} +C_0 \varepsilon  \nu^{-\frac{4}{3}}+ +C_0^2 \varepsilon  \nu^{-\frac{5}{3}}\right)  \nonumber \\
\lesssim &\varepsilon^2,
\end{align}
where   $C_0^2 \varepsilon  \nu^{-\frac{5}{3}}$ is sufficiently small. Submitting the estimates  \eqref{5.6}, \eqref{5.10}, \eqref{5.13}, \eqref{5.25}--\eqref{5.27}, \eqref{5.33}, \eqref{5.37}--\eqref{5.28} and \eqref{5.34} into \eqref{5.4}, \eqref{4..2} and \eqref{4..5} hold with $8$ replaced by $4$ on the right-hand side.

\subsection{Energy estimates on $Q_0^1$}\label{4.3.2}
We recall  $Q_0^1$ satisfy the  equation
\begin{align*}
\partial_{t} Q_0^{1}-\nu \Delta Q_0^{1}=-\left(U \cdot \nabla_{L} Q^{1} \right)_0- \left(Q^{j} \partial_{j}^{L} U^{1}\right)_0-2 \left(\partial_{i}^{L} U^{j} \partial_{i j}^{L} U^{1} \right)_0.
\end{align*}
The $H^N$ energy estimate of $Q_0^1$ gives
\begin{align}\label{5.40}
\frac{1}{2} \left\|  Q_0^1 (T)  \right\|_{H^N}^2+\nu \left\|   \nabla Q_0^1 \right\|_{L^2 H^N}^2 &= \frac{1}{2} \left\|  Q_0^1 (1)  \right\|_{H^N}^2- \int_{1}^{T} \int  {\left\langle D \right\rangle}^{N}  Q_0^1  {\left\langle D \right\rangle}^{N} \Big[  \left(U \cdot \nabla_{L} Q^{1} \right)_0 \nonumber \\
&\quad + \left(Q^{j} \partial_{j}^{L} U^{1}\right)_0+2 \left(\partial_{i}^{L} U^{j} \partial_{i j}^{L} U^{1} \right)_0   \Big]dVdt  \nonumber \\
&\overset{\Delta}{=}\frac{1}{2} \left\|  Q_0^1 (1) \right\|_{H^N}^2+ \mathcal{T}+\mathcal{NLS}1+\mathcal{NLS}2.
\end{align}
Then  we need to estimate the transport term $\mathcal{T}$, the nonlinear stretching $\mathcal{NLS}1$, $\mathcal{NLS}2$, respectively. Note that
\begin{align}\label{5.41}
\mathcal{T}= - \int_{1}^{T} \int  {\left\langle D \right\rangle}^{N}  Q_0^1  {\left\langle D \right\rangle}^{N} \Big[ U_0 \cdot \nabla Q_0^{1} +\left(U_{\neq} \cdot \nabla_{L} Q_{\neq}^{1} \right)_0 \Big]dVdt \overset{\Delta}{=} \mathcal{T}_1+\mathcal{T}_2.
\end{align}
For $\mathcal{T}_1$, by \eqref{4..1}, \eqref{4..16}--\eqref{4..18}, it yields
\begin{align}
\mathcal{T}_1 = &- \int_{1}^{T} \int  {\left\langle D \right\rangle}^{N}  Q_0^1  {\left\langle D \right\rangle}^{N} \left( U_0^2 \partial_{Y} Q_0^{1} +U_0^3 \partial_{Z} Q_0^{1} \right)dVdt  \nonumber \\
\leqslant & \left\| \nabla U_0^1 \right\|_{L^2 H^{N+1}} \left\| U_0^{2, 3} \right\|_{L^{\infty} H^N} \left\| \nabla Q_0^1 \right\|_{L^2 H^N}  \nonumber \\
\lesssim & \varepsilon \nu^{-\frac{1}{2}}  \left( C_1 \varepsilon  \nu^{-1} + C_0 \varepsilon  \nu^{-1} \right) \varepsilon \nu^{-\frac{1}{2}}  \nonumber \\
\lesssim & \varepsilon^2 \left( C_1 \varepsilon  \nu^{-2} + C_0 \varepsilon  \nu^{-2}      \right) \nonumber\\
\lesssim & \varepsilon^2,
\end{align}
where $\varepsilon  \nu^{-2}$ is sufficiently small. For $\mathcal{T}_2$, we use the fact $\nabla_{L} \cdot U_{\neq}=0$, $Q_{\neq}^1= -|\nabla_{X, Z}|^{-1} |\nabla_{L}| \check{K}_{\neq}^1$, \eqref{4..2}, \eqref{4..1} and \eqref{4..22}--\eqref{4..24}    to get
\begin{align}
\mathcal{T}_2 =& \int_{1}^{T} \int  {\left\langle D \right\rangle}^{N} \nabla Q_0^1  {\left\langle D \right\rangle}^{N} \left(U_{\neq} Q_{\neq}^{1} \right)_0 dVdt  \nonumber \\
\leqslant & \left\| \nabla Q_0^1 \right\|_{L^{2} H^N} \left\| U_{\neq} \right\|_{L^{\infty} H^N} \left\|  - |\nabla_{X, Z}|^{-1} |\nabla_{L}| \check{K}_{\neq}^1 \right\|_{L^2 H^N}  \nonumber \\
\leqslant & \left\| \nabla Q_0^1 \right\|_{L^{2} H^N} \left\| U_{\neq} \right\|_{L^{\infty} H^N} \left\|  M \nabla_{L} \check{K}_{\neq}^1 \right\|_{L^2 H^N}   \nonumber \\
\lesssim & \varepsilon \nu^{-\frac{1}{2}}  \left( \varepsilon  + C_0 \varepsilon  \nu^{-\frac{1}{3}}  \right) \varepsilon \nu^{-\frac{1}{2}}  \nonumber \\
\lesssim & \varepsilon^2 \left( \varepsilon \nu^{-1}  + C_0 \varepsilon  \nu^{-\frac{4}{3}}      \right) \nonumber\\
\lesssim & \varepsilon^2.
\end{align}

To estimate $\mathcal{NLS}1$, we divide it into
\begin{align}\label{5.44}
\mathcal{NLS}1&=- \int_{1}^{T} \int  {\left\langle D \right\rangle}^{N}  Q_0^1  {\left\langle D \right\rangle}^{N} \Big[  Q_0^{j} \partial_{j} U_0^{1}+ \left(Q_{\neq}^{j} \partial_{j}^{L} U_{\neq}^{1}\right)_0  \Big]dVdt \nonumber \\
&\overset{\Delta}{=} \mathcal{NLS}1 (0, 0)+\mathcal{NLS}1(\neq, \neq).
\end{align}
For $\mathcal{NLS}1 (0, 0)$, using \eqref{4..1} and \eqref{4..16}--\eqref{4..18}   yields
\begin{align*}
\mathcal{NLS}1 (0, 0)=& - \int_{1}^{T} \int  {\left\langle D \right\rangle}^{N}  Q_0^1  {\left\langle D \right\rangle}^{N} \left( Q_0^2 \partial_{Y} U_0^{1} +Q_0^3 \partial_{Z} U_0^{1} \right)dVdt \nonumber\\
\leqslant & \left\| Q_0^1  \right\|_{L^{\infty} H^N} \left\| \nabla  U_0^{2, 3}  \right\|_{L^2 H^{N+1}} \left\|   \nabla U_0^1 \right\|_{L^2 H^N}  \nonumber \\
\lesssim &  \varepsilon \left(C_1 \varepsilon  \nu^{-1} \nu^{-\frac{1}{2}}+ C_0 \varepsilon  \nu^{-1} \nu^{-\frac{1}{2}} \right) \varepsilon \nu^{-\frac{1}{2}} \nonumber \\
\lesssim & \varepsilon^2 \left( C_1 \varepsilon  \nu^{-2} + C_0 \varepsilon  \nu^{-2}      \right) \nonumber\\
\lesssim & \varepsilon^2.
\end{align*} 
For $\mathcal{NLS}1(\neq, \neq)$, by $Q_{\neq}^1= -|\nabla_{X, Z}|^{-1} |\nabla_{L}| \check{K}_{\neq}^1$, $Q_{\neq}^2= -|\partial_{X}|^{-1} |\nabla_{L}| \check{K}_{\neq}^2$, the lower bound \eqref{4.11} of $m$, \eqref{4..2}--\eqref{4..5}, \eqref{4..1} and \eqref{4..22}--\eqref{4..23}, we can deduce
\begin{align*}
\mathcal{NLS}1(\neq, \neq) \leqslant & \left\| Q_0^1 \right\|_{L^{\infty} H^N} \left\| Q_{\neq}^j \right\|_{L^2 H^N} \left\| \nabla_{L} U_{\neq}^1  \right\|_{L^2 H^N}   \nonumber \\
\leqslant & \left\| Q_0^1 \right\|_{L^{\infty} H^N} \left( \left\| M \nabla_{L} \check{K}_{\neq}^{1, 2} \right\|_{L^2 H^N} + \left\| m M Q_{\neq}^3 \right\|_{L^2 H^N} \sup m^{-1}  \right) \left\| \nabla_{L} U_{\neq}^1  \right\|_{L^2 H^N} \nonumber \\
\lesssim &  \varepsilon  \left( \varepsilon  \nu^{-\frac{1}{2}}+ C_0 \varepsilon  \nu^{-\frac{1}{3}}  \nu^{-\frac{1}{6}}   \nu^{-\frac{2}{3}}\right) \varepsilon \nu^{-\frac{1}{2}}  \nonumber \\
\lesssim & \varepsilon^2 \left( \varepsilon \nu^{-1} + C_0 \varepsilon  \nu^{-\frac{5}{3}}      \right) \nonumber\\
\lesssim & \varepsilon^2,
\end{align*}
which suffices for  $C_0 \varepsilon  \nu^{-\frac{5}{3}}$ sufficiently small.

To estimate $\mathcal{NLS}2$, we decompose it into
\begin{align}\label{5.47}
\mathcal{NLS}2=& - 2 \int_{1}^{T} \int  {\left\langle D \right\rangle}^{N}  Q_0^1  {\left\langle D \right\rangle}^{N} \Big[   \partial_{i} U_0^{j} \partial_{i j} U_0^{1} + \left(\partial_{i}^{L} U_{\neq}^{j} \partial_{i j}^{L} U_{\neq}^{1} \right)_0   \Big]dVdt \nonumber\\
\overset{\Delta}{=} & \mathcal{NLS}2(0, 0)+\mathcal{NLS}2 (\neq, \neq).
\end{align}
On the one hand,  we have
\begin{align}
\mathcal{NLS}2(0, 0) \leqslant & \left\| Q_0^1    \right\|_{L^{\infty} H^N} \left\|   \nabla U_0^{2, 3} \right\|_{L^2 H^N} \left\| \nabla U_0^1 \right\|_{L^2 H^{N+1}} \nonumber \\
\lesssim & \varepsilon \left( C_1 \varepsilon  \nu^{-1} \nu^{-\frac{1}{2}} + C_0 \varepsilon  \nu^{-1} \nu^{-\frac{1}{2}} \right) \varepsilon \nu^{-\frac{1}{2}}  \nonumber \\
\lesssim &  \varepsilon^2 \left( C_0 \varepsilon  \nu^{-2}  \right) \nonumber \\
\lesssim &\varepsilon^2.
\end{align}
On the other hand, applying $U_{\neq}^1= -|\nabla_{X, Z}|^{-1} |\nabla_{L}|^{-1} \check{K}_{\neq}^1$, \eqref{4..2}, \eqref{4..1} and \eqref{4..22}--\eqref{4..24} yields
\begin{align*}
\mathcal{NLS}2 (\neq, \neq) \leqslant & \left\|  Q_0^1 \right\|_{L^{\infty} H^N} \left( \left\| \nabla_{L} U_{\neq}^{1, 2} \right\|_{L^2 H^N}+ \left\| \nabla_{L} U_{\neq}^3 \right\|_{L^2 H^N}\right) \left\|  -\partial_{i j}^{L} |\nabla_{X, Z}|^{-1} |\nabla_{L}|^{-1} \check{K}_{\neq}^1 \right\|_{L^2 H^N}  \nonumber \\
\leqslant & \left\|  Q_0^1 \right\|_{L^{\infty} H^N} \left( \left\| \nabla_{L} U_{\neq}^{1, 2} \right\|_{L^2 H^N}+ \left\| \nabla_{L} U_{\neq}^3 \right\|_{L^2 H^N}\right) \left\|  M \nabla_{L} \check{K}_{\neq}^1 \right\|_{L^2 H^N}  \nonumber \\
\lesssim &  \varepsilon \left( \varepsilon \nu^{-\frac{1}{2}}+ C_0 \varepsilon  \nu^{-\frac{1}{3}}  \nu^{-\frac{1}{2}} \right) \varepsilon \nu^{-\frac{1}{2}}  \nonumber \\
\lesssim &  \varepsilon^2 \left( \varepsilon \nu^{-1} + C_0 \varepsilon  \nu^{-\frac{4}{3}}  \right)  \nonumber \\
\lesssim & \varepsilon^2.
\end{align*}
Putting the estimates (\ref{5.41}), (\ref{5.44}) and (\ref{5.47}) into (\ref{5.40}), we finish the proof of \eqref{4..1}, with $8$ replaced by $4$.
\subsection{Energy estimates on  $Q_0^2$}\label{4.3.3}
Note that $Q_0^2$ satisfies
\begin{align*}
&\partial_{t} Q_0^{2}-\nu \Delta Q_0^{2}+ Q_0^1- \partial_{Y Y} U_0^{1}  \nonumber \\
&\quad=-\left(U \cdot \nabla_{L} Q^{2} \right)_0-\left(Q^{j} \partial_{j}^{L} U^{2}\right)_0-2 \left(\partial_{i}^{L} U^{j} \partial_{i j}^{L} U^{2}\right)_0+\partial_{Y} \left(\partial_{i}^{L} U^{j} \partial_{j}^{L} U^{i}\right)_0.
\end{align*}
Now we continue to estimate the $H^N$-norm for $Q_0^2$. Unlike $Q_0^1$, here we need additional estimates of the lift up term $\mathcal{LU}2$ and the nonlinear pressure term $\mathcal{NLP}$. An energy estimate gives
\begin{align}\label{5.51}
\frac{1}{2} \left\|  Q_0^2 (T)  \right\|_{H^N}^2+\nu \left\|   \nabla Q_0^2 \right\|_{L^2 H^N}^2 &= \frac{1}{2} \left\|  Q_0^2 (1)  \right\|_{H^N}^2- \int_{1}^{T} \int  {\left\langle D \right\rangle}^{N}  Q_0^2  {\left\langle D \right\rangle}^{N} \Big[ \left(  Q_0^1- \partial_{Y Y} U_0^{1} \right) \nonumber \\
&\quad + \left(U \cdot \nabla_{L} Q^2 \right)_0+ \left(Q^{j} \partial_{j}^{L} U^2\right)_0+2 \left(\partial_{i}^{L} U^{j} \partial_{i j}^{L} U^2 \right)_0    \nonumber \\
& \quad  -\partial_{Y} \left(\partial_{i}^{L} U^{j} \partial_{j}^{L} U^{i}\right)_0 \Big]dVdt  \nonumber \\
&\overset{\Delta}{=}\frac{1}{2} \left\|  Q_0^2 (1) \right\|_{H^N}^2+ \mathcal{LU}2+\mathcal{T}+\mathcal{NLS}1+\mathcal{NLS}2+\mathcal{NLP}.
\end{align}

For  the lift up term $\mathcal{LU}2$, by \eqref{4..16}--\eqref{4..17}, it holds
\begin{align}\label{5.52}
\mathcal{LU}2 =& - \int_{1}^{T} \int  {\left\langle D \right\rangle}^{N}  Q_0^2  {\left\langle D \right\rangle}^{N} \left(  Q_0^1- \partial_{Y Y} U_0^{1} \right) dVdt \nonumber \\
\leqslant &  \int_{1}^{T} \int (1+|k, \eta, l|^2)^{\frac{N}{2}} \widehat{Q_0^2} \, (1+|k, \eta, l|^2)^{\frac{N}{2}} \frac{k^2+l^2}{k^2+\eta^2+l^2} \widehat{Q_0^1} d \xi dt \nonumber \\
\leqslant &  \int_{1}^{T} \int {\left\langle D \right\rangle}^{N} Q_0^2 \cdot {\left\langle D \right\rangle}^{N} Q_0^1 dVdt \nonumber \\
\leqslant &  \left\|  \nabla U_0^2 \right\|_{L^2 H^{N+1}} \left\|  \nabla U_0^1 \right\|_{L^2 H^{N+1}} \nonumber \\
\lesssim &  C_1 \varepsilon  \nu^{-1} \nu^{-\frac{1}{2}} \varepsilon \nu^{-\frac{1}{2}}  \nonumber \\
\lesssim & \left( C_1 \varepsilon  \nu^{-1}\right)^2 \left(   \frac{1}{C_1}\right)   \nonumber \\
\lesssim & \left( C_1 \varepsilon  \nu^{-1}\right)^2, 
\end{align}
which suffices for $C_1$ sufficiently large. Then,  we divide the transport term $\mathcal{T}$ into
\begin{align}\label{5.53}
\mathcal{T}= - \int_{1}^{T} \int  {\left\langle D \right\rangle}^{N}  Q_0^2  {\left\langle D \right\rangle}^{N} \Big[ U_0 \cdot \nabla Q_0^{2} +\left(U_{\neq} \cdot \nabla_{L} Q_{\neq}^{2} \right)_0 \Big]dVdt \overset{\Delta}{=} \mathcal{T}_1+\mathcal{T}_2.
\end{align}
For $\mathcal{T}_1$, using  \eqref{4..4}, \eqref{4..17} and \eqref{4..18} yields
\begin{align}
\mathcal{T}_1 = &- \int_{1}^{T} \int  {\left\langle D \right\rangle}^{N}  Q_0^2  {\left\langle D \right\rangle}^{N} \left( U_0^2 \partial_{Y} Q_0^2 +U_0^3 \partial_{Z} Q_0^2 \right)dVdt  \nonumber \\
\leqslant & \left\| \nabla U_0^2 \right\|_{L^2 H^{N+1}} \left\| U_0^{2, 3} \right\|_{L^{\infty} H^N} \left\| \nabla Q_0^2 \right\|_{L^2 H^N}  \nonumber \\
\lesssim &  C_1 \varepsilon  \nu^{-1} \nu^{-\frac{1}{2}}  \left( C_1 \varepsilon  \nu^{-1} + C_0 \varepsilon  \nu^{-1} \right) C_1 \varepsilon  \nu^{-1} \nu^{-\frac{1}{2}}  \nonumber \\
\lesssim & \left( C_1 \varepsilon  \nu^{-1}\right)^2 \left( C_1 \varepsilon  \nu^{-2} + C_0 \varepsilon  \nu^{-2}      \right) \nonumber\\
\lesssim & \left( C_1 \varepsilon  \nu^{-1}\right)^2.
\end{align}
Notice that the fact $\nabla_{L} U_{\neq}=0$. For $\mathcal{T}_2$, by $Q_{\neq}^2= -|\partial_{X}|^{-1} |\nabla_{L}| \check{K}_{\neq}^2$,  \eqref{4..5}, \eqref{4..4} and \eqref{4..22}--\eqref{4..24}, we deduce
\begin{align*}
\mathcal{T}_2 =&  \int_{1}^{T} \int  {\left\langle D \right\rangle}^{N} \nabla Q_0^2  {\left\langle D \right\rangle}^{N} \left(U_{\neq} Q_{\neq}^{2} \right)_0 dVdt   \nonumber \\
\leqslant & \left\| \nabla Q_0^2 \right\|_{L^2 H^{N}} \left\| U_{\neq} \right\|_{L^{\infty} H^N} \left\| M \nabla_{L} \check{K}_{\neq}^2 \right\|_{L^2 H^N}  \nonumber \\
\lesssim & C_1 \varepsilon  \nu^{-1} \nu^{-\frac{1}{2}}   \left( \varepsilon  + C_0 \varepsilon  \nu^{-\frac{1}{3}}   \right) \varepsilon \nu^{-\frac{1}{2}}  \nonumber \\
\lesssim & \left( C_1 \varepsilon  \nu^{-1}\right)^2 \left(\frac{1}{C_1} \varepsilon   +\frac{C_0}{C_1} \varepsilon  \nu^{-\frac{1}{3}}      \right) \nonumber\\
\lesssim & \left( C_1 \varepsilon  \nu^{-1}\right)^2.
\end{align*}

We estimate the nonlinear stretching term $\mathcal{NLS}1$ as follows
\begin{align}\label{5.56}
\mathcal{NLS}1&=- \int_{1}^{T} \int  {\left\langle D \right\rangle}^{N}  Q_0^2  {\left\langle D \right\rangle}^{N} \Big[  Q_0^{j} \partial_{j} U_0^2+ \left(Q_{\neq}^{j} \partial_{j}^{L} U_{\neq}^2\right)_0  \Big]dVdt \nonumber \\
&\overset{\Delta}{=} \mathcal{NLS}1 (0, 0)+\mathcal{NLS}1(\neq, \neq).
\end{align}
For $\mathcal{NLS}1 (0, 0)$, notice that $\partial_{j} U_0^2=0$ if $j=1$.  Using \eqref{4..4} and \eqref{4..17}--\eqref{4..18} yields
\begin{align*}
\mathcal{NLS}1 (0, 0) \leqslant & \left\| Q_0^2  \right\|_{L^{\infty} H^N} \left\| Q_0^j  \right\|_{L^2 H^N} \left\|   \nabla U_0^2 \right\|_{L^2 H^N} \cdot 1_{j \neq 1} \nonumber \\
\leqslant & \left\| Q_0^2  \right\|_{L^{\infty} H^N} \left\| \nabla  U_0^{2, 3}  \right\|_{L^2 H^{N+1}} \left\|   \nabla U_0^2 \right\|_{L^2 H^N}  \nonumber \\
\lesssim & C_1 \varepsilon  \nu^{-1} \left(C_1 \varepsilon  \nu^{-1} \nu^{-\frac{1}{2}}+ C_0 \varepsilon  \nu^{-1} \nu^{-\frac{1}{2}} \right) C_1 \varepsilon  \nu^{-1} \nu^{-\frac{1}{2}} \nonumber \\
\lesssim & \left( C_1 \varepsilon  \nu^{-1}\right)^2 \left( C_1 \varepsilon  \nu^{-2} + C_0 \varepsilon  \nu^{-2}      \right) \nonumber\\
\lesssim & \left( C_1 \varepsilon  \nu^{-1}\right)^2.
\end{align*}
To estimate $\mathcal{NLS}1(\neq, \neq)$, we use $Q_{\neq}^1= -|\nabla_{X, Z}|^{-1} |\nabla_{L}| \check{K}_{\neq}^1$, $Q_{\neq}^2= -|\partial_{X}|^{-1} |\nabla_{L}| \check{K}_{\neq}^2$, \eqref{4.11}, \eqref{4..2}--\eqref{4..5}, \eqref{4..4}, Corollary \ref{cor4.1} and \eqref{4..23}--\eqref{4..21} to get
\begin{align*}
\mathcal{NLS}1(\neq, \neq) \leqslant & \left\| Q_0^2 \right\|_{L^{\infty} H^N} \left\| Q_{\neq}^j \right\|_{L^2 H^N} \left\| \nabla_{L} U_{\neq}^2  \right\|_{L^2 H^N}   \nonumber \\
\leqslant & \left\| Q_0^2 \right\|_{L^{\infty} H^N} \left( \left\| M \nabla_{L} \check{K}_{\neq}^{1, 2} \right\|_{L^2 H^N} + \left\| m M Q_{\neq}^3 \right\|_{L^2 H^N} \sup m^{-1}  \right) \left\| \nabla_{L} U_{\neq}^2  \right\|_{L^2 H^N} \nonumber \\
\lesssim & C_1 \varepsilon  \nu^{-1} \left( \varepsilon \nu^{-\frac{1}{2}} + C_0 \varepsilon  \nu^{-\frac{1}{3}}  \nu^{-\frac{1}{6}}   \nu^{-\frac{2}{3}} \right) \varepsilon \nu^{-\frac{1}{2}}  \nonumber \\
\lesssim & \left( C_1 \varepsilon  \nu^{-1}\right)^2 \left( \frac{1}{C_1} \varepsilon + \frac{C_0}{C_1} \varepsilon  \nu^{-\frac{2}{3}}      \right) \nonumber\\
\lesssim & \left( C_1 \varepsilon  \nu^{-1}\right)^2,
\end{align*}
which suffices for $\frac{C_0}{C_1}\varepsilon \nu^{-\frac{2}{3}}$  sufficiently small. For the nonlinear stretching term $\mathcal{NLS}2$, we similarly divide it into
\begin{align}\label{5.59}
\mathcal{NLS}2&= - 2 \int_{1}^{T} \int  {\left\langle D \right\rangle}^{N}  Q_0^2  {\left\langle D \right\rangle}^{N} \Big[   \partial_{i} U_0^{j} \partial_{i j} U_0^2 + \left(\partial_{i}^{L} U_{\neq}^{j} \partial_{i j}^{L} U_{\neq}^2 \right)_0   \Big]dVdt \nonumber\\
&=2 \int_{1}^{T} \int  {\left\langle D \right\rangle}^{N}  \partial_{j} Q_0^2  {\left\langle D \right\rangle}^{N} \left(  \partial_{i} U_0^{j} \partial_{i} U_0^2\right)dVdt - 2 \int_{1}^{T} \int  {\left\langle D \right\rangle}^{N}  Q_0^2  {\left\langle D \right\rangle}^{N}  \left(\partial_{i}^{L} U_{\neq}^{j} \partial_{i j}^{L} U_{\neq}^2 \right)_0  dVdt \nonumber\\
&\overset{\Delta}{=} \mathcal{NLS}2(0, 0)+\mathcal{NLS}2 (\neq, \neq).
\end{align}
We consider $\mathcal{NLS}2(0, 0)$. By \eqref{4..4} and \eqref{4..17}--\eqref{4..18}, it holds
\begin{align*}
\mathcal{NLS}2(0, 0) \leqslant & \left\| \nabla Q_0^2    \right\|_{L^2 H^N} \left\|   \nabla U_0^j  \right\|_{L^{\infty} H^N} \left\| \nabla U_0^2 \right\|_{L^2 H^N}  \cdot 1_{i, j \neq 1}\nonumber \\
\leqslant & \left\| \nabla Q_0^2    \right\|_{L^2 H^N} \left\|   U_0^{2, 3} \right\|_{L^{\infty} H^{N+1}} \left\| \nabla U_0^2 \right\|_{L^2 H^N} \nonumber \\
\lesssim & C_1 \varepsilon  \nu^{-1} \nu^{-\frac{1}{2}}  \left( C_1 \varepsilon  \nu^{-1}  + C_0 \varepsilon  \nu^{-1}  \right) C_1 \varepsilon  \nu^{-1} \nu^{-\frac{1}{2}}   \nonumber \\
\lesssim &  \left( C_1 \varepsilon  \nu^{-1}\right)^2 \left( C_0 \varepsilon  \nu^{-2}  \right) \nonumber \\
\lesssim & \left( C_1 \varepsilon \nu^{-1}\right)^2,
\end{align*}
which suffices for $\varepsilon  \nu^{-2}$  sufficiently small. For $\mathcal{NLS}2 (\neq, \neq)$, we use  Corollary \ref{cor4.1}, \eqref{4..4}, \eqref{4..2}--\eqref{4..5}, \eqref{4..24} and  \eqref{5.1}  to deduce
\begin{align*}
\mathcal{NLS}2 (\neq, \neq) \leqslant & \left\|  Q_0^2 \right\|_{L^{\infty} H^N} \left( \left\| \nabla_{L} U_{\neq}^{1, 2} \right\|_{L^2 H^N}+ \left\| \nabla_{L} U_{\neq}^3 \right\|_{L^2 H^N}\right) \left\|  -\partial_{i j}^{L} |\partial_{X}|^{-1} |\nabla_{L}|^{-1} \check{K}_{\neq}^2 \right\|_{L^2 H^N}  \nonumber \\
\leqslant & \left\|  Q_0^2 \right\|_{L^{\infty} H^N} \left( \left\| M \check{K}_{\neq}^{1, 2} \right\|_{L^2 H^N}+ \left\| \nabla_{L} U_{\neq}^3 \right\|_{L^2 H^N}\right) \left\|  M \nabla_{L} \check{K}_{\neq}^2 \right\|_{L^2 H^N}  \nonumber \\
\lesssim &  C_1 \varepsilon  \nu^{-1} \left( \varepsilon \nu^{-\frac{1}{6}}+ C_0 \varepsilon  \nu^{-\frac{1}{3}}  \nu^{-\frac{1}{2}} \right) \varepsilon \nu^{-\frac{1}{2}} \nonumber \\
\lesssim &  \left( C_1 \varepsilon  \nu^{-1}\right)^2 \left(\frac{1}{C_1} \varepsilon + \frac{C_0}{C_1} \varepsilon  \nu^{-\frac{1}{3}}  \right)  \nonumber \\
\lesssim & \left( C_1 \varepsilon  \nu^{-1}\right)^2.
\end{align*}

Finally we consider $\mathcal{NLP}$ which divide into 
\begin{align*}\label{5.62}
\mathcal{NLP} &= \int_{1}^{T} \int  {\left\langle D \right\rangle}^{N}  Q_0^2  {\left\langle D \right\rangle}^{N} \partial_{Y} \Big[ \left(\partial_{i} U_0^{j} \partial_{j} U_0^{i}\right)+ \left(\partial_{i}^{L} U_{\neq}^{j} \partial_{j}^{L} U_{\neq}^{i}\right)_0 \Big] dVdt  \nonumber \\
&\overset{\Delta}{=} \mathcal{NLP}(0, 0)+\mathcal{NLP} (\neq, \neq).
\end{align*}
For $\mathcal{NLP}(0, 0)$, note that $\partial_{i} U_0^{j} \partial_{j} U_0^{i}=\partial_{i j} ( U_0^{j} U_0^{i})=0$ if $i=1$ or $j=1$. Applying \eqref{4..4} and \eqref{4..17}--\eqref{4..18} yields
\begin{align}
\mathcal{NLP}(0, 0)=& -\int_{1}^{T} \int  {\left\langle D \right\rangle}^{N} \partial_{Y} Q_0^2  {\left\langle D \right\rangle}^{N}   \left(\partial_{i} U_0^{j} \partial_{j} U_0^{i}\right) dVdt \nonumber \\
\leqslant & \left\| \nabla Q_0^2 \right\|_{L^2 H^N} \left\| \nabla U_0^j \right\|_{L^{\infty} H^N} \left\| \nabla U_0^i \right\|_{L^2 H^N} \cdot
1_{i, j \neq 1} \nonumber \\
\leqslant & \left\| \nabla Q_0^2 \right\|_{L^2 H^N} \left\| U_0^{2, 3} \right\|_{L^{\infty} H^{N+1}} \left\| \nabla U_0^{2, 3} \right\|_{L^2 H^N}  \nonumber \\
\lesssim & C_1 \varepsilon  \nu^{-1} \nu^{-\frac{1}{2}} \left( C_1 \varepsilon  \nu^{-1}+C_0 \varepsilon  \nu^{-1}  \right) \left( C_1 \varepsilon  \nu^{-1} \nu^{-\frac{1}{2}}+C_0 \varepsilon  \nu^{-1} \nu^{-\frac{1}{2}}\right)  \nonumber\\
\lesssim &  \left( C_1 \varepsilon  \nu^{-1}\right)^2 \left( C_0 \varepsilon  \nu^{-2}+ \frac{C_0^2}{C_1} \varepsilon   \nu^{-2}  \right)  \nonumber \\
\lesssim & \left( C_1 \varepsilon  \nu^{-1}\right)^2.
\end{align}
 To estimate $\mathcal{NLP} (\neq, \neq)$, by using the fact $|\widehat{\nabla_{L}}| \leqslant | \widehat{m^{\frac{1}{2}} \Delta_{L}}|$, \eqref{4.11}, \eqref{4..4}, \eqref{4..2}--\eqref{4..5}, \eqref{4..22}--\eqref{4..24}, \eqref{5.1} and  integration by parts, we have
\begin{align*}
\mathcal{NLP} (\neq, \neq)& = -\int_{1}^{T} \int  {\left\langle D \right\rangle}^{N} \partial_{Y} Q_0^2  {\left\langle D \right\rangle}^{N}   \left(\partial_{i}^{L} U_{\neq}^{j} \partial_{j}^{L} U_{\neq}^{i}\right)_0  dVdt  \nonumber \\
&\leqslant   \left\| \nabla Q_0^2 \right\|_{L^2 H^N} \left\| \nabla_{L} U_{\neq}^j \right\|_{L^{\infty} H^N} \left\| \nabla_{L} U_{\neq}^i \right\|_{L^2 H^N}    \nonumber \\
&\leqslant   \left\| \nabla Q_0^2 \right\|_{L^2 H^N} \left( \left\| M \check{K}_{\neq}^{1, 2} \right\|_{L^{\infty} H^N}+ \left\| m  \Delta_{L} U_{\neq}^3 \right\|_{L^{\infty} H^N} \sup m^{-\frac{1}{2}}  \right)   \nonumber \\
& \quad  \times \left( \left\| \nabla_{L} U_{\neq}^{1, 2}  \right\|_{L^2 H^N}+\left\| \nabla_{L} U_{\neq}^3  \right\|_{L^2 H^N}\right) \nonumber \\
& \lesssim  C_1 \varepsilon  \nu^{-1} \nu^{-\frac{1}{2}} \left( \varepsilon+ C_0 \varepsilon  \nu^{-\frac{1}{3}}   \nu^{-\frac{1}{3}}\right) \left(\varepsilon \nu^{-\frac{1}{2}}+ C_0 \varepsilon  \nu^{-\frac{1}{3}} \nu^{-\frac{1}{2}}\right)     \nonumber \\
&\lesssim   \left( C_1 \varepsilon  \nu^{-1}\right)^2 \left(\frac{1}{C_1} \varepsilon + \frac{C_0}{C_1} \varepsilon  \nu^{-\frac{2}{3}} +\frac{C_0}{C_1} \varepsilon \nu^{-\frac{1}{3}} + \frac{C_0^2}{C_1} \varepsilon  \nu^{-1} \right)  \nonumber \\
&\lesssim  \left( C_1 \varepsilon  \nu^{-1}\right)^2.
\end{align*}
Thus, \eqref{4..4} holds with $8$ replaced by $4$ immediately  follows \eqref{5.51}, \eqref{5.52}, \eqref{5.53}, \eqref{5.56} and \eqref{5.59}.

\section{Energy estimates on $Q^3$}\label{sec6}
In this section, we want to prove that under the bootstrap hypotheses of Proposition \ref{pro4.1}, the estimates on $Q_{\neq}^3$ and $Q_0^3$  hold (i.e., \eqref{4..8}  and \eqref{4..7}), with 8 replaced by 4 on the right-hand side.

\subsection{Energy estimates on $Q_{\neq}^3$}\label{4.4.1}
First, $Q_{\neq}^3$ satisfies the following equation
\begin{align*}
&\partial_{t} Q_{\neq}^{3}-\nu \Delta_{L} Q_{\neq}^{3}+2  \partial_{X Y}^{L} U_{\neq}^{3}-  \partial_{X Z} U_{\neq}^{2}- \partial_{Z Y}^{L} U_{\neq}^{1}  \nonumber\\
&\quad=-\left(U \cdot \nabla_{L} Q^{3}\right)_{\neq} -\left(Q^{j} \partial_{j}^{L} U^{3}\right)_{\neq}-2 \left(\partial_{i}^{L} U^{j} \partial_{i j}^{L} U^{3}\right)_{\neq}+\partial_{Z}\left(\partial_{i}^{L} U^{j} \partial_{j}^{L} U^{i}\right)_{\neq}. 
\end{align*}
An energy estimates gives
\begin{align}\label{6.2}
&\frac{1}{2} \left\| m  M Q_{\neq}^3 (T) \right\|_{H^N}^2+\left\| \sqrt{-\dot{M}M} m   Q_{\neq}^3  \right\|_{L^2 H^N}^2+ \left\| \sqrt{-\dot{m}m} M Q_{\neq}^3 \right\|_{L^2 H^N}^2+\nu \left\| m M \nabla_{L} Q_{\neq}^3 \right\|_{L^2 H^N}^2  \nonumber \\
&\quad=\frac{1}{2} \left\| m  M Q_{\neq}^3 (1) \right\|_{H^N}^2 -\int_{1}^{T}\int m  M {\left\langle D \right\rangle}^{N} Q_{\neq}^3  m M {\left\langle D \right\rangle}^{N} \Big[ 2 \partial_{X Y}^{L} U_{\neq}^3 - \partial_{X Z} U_{\neq}^2- \partial_{Z Y}^{L} U_{\neq}^1 \nonumber \\
&\qquad+\left(U \cdot \nabla_{L} Q^{3}\right)_{\neq} +\left(Q^{j} \partial_{j}^{L} U^{3}\right)_{\neq}+2 \left(\partial_{i}^{L} U^{j} \partial_{i j}^{L} U^{3}\right)_{\neq}-\partial_{Z}\left(\partial_{i}^{L} U^{j} \partial_{j}^{L} U^{i}\right)_{\neq} \Big]    dVdt \nonumber \\
&\quad\overset{\Delta}{=}\frac{1}{2} \left\| m  M Q_{\neq}^3 (1) \right\|_{H^N}^2+\mathcal{LS}3+\mathcal{LP}+\mathcal{LU}3+ \mathcal{TQ}+\mathcal{NLS}1+\mathcal{NLS}2+\mathcal{NLP},  
\end{align}
where the following fact
\begin{align*}
&\int m  M {\left\langle D \right\rangle}^{N} Q_{\neq}^3  m M {\left\langle D \right\rangle}^{N} \partial_t Q_{\neq}^3 dV  \nonumber \\&\quad
=\frac{1}{2}\frac{d}{dt} \left\| m  M Q_{\neq}^3  \right\|_{H^N}^2+\left\| \sqrt{-\dot{M}M} m   Q_{\neq}^3  \right\|_{H^N}^2+ \left\| \sqrt{-\dot{m}m} M Q_{\neq}^3 \right\|_{H^N}^2,
\end{align*}
is used.

We estimate the linear stretching term $\mathcal{LS}3$. By the definition of $m$, H\"{o}lder's inequality and (\ref{4...3}), it gives
\begin{align}\label{6.3}
\mathcal{LS}3 =&-\int_{1}^{T}\int m  M {\left\langle D \right\rangle}^{N} Q_{\neq}^3  m M {\left\langle D \right\rangle}^{N} \left(2  \partial_{X Y}^{L} U_{\neq}^3 \right) dVdt   \nonumber \\
=& \int_{1}^{T} \int m M \widehat{{\left\langle D \right\rangle}^N}  \widehat{Q_{\neq}^3} m M \widehat{{\left\langle D \right\rangle}^N} \frac{-2 k (\eta- kt)}{k^2+(\eta- kt)^2+l^2} \widehat{Q_{\neq}^3} dV dt  \nonumber \\
=&-\int_{1}^{T} \int m M \widehat{{\left\langle D \right\rangle}^N}  \widehat{Q_{\neq}^3} m M \widehat{{\left\langle D \right\rangle}^N} \frac{2 k (\eta- kt)}{k^2+(\eta- kt)^2+l^2} \cdot 1_{\{t: \,  0<  t <\frac{\eta}{k}\}} \widehat{Q_{\neq}^3} dV dt\nonumber \\
&+ \int_{1}^{T} \int m M \widehat{{\left\langle D \right\rangle}^N}  \widehat{Q_{\neq}^3} m M \widehat{{\left\langle D \right\rangle}^N} \frac{-\dot{m}}{m} \cdot 1_{\{t: \,   t \in [\frac{\eta}{k}, \frac{\eta}{k}+1000\nu^{-\frac{1}{3}}]\}} \widehat{Q_{\neq}^3} dV dt  \nonumber \\
&-\int_{1}^{T} \int m M \widehat{{\left\langle D \right\rangle}^N}  \widehat{Q_{\neq}^3} m M \widehat{{\left\langle D \right\rangle}^N} \frac{2 k (\eta- kt)}{k^2+(\eta- kt)^2+l^2} \cdot 1_{\{t: \,    t >\frac{\eta}{k}+1000  \nu^{-\frac{1}{3}}\}} \widehat{Q_{\neq}^3} dV dt            \nonumber \\
\leqslant &  \left\| \sqrt{-\dot{m}m} M Q_{\neq}^3 \right\|_{L^2 H^N}^2+ \frac{\nu}{2} \left\| m M \nabla_{L} Q_{\neq}^3   \right\|_{L^2 H^N}^2.
\end{align}

For the linear pressure term $\mathcal{LP}$, by $U_{\neq}^2= -|\partial_{X}|^{-1} |\nabla_{L}|^{-1} \check{K}_{\neq}^2$,  H\"{o}lder's inequality, \eqref{4..5}--\eqref{4..8} and Corollary \ref{cor4.1}, we have
\begin{align}\label{6.4}
\mathcal{LP} \leqslant&  \int_{1}^{T}\int m  M {\left\langle D \right\rangle}^{N} Q_{\neq}^3  m M {\left\langle D \right\rangle}^{N} \partial_{X Z} U_{\neq}^2 dVdt \nonumber \\
\leqslant &  \left\|  m M Q_{\neq}^3  \right\|_{L^2 H^N} \left\| M \check{K}_{\neq}^2   \right\|_{L^2 H^N} \nonumber \\
\lesssim &  C_0 \varepsilon  \nu^{-\frac{1}{3}} \nu^{-\frac{1}{6}} \varepsilon \nu^{-\frac{1}{6}} \nonumber\\
\lesssim & \left(C_0 \varepsilon  \nu^{-\frac{1}{3}}  \right)^2 \left( \frac{1}{C_0} \right) \nonumber \\
\lesssim & \left(C_0 \varepsilon \nu^{-\frac{1}{3}}  \right)^2,
\end{align}
which suffices for $C_0$ sufficiently large.
 
For  the estimate of lift up term $\mathcal{LU}3$, using $U_{\neq}^1= -|\nabla_{X, Z}|^{-1} |\nabla_{L}|^{-1} \check{K}_{\neq}^1$, \eqref{4..2}, \eqref{4..27} and Corollary \ref{cor4.1} yields
\begin{align}\label{6.5}
\mathcal{LU}3 =& \int_{1}^{T}\int m  M {\left\langle D \right\rangle}^{N} Q_{\neq}^3  m M {\left\langle D \right\rangle}^{N}  \partial_{Z Y}^{L} U_{\neq}^1 dVdt \nonumber \\
\leqslant&  \left\| m Q_{\neq}^3   \right\|_{L^2 H^N}  \left\| M \check{K}_{\neq}^1 \right\|_{L^2 H^N}   \nonumber \\
\lesssim &  C_0 \varepsilon \nu^{-\frac{1}{2}}   \varepsilon \nu^{-\frac{1}{6}} \nonumber \\
\lesssim & \left(C_0 \varepsilon \nu^{-\frac{1}{3}}  \right)^2 \left( \frac{1}{C_0}\right)   \nonumber \\
\lesssim & \left(C_0 \varepsilon \nu^{-\frac{1}{3}}  \right)^2.
\end{align}
We estimate the transport term $\mathcal{TQ}$ which is written  as
\begin{align}\label{6.6}
\mathcal{TQ}=&-\int_{1}^{T}\int m M {\left\langle D \right\rangle}^{N} Q_{\neq}^3 m M {\left\langle D \right\rangle}^{N} \left( U \cdot \nabla_{L} Q^3 \right)_{\neq} dVdt  \nonumber \\
=& -\int_{1}^{T}\int m M {\left\langle D \right\rangle}^{N} Q_{\neq}^3 m M {\left\langle D \right\rangle}^{N} \left[U_0 \cdot \nabla_{L} Q_{\neq}^{3} + U_{\neq} \cdot \nabla Q_0^{3}+ \left( U_{\neq} \cdot \nabla_{L} Q_{\neq}^{3} \right)_{\neq}\right] dVdt  \nonumber \\
\overset{\Delta}{=}& \mathcal{TQ}_{0, \neq} +\mathcal{TQ}_{\neq, 0}+\mathcal{TQ}_{\neq, \neq}.
\end{align}
 
For $\mathcal{TQ}_{0, \neq}$, by \eqref{4..8}, \eqref{4..16}--\eqref{4..18} and \eqref{4..27}, we obtain
\begin{align}
\mathcal{TQ}_{0, \neq} \leqslant & \left\| m  Q_{\neq}^3  \right\|_{L^2 H^N} \left\| U_0  \right\|_{L^{\infty} H^N} \left\| m M \nabla_{L} Q_{\neq}^3  \right\|_{L^2 H^N}  \nonumber \\
\lesssim & C_0 \varepsilon  \nu^{-\frac{1}{2}} \left( \varepsilon+C_1 \varepsilon \nu^{-1}+  C_0 \varepsilon \nu^{-1}\right) C_0 \varepsilon  \nu^{-\frac{1}{3}} \nu^{-\frac{1}{2}} \nonumber \\
\lesssim & \left(C_0 \varepsilon \nu^{-\frac{1}{3}}  \right)^2  \left(\varepsilon \nu^{-\frac{2}{3}} +C_1 \varepsilon \nu^{-\frac{5}{3}}+ C_0 \varepsilon \nu^{-\frac{5}{3}}\right)  \nonumber \\
\lesssim & \left(C_0 \varepsilon  \nu^{-\frac{1}{3}}  \right)^2.
\end{align} 
Considering  $\mathcal{TQ}_{\neq, 0}$ and using \eqref{4..7}, \eqref{4..23}, \eqref{4..24} and \eqref{4..27} yields
\begin{align}
\mathcal{TQ}_{\neq, 0} = & -\int_{1}^{T}\int m M {\left\langle D \right\rangle}^{N} Q_{\neq}^3 m M {\left\langle D \right\rangle}^{N} \left( U_{\neq}^2 \cdot \partial_{Y} Q_0^{3}+ U_{\neq}^3 \cdot \partial_{Z} Q_0^{3} \right) dVdt  \nonumber \\
\leqslant & \left\| m  Q_{\neq}^3  \right\|_{L^2 H^N} \left\| U_{\neq}^{2, 3}  \right\|_{L^{\infty} H^N} \left\|  \nabla Q_0^3  \right\|_{L^2 H^N} \nonumber \\
\lesssim & C_0 \varepsilon  \nu^{-\frac{1}{2}} \left( \varepsilon + C_0 \varepsilon \nu^{-\frac{1}{3}}  \right) C_0 \varepsilon \nu^{-1}  \nu^{-\frac{1}{2}}   \nonumber \\
\lesssim & \left(C_0 \varepsilon \nu^{-\frac{1}{3}}  \right)^2 \left(\varepsilon \nu^{-\frac{4}{3}}+ C_0 \varepsilon  \nu^{-\frac{5}{3}} \right) \nonumber \\
\lesssim & \left(C_0 \varepsilon \nu^{-\frac{1}{3}}  \right)^2.
\end{align}
For $\mathcal{TQ}_{\neq, \neq}$, using \eqref{4..8} and \eqref{4..22}--\eqref{4..24} yields
\begin{align*}
\mathcal{TQ}_{\neq, \neq} \leqslant & \left\|  m  M Q_{\neq}^3 \right\|_{L^{\infty} H^N} \left\|  U_{\neq}^{1, 2} \right\|_{L^2 H^N} \left\| m  M \nabla_{L} Q_{\neq}^3   \right\|_{L^2 H^N}   \nonumber \\
&+ \left\|  m  M Q_{\neq}^3 \right\|_{L^{\infty} H^N} \left\|  U_{\neq}^3 \right\|_{L^2 H^N} \left\| m  M \nabla_{L} Q_{\neq}^3   \right\|_{L^2 H^N} \nonumber \\
\lesssim & C_0 \varepsilon \nu^{-\frac{1}{3}} \varepsilon \nu^{-\frac{1}{6}} C_0 \varepsilon \nu^{-\frac{1}{3}}  \nu^{-\frac{1}{2}} + C_0 \varepsilon \nu^{-\frac{1}{3}}  C_0  \varepsilon \nu^{-\frac{1}{3}}  \nu^{-\frac{1}{6}} C_0 \varepsilon \nu^{-\frac{1}{3}}  \nu^{-\frac{1}{2}}  \nonumber \\
\lesssim & \left(C_0 \varepsilon \nu^{-\frac{1}{3}}  \right)^2 \left(\varepsilon \nu^{-\frac{2}{3}}+ C_0 \varepsilon  \nu^{-1} \right)  \nonumber \\
\lesssim & \left(C_0 \varepsilon \nu^{-\frac{1}{3}}  \right)^2.
\end{align*}

For the estimate of the nonlinear stretching term $\mathcal{NLS} 1$, we first divide $\mathcal{NLS} 1$ into
\begin{align}\label{6.10}
\mathcal{NLS} 1= &-\int_{1}^{T}\int m M {\left\langle D \right\rangle}^{N} Q_{\neq}^3 m M {\left\langle D \right\rangle}^{N} \left[Q_0^j \partial_{j}^{L} U_{\neq}^{3} + Q_{\neq}^j \partial_{j}U_0^{3}+ \left( Q_{\neq}^j \partial_{j}^{L} U_{\neq}^{3} \right)_{\neq}\right] dVdt  \nonumber \\
\overset{\Delta}{=}& \mathcal{NLS}  1 (0, \neq)+\mathcal{NLS}  1 (\neq, 0)+\mathcal{NLS}  1 (\neq, \neq).
\end{align}
For $\mathcal{NLS}  1 (0, \neq)$, by using \eqref{4..1}, \eqref{4..4}, \eqref{4..7}, \eqref{4..24} and \eqref{4..27}, it holds
\begin{align}
\mathcal{NLS}  1 (0, \neq) \leqslant & \left\| m  Q_{\neq}^3  \right\|_{L^2 H^N} \left\| Q_0^{j}  \right\|_{L^{\infty} H^N} \left\| \nabla_{L} U_{\neq}^3  \right\|_{L^2 H^N}  \nonumber \\
\lesssim & C_0 \varepsilon \nu^{-\frac{1}{2}}  \left( \varepsilon+ C_1 \varepsilon \nu^{-1} + C_0 \varepsilon \nu^{-1}\right) C_0 \varepsilon \nu^{-\frac{1}{3}}  \nu^{-\frac{1}{2}} \nonumber \\
\lesssim & \left(C_0 \varepsilon \nu^{-\frac{1}{3}}  \right)^2\left( \varepsilon \nu^{-\frac{2}{3}}+C_1 \varepsilon \nu^{-\frac{5}{3}}+C_0 \varepsilon \nu^{-\frac{5}{3}}\right)  \nonumber \\
\lesssim &\left(C_0 \varepsilon \nu^{-\frac{1}{3}}  \right)^2.
\end{align}
For $\mathcal{NLS}  1 (\neq, 0)$, we use $Q_{\neq}^2= -|\partial_{X}|^{-1} |\nabla_{L}| \check{K}_{\neq}^2$, \eqref{4..5}, \eqref{4..8}, \eqref{4..18} and \eqref{4..27} to deduce
\begin{align}
\mathcal{NLS}  1 (\neq, 0) \leqslant & \left\| m  Q_{\neq}^3  \right\|_{L^2 H^N} \left\| M Q_{\neq}^{2, 3}  \right\|_{L^{2} H^N} \left\| \nabla U_0^3  \right\|_{L^{\infty} H^N}   \nonumber \\
\leqslant & \left\| m  Q_{\neq}^3  \right\|_{L^2 H^N} \left( \left\| M \nabla_{L} \check{K}_{\neq}^{2}  \right\|_{L^{2} H^N}+  \left\| m M Q_{\neq}^3 \right\|_{L^{2} H^N} \right) \left\| \nabla U_0^3  \right\|_{L^{\infty} H^N} \nonumber \\
\lesssim & C_0 \varepsilon \nu^{-\frac{1}{2}}  \left( \varepsilon \nu^{-\frac{1}{2}}  + C_0 \varepsilon \nu^{-\frac{1}{2}} \right) C_0 \varepsilon \nu^{-1}  \nonumber \\
\lesssim & \left(C_0 \varepsilon \nu^{-\frac{1}{3}}  \right)^2 \left( \varepsilon \nu^{-\frac{4}{3}}+C_0 \varepsilon \nu^{-\frac{4}{3}}\right)  \nonumber \\
\lesssim &\left(C_0 \varepsilon  \nu^{-\frac{1}{3}} \right)^2,
\end{align}
which suffices for   $C_0 \varepsilon  \nu^{-\frac{4}{3}}$  sufficiently small.
Next, by the fact $Q_{\neq}^1= -|\nabla_{X, Z}|^{-1} |\nabla_{L}| \check{K}_{\neq}^1$, $Q_{\neq}^2= -|\partial_{X}|^{-1} |\nabla_{L}| \check{K}_{\neq}^2$, \eqref{4..2}, \eqref{4..5}, \eqref{4..8}, \eqref{4..24} and Corollary \ref{cor4.1}, the term $\mathcal{NLS}  1 (\neq, \neq)$ can be estimated as follows
\begin{align*}
\mathcal{NLS}  1 (\neq, \neq) \leqslant & \left\| m  Q_{\neq}^3  \right\|_{L^{\infty} H^N} \left\| mMQ_{\neq}^{j}  \right\|_{L^{2} H^N} \left\| \nabla_{L} U_{\neq}^3  \right\|_{L^2 H^N}   \nonumber \\
\leqslant & \left\| m  Q_{\neq}^3 \right\|_{L^{\infty} H^N} \left( \left\| m  M Q_{\neq}^{1, 2}  \right\|_{L^{2} H^N} +\left\| m M Q_{\neq}^{3}  \right\|_{L^{2} H^N}\right) \left\| \nabla_{L} U_{\neq}^3  \right\|_{L^2 H^N} \nonumber \\
\leqslant & \left\| m  Q_{\neq}^3 \right\|_{L^{\infty} H^N} \left( \left\|  M \nabla_{L} \check{K}_{\neq}^{1, 2}  \right\|_{L^{2} H^N} +\left\| m M Q_{\neq}^{3}  \right\|_{L^{2} H^N}\right) \left\| \nabla_{L} U_{\neq}^3  \right\|_{L^2 H^N} \nonumber \\
\lesssim & C_0 \varepsilon \nu^{-\frac{1}{3}}  \left( \varepsilon \nu^{-\frac{1}{2}}+ C_0 \varepsilon \nu^{-\frac{1}{3}} \nu^{-\frac{1}{6}}\right) C_0 \varepsilon \nu^{-\frac{1}{3}}  \nu^{-\frac{1}{2}} \nonumber \\
\lesssim & \left(C_0 \varepsilon \nu^{-\frac{1}{3}}  \right)^2 \left(\varepsilon \nu^{-1}+C_0 \varepsilon \nu^{-1}\right)  \nonumber \\
\lesssim &\left(C_0 \varepsilon \nu^{-\frac{1}{3}}  \right)^2,
\end{align*}
which suffices for  $C_0 \varepsilon \nu^{-1}$  sufficiently small.

We estimate for the nonlinear stretching term $\mathcal{NLS}2$ which is divided into
\begin{align*}
\mathcal{NLS}  2=& -2\int_{1}^{T}\int m M {\left\langle D \right\rangle}^{N} Q_{\neq}^3 m  M {\left\langle D \right\rangle}^{N} \Big[   \partial_{i} U_0^{j} \partial_{i j}^{L} U_{\neq}^{3} + \partial_{i}^{L} U_{\neq}^{j} \partial_{i j} U_0^{3} \nonumber \\
&   + \left(\partial_{i}^{L} U_{\neq}^{j} \partial_{i j}^{L} U_{\neq}^{3} \right)_{\neq} \Big]   dVdt \nonumber \\
\overset{\Delta}{=}& \mathcal{NLS} 2 (0, \neq)+\mathcal{NLS} 2 (\neq, 0)+\mathcal{NLS} 2 (\neq, \neq).
\end{align*}
For $\mathcal{NLS} 2 (0, \neq)$, we use \eqref{4..16}--\eqref{4..18}, \eqref{4..21}, \eqref{4..27} and Corollary \ref{cor4.1} to obtain
\begin{align}
\mathcal{NLS} 2 (0, \neq) \leqslant & \left\| m Q_{\neq}^3 \right\|_{L^2 H^N} \left\| \nabla U_0^j \right\|_{L^{\infty} H^{N}} \left\| m  \Delta_{L} U_{\neq}^3  \right\|_{L^2 H^N} \cdot 1_{i \neq 1}  \nonumber \\
\leqslant & \left\| m  Q_{\neq}^3 \right\|_{L^2 H^N} \left( \left\|  U_0^1 \right\|_{L^{\infty} H^{N+1}} +\left\|  U_0^{2, 3} \right\|_{L^{\infty} H^{N+1}}   \right)  \left\|  m  \Delta_{L} U_{\neq}^3 \right\|_{L^2 H^N}   \nonumber \\
\lesssim & C_0 \varepsilon \nu^{-\frac{1}{2}}  \left( \varepsilon + C_1 \varepsilon \nu^{-1}  + C_0 \varepsilon \nu^{-1} \right) C_0 \varepsilon \nu^{-\frac{1}{3}} \nu^{-\frac{1}{6}} \nonumber \\
\lesssim & \left(C_0 \varepsilon \nu^{-\frac{1}{3}}  \right)^2 \left( \varepsilon \nu^{-\frac{1}{3}} +C_1 \varepsilon \nu^{-\frac{4}{3}}+C_0 \varepsilon \nu^{-\frac{4}{3}}\right)  \nonumber \\
\lesssim &\left(C_0 \varepsilon \nu^{-\frac{1}{3}}  \right)^2.
\end{align}
Notice that $\partial_{i j} U_0^{3}=0$ if $i=1$ or $j=1$. By \eqref{4..18}, \eqref{4..23}, \eqref{4..24} and \eqref{4..27}, it holds
\begin{align}
\mathcal{NLS} 2 (\neq, 0) \leqslant & \left\| m Q_{\neq}^3 \right\|_{L^2 H^N} \left\| \nabla_{L} U_{\neq}^j \right\|_{L^2 H^N} \left\|  \nabla  U_0^3  \right\|_{L^{\infty} H^{N+1}} \cdot 1_{i, j \neq 1}  \nonumber \\
\leqslant & \left\| m  Q_{\neq}^3 \right\|_{L^2 H^N} \left( \left\| \nabla_{L} U_{\neq}^2   \right\|_{L^2 H^N} +\left\| \nabla_{L} U_{\neq}^3   \right\|_{L^2 H^N}     \right) \left\|  U_0^3  \right\|_{L^{\infty} H^{N+2}} \nonumber \\
\lesssim & C_0 \varepsilon \nu^{-\frac{1}{2}}   \left(  \varepsilon \nu^{-\frac{1}{2}} + C_0 \varepsilon \nu^{-\frac{1}{3}} \nu^{-\frac{1}{2}}\right) C_0 \varepsilon  \nu^{-1}  \nonumber \\
\lesssim & \left(C_0 \varepsilon \nu^{-\frac{1}{3}}  \right)^2 \left( \varepsilon \nu^{-\frac{4}{3}} +C_0 \varepsilon  \nu^{-\frac{5}{3}}\right)  \nonumber \\
\lesssim &\left(C_0 \varepsilon \nu^{-\frac{1}{3}}  \right)^2.
\end{align} 
By \eqref{4.11}, \eqref{4..8}, \eqref{4..22}--\eqref{4..24} and Corollary \ref{cor4.1}, we can estimate $\mathcal{NLS} 2 (\neq, \neq)$ as follows
\begin{align*}
\mathcal{NLS} 2 (\neq, \neq)  \leqslant & \left\| m M Q_{\neq}^3 \right\|_{L^{\infty} H^N} \left( \left\|  \nabla_{L} U_{\neq}^{1, 2} \right\|_{L^2 H^N} +\left\|  \nabla_{L} U_{\neq}^3 \right\|_{L^2 H^N}    \right)  \left\|  m \Delta_{L} U_{\neq}^3 \right\|_{L^2 H^N} \sup m^{-1}   \nonumber \\
\lesssim & C_0 \varepsilon  \nu^{-\frac{1}{3}} \left( \varepsilon \nu^{-\frac{1}{2}} + C_0 \varepsilon \nu^{-\frac{1}{3}} \nu^{-\frac{1}{2}}\right) C_0  \varepsilon \nu^{-\frac{1}{3}}  \nu^{-\frac{1}{6}}  \nu^{-\frac{2}{3}} \nonumber \\
\lesssim & \left(C_0 \varepsilon \nu^{-\frac{1}{3}}  \right)^2 \left(\varepsilon \nu^{-\frac{4}{3}}+C_0 \varepsilon  \nu^{-\frac{5}{3}}\right)  \nonumber \\
\lesssim &\left(C_0 \varepsilon \nu^{-\frac{1}{3}}  \right)^2.
\end{align*}
 
When we estimate  the nonlinear pressure term $\mathcal{NLP}$, it is divided into
\begin{align}\label{6.18}
\mathcal{NLP} =& \int_{1}^{T}\int m M {\left\langle D \right\rangle}^{N} Q_{\neq}^3 m M {\left\langle D \right\rangle}^{N} \partial_{Z} \left[ \partial_{i} U_0^{j} \partial_{j}^{L} U_{\neq}^{i} +\partial_{i}^{L} U_{\neq}^{j} \partial_{j} U_0^{i} +\left( \partial_{i}^{L} U_{\neq}^{j} \partial_{j}^{L} U_{\neq}^{i}\right)_{\neq}   \right] dVdt \nonumber \\
\overset{\Delta}{=}& \mathcal{NLP} (0, \neq)+\mathcal{NLP} (\neq, 0)+\mathcal{NLP} (\neq, \neq).
\end{align}
For $\mathcal{NLP} (0, \neq)$, by  integration by parts and using  \eqref{4..8}, \eqref{4..16}--\eqref{4..18}, \eqref{4..23} and \eqref{4..24}, we deduce
\begin{align}
\mathcal{NLP} (0, \neq) =& -\int_{1}^{T}\int m M {\left\langle D \right\rangle}^{N} \partial_{Z} Q_{\neq}^3 m M {\left\langle D \right\rangle}^{N}    \left(\partial_{i} U_0^{j} \partial_{j}^{L} U_{\neq}^{i} \right)dVdt \cdot 1_{i \neq 1} \nonumber \\
\leqslant & \left\| m M \nabla_{L} Q_{\neq}^3  \right\|_{L^2 H^N}  \left\| U_0^j  \right\|_{L^{\infty} H^{N+1}}  \left(\left\| \nabla_{L} U_{\neq}^{2}  \right\|_{L^2 H^N}+\left\| \nabla_{L} U_{\neq}^{3}  \right\|_{L^2 H^N} \right)  \nonumber \\
\lesssim & C_0 \varepsilon \nu^{-\frac{1}{3}}  \nu^{-\frac{1}{2}} \left(\varepsilon+ C_1 \varepsilon \nu^{-1} + C_0 \varepsilon \nu^{-1}\right)\left( \varepsilon \nu^{-\frac{1}{2}} + C_0 \varepsilon \nu^{-\frac{1}{3}} \nu^{-\frac{1}{2}}\right) \nonumber \\
\lesssim & \left(C_0 \varepsilon \nu^{-\frac{1}{3}}  \right)^2 \left(C_0 \varepsilon \nu^{-2}\right)  \nonumber \\
\lesssim &\left(C_0 \varepsilon \nu^{-\frac{1}{3}}  \right)^2,
\end{align}
which suffices for $C_0 \varepsilon  \nu^{-2}$  sufficiently small. Similarly, we can obtain 
\begin{align}
\mathcal{NLP} (\neq, 0) \lesssim \left(C_0 \varepsilon \nu^{-\frac{1}{3}}  \right)^2.
\end{align}
For $\mathcal{NLP} (\neq, \neq)$, using integration by parts, \eqref{4.11}, \eqref{4..2}--\eqref{4..8}, \eqref{4..22}--\eqref{4..24}, \eqref{5.1} and the fact $|\widehat{\nabla_{L}}| \leqslant | \widehat{m^{\frac{1}{2}} \Delta_{L}}|$ yields
\begin{align}
\mathcal{NLP} (\neq, \neq) =& -\int_{1}^{T}\int m M {\left\langle D \right\rangle}^{N} \partial_{Z} Q_{\neq}^3 m M {\left\langle D \right\rangle}^{N}  \left( \partial_{i}^{L} U_{\neq}^{j} \partial_{j}^{L} U_{\neq}^{i}\right)_{\neq}    dVdt \nonumber \\
\leqslant & \left\| m  M \nabla_{L} Q_{\neq}^3 \right\|_{L^2 H^N} \left\| \nabla_{L} U_{\neq}^j \right\|_{L^{\infty} H^N} \left\| \nabla_{L} U_{\neq}^i \right\|_{L^2 H^N}  \nonumber \\
\leqslant & \left\| m  M \nabla_{L} Q_{\neq}^3 \right\|_{L^2 H^N} \left(  \left\| M \check{K}_{\neq}^{1, 2}  \right\|_{L^{\infty} H^N}+\left\| m \Delta_{L} U_{\neq}^{3}  \right\|_{L^{\infty} H^N} \sup m^{-1}  \right) \nonumber \\
& \times \left( \left\| \nabla_{L} U_{\neq}^1 \right\|_{L^2 H^N} +\left\| \nabla_{L} U_{\neq}^2 \right\|_{L^2 H^N}  +\left\| \nabla_{L} U_{\neq}^3 \right\|_{L^2 H^N}  \right) \nonumber \\
\lesssim & C_0 \varepsilon \nu^{-\frac{1}{3}}  \nu^{-\frac{1}{2}} \left( \varepsilon + C_0 \varepsilon \nu^{-\frac{1}{3}}   \nu^{-\frac{2}{3}}\right) \times \left( \varepsilon \nu^{-\frac{1}{2}}+C_0\varepsilon \nu^{-\frac{1}{3}}  \nu^{-\frac{1}{2}}\right)  \nonumber \\
\lesssim & \left(C_0 \varepsilon \nu^{-\frac{1}{3}}  \right)^2\left( C_0 \varepsilon  \nu^{-2}\right) \nonumber \\
\lesssim & \left(C_0 \varepsilon \nu^{-\frac{1}{3}}  \right)^2.
\end{align}
Submitting the estimates (\ref{6.3})--(\ref{6.6}), (\ref{6.10}) and (\ref{6.18}) into (\ref{6.2}), then (\ref{4..8}) holds with $8$ replace by $4$ on the right-hand side.

\subsection{Energy estimates on $Q_0^3$}\label{4.4.2}
Notice that $Q_0^3$ satisfies the following equation
\begin{align*}
&\partial_{t} Q_0^{3}-\nu \Delta Q_0^{3}- \partial_{Z Y} U_0^{1}  \nonumber \\
&\quad=-\left(U \cdot \nabla_{L} Q^{3}\right)_0-\left(Q^{j} \partial_{j}^{L} U^{3}\right)_0-2 \left(\partial_{i}^{L} U^{j} \partial_{i j}^{L} U^{3} \right)_0+\partial_{Z}\left(\partial_{i}^{L} U^{j} \partial_{j}^{L} U^{i}\right)_0.
\end{align*}
We give the $H^N$ estimate for $Q_0^3$
\begin{align}\label{4..176}
\frac{1}{2} \left\|  Q_0^3 (T)  \right\|_{H^N}^2+\nu \left\|   \nabla Q_0^3 \right\|_{L^2 H^N}^2 &= \frac{1}{2} \left\|  Q_0^3 (1)  \right\|_{H^N}^2+ \int_{1}^{T} \int  {\left\langle D \right\rangle}^{N}  Q_0^3  {\left\langle D \right\rangle}^{N} \Big[   \partial_{Z Y} U_0^{1} - \left(U \cdot \nabla_{L} Q^3 \right)_0 \nonumber \\
&\quad - \left(Q^{j} \partial_{j}^{L} U^3\right)_0-2 \left(\partial_{i}^{L} U^{j} \partial_{i j}^{L} U^3 \right)_0 +\partial_{Z} \left(\partial_{i}^{L} U^{j} \partial_{j}^{L} U^{i}\right)_0 \Big]dVdt  \nonumber \\
&\overset{\Delta}{=}\frac{1}{2} \left\|  Q_0^3 (1) \right\|_{H^N}^2+ \mathcal{LU}3+\mathcal{T}+\mathcal{NLS}1+\mathcal{NLS}2+\mathcal{NLP},
\end{align}
where the right-hand side  of  \eqref{4..176} will be estimated term by term as follows. We first consider the lift up term  $\mathcal{LU}3$. Using H\"{o}lder's inequality, \eqref{4..16} and \eqref{4..18} yields
\begin{align}
\mathcal{LU}3 =& \int_{1}^{T} \int  {\left\langle D \right\rangle}^{N}  Q_0^3  {\left\langle D \right\rangle}^{N}  \partial_{Z Y} \Delta^{-1} Q_0^{1} dVdt \nonumber \\
\leqslant &  \int_{1}^{T} \int (1+|k, \eta, l|^2)^{\frac{N}{2}} \widehat{Q_0^3} \, (1+|k, \eta, l|^2)^{\frac{N}{2}} \frac{\eta l}{k^2+\eta^2+l^2} \widehat{Q_0^1} d \xi dt \nonumber \\
\leqslant &  \int_{1}^{T} \int {\left\langle D \right\rangle}^{N} Q_0^3 \cdot {\left\langle D \right\rangle}^{N} Q_0^1 dVdt \nonumber \\
\leqslant &  \left\|  \nabla U_0^3 \right\|_{L^2 H^{N+1}} \left\|  \nabla U_0^1 \right\|_{L^2 H^{N+1}} \nonumber \\
\lesssim &  C_0 \varepsilon  \nu^{-1} \nu^{-\frac{1}{2}} \varepsilon \nu^{-\frac{1}{2}}  \nonumber \\
\lesssim & \left( C_0 \varepsilon  \nu^{-1}\right)^2 \left(   \frac{1}{C_0}\right)   \nonumber \\
\lesssim & \left( C_0 \varepsilon  \nu^{-1}\right)^2, 
\end{align}
which suffices for $C_0$ sufficiently large. The term $\mathcal{T}$ can be divided into
\begin{align*}
\mathcal{T}= - \int_{1}^{T} \int  {\left\langle D \right\rangle}^{N}  Q_0^3  {\left\langle D \right\rangle}^{N} \Big[ U_0 \cdot \nabla Q_0^{3} +\left(U_{\neq} \cdot \nabla_{L} Q_{\neq}^{3} \right)_0 \Big]dVdt \overset{\Delta}{=} \mathcal{T}_1+\mathcal{T}_2.
\end{align*}
For $\mathcal{T}_1$, we use \eqref{4..7} and \eqref{4..17}--\eqref{4..18} to deduce
\begin{align}
\mathcal{T}_1 = &- \int_{1}^{T} \int  {\left\langle D \right\rangle}^{N}  Q_0^3  {\left\langle D \right\rangle}^{N} \left( U_0^2 \partial_{Y} Q_0^3 +U_0^3 \partial_{Z} Q_0^3 \right)dVdt  \nonumber \\
\leqslant & \left\| \nabla U_0^3 \right\|_{L^2 H^{N+1}} \left\| U_0^{2, 3} \right\|_{L^{\infty} H^N} \left\| \nabla Q_0^3 \right\|_{L^2 H^N}  \nonumber \\
\lesssim &  C_0 \varepsilon  \nu^{-1} \nu^{-\frac{1}{2}}  \left( C_1 \varepsilon  \nu^{-1} + C_0 \varepsilon  \nu^{-1} \right) C_0 \varepsilon  \nu^{-1} \nu^{-\frac{1}{2}}  \nonumber \\
\lesssim & \left( C_0 \varepsilon  \nu^{-1}\right)^2 \left( C_1 \varepsilon  \nu^{-2} + C_0 \varepsilon  \nu^{-2}      \right) \nonumber\\
\lesssim & \left( C_0 \varepsilon  \nu^{-1}\right)^2,
\end{align}
which suffices for $C_0 \varepsilon  \nu^{-2}$  sufficiently small. For $\mathcal{T}_2$, by \eqref{4.11}, \eqref{4..2}--\eqref{4..8}, \eqref{4..18} and \eqref{4..24}, we have
\begin{align}
\mathcal{T}_2 \leqslant & \left\| \nabla U_0^3 \right\|_{L^{\infty} H^{N+1}} \left\| U_{\neq}^i \right\|_{L^{2} H^N} \left\| \nabla_{L} Q_{\neq}^3 \right\|_{L^2 H^N}  \nonumber \\
\leqslant & \left\| \nabla U_0^3 \right\|_{L^{\infty} H^{N+1}} \left( \left\| M \check{K}_{\neq}^{1, 2} \right\|_{L^{2} H^N} + \left\|  U_{\neq}^3 \right\|_{L^{2} H^N}  \right) \nonumber \\
&  \times \left\| m  M \nabla_{L} Q_{\neq}^3 \right\|_{L^2 H^N} \sup m^{-1}  \nonumber \\
\lesssim & C_0 \varepsilon  \nu^{-1}  \left( \varepsilon  \nu^{-\frac{1}{6}} + C_0 \varepsilon  \nu^{-\frac{1}{3}} \nu^{-\frac{1}{6}}   \right) C_0 \varepsilon \nu^{-\frac{1}{3}}   \nu^{-\frac{1}{2}}  \nu^{-\frac{2}{3}} \nonumber \\
\lesssim & \left( C_0 \varepsilon  \nu^{-1}\right)^2 \left({C_0} \varepsilon  \nu^{-1} + \varepsilon  \nu^{-\frac{2}{3}}      \right) \nonumber\\
\lesssim & \left( C_0 \varepsilon  \nu^{-1}\right)^2.
\end{align}

Next, we consider the estimate for $\mathcal{NLS}1$ as follows
\begin{align*}
\mathcal{NLS}1&=- \int_{1}^{T} \int  {\left\langle D \right\rangle}^{N}  Q_0^3  {\left\langle D \right\rangle}^{N} \Big[  Q_0^{j} \partial_{j} U_0^3+ \left(Q_{\neq}^{j} \partial_{j}^{L} U_{\neq}^3\right)_0  \Big]dVdt \nonumber \\
&\overset{\Delta}{=} \mathcal{NLS}1 (0, 0)+\mathcal{NLS}1(\neq, \neq).
\end{align*}
For $\mathcal{NLS}1 (0, 0)$, note that $\partial_{j} U_0^3=0$ if $j=1$. And applying \eqref{4..7}, \eqref{4..17}--\eqref{4..18} yields
\begin{align}
\mathcal{NLS}1 (0, 0) \leqslant & \left\| Q_0^3  \right\|_{L^{\infty} H^N} \left\| Q_0^j  \right\|_{L^2 H^N} \left\|   \nabla U_0^3 \right\|_{L^2 H^N} \cdot  1_{j \neq 1} \nonumber \\
\leqslant & \left\| Q_0^3  \right\|_{L^{\infty} H^N} \left\| \nabla  U_0^{2, 3}  \right\|_{L^2 H^{N+1}} \left\|   \nabla U_0^3 \right\|_{L^2 H^N}  \nonumber \\
\lesssim & C_0 \varepsilon  \nu^{-1} \left(C_1 \varepsilon  \nu^{-1} \nu^{-\frac{1}{2}}+ C_0 \varepsilon  \nu^{-1} \nu^{-\frac{1}{2}} \right) C_0 \varepsilon  \nu^{-1} \nu^{-\frac{1}{2}} \nonumber \\
\lesssim & \left( C_0 \varepsilon  \nu^{-1}\right)^2 \left( C_1 \varepsilon  \nu^{-2} + C_0 \varepsilon  \nu^{-2}      \right) \nonumber\\
\lesssim & \left( C_0 \varepsilon  \nu^{-1}\right)^2.
\end{align}
To estimate $\mathcal{NLS}1(\neq, \neq)$, we use \eqref{4.11}, \eqref{4..2}--\eqref{4..5}, \eqref{4..7},  \eqref{4..24} and \eqref{5.1} to obtain
\begin{align*}
\mathcal{NLS}1(\neq, \neq) \leqslant & \left\| Q_0^3 \right\|_{L^{\infty} H^N} \left\| Q_{\neq}^j \right\|_{L^2 H^N} \left\| \nabla_{L} U_{\neq}^3  \right\|_{L^2 H^N}   \nonumber \\
\leqslant & \left\| Q_0^3 \right\|_{L^{\infty} H^N} \left( \left\|  Q_{\neq}^{1, 2} \right\|_{L^2 H^N} + \left\| m Q_{\neq}^3 \right\|_{L^2 H^N}  \sup m^{-1} \right) \left\| \nabla_{L} U_{\neq}^3  \right\|_{L^2 H^N} \nonumber \\
\leqslant & \left\| Q_0^3 \right\|_{L^{\infty} H^N} \left( \left\|  M \nabla_{L} \check{K}_{\neq}^{1, 2} \right\|_{L^2 H^N} + \left\| m Q_{\neq}^3 \right\|_{L^2 H^N} \sup m^{-1}  \right) \left\| \nabla_{L} U_{\neq}^3  \right\|_{L^2 H^N} \nonumber \\
\lesssim & C_0 \varepsilon  \nu^{-1} \left( \varepsilon \nu^{-\frac{1}{2}} + C_0 \varepsilon  \nu^{-\frac{1}{2}}    \nu^{-\frac{2}{3}}\right) C_0 \varepsilon  \nu^{-\frac{1}{3}}  \nu^{-\frac{1}{2}} \nonumber \\
\lesssim & \left( C_0 \varepsilon  \nu^{-1}\right)^2 \left(  \varepsilon \nu^{-\frac{1}{3}} + {C_0} \varepsilon  \nu^{-1}      \right) \nonumber\\
\lesssim & \left( C_0 \varepsilon  \nu^{-1}\right)^2.
\end{align*} 
The term $\mathcal{NLS}2$ can also be divided into
\begin{align}
\mathcal{NLS}2&= - 2 \int_{1}^{T} \int  {\left\langle D \right\rangle}^{N}  Q_0^3  {\left\langle D \right\rangle}^{N} \Big[   \partial_{i} U_0^{j} \partial_{i j} U_0^3 + \left(\partial_{i}^{L} U_{\neq}^{j} \partial_{i j}^{L} U_{\neq}^3 \right)_0   \Big]dVdt \nonumber\\
&\overset{\Delta}{=} \mathcal{NLS}2(0, 0)+\mathcal{NLS}2 (\neq, \neq).
\end{align}
For $\mathcal{NLS}2(0, 0)$, by \eqref{4..7} and \eqref{4..17}--\eqref{4..18}, it holds
\begin{align}
\mathcal{NLS}2(0, 0) \leqslant & \left\| Q_0^3    \right\|_{L^{\infty} H^N} \left\|   \nabla U_0^j  \right\|_{L^2 H^{N}} \left\| \nabla U_0^3 \right\|_{L^2 H^{N+1}} \cdot 1_{i, j \neq 1}\nonumber \\
\leqslant & \left\| Q_0^3    \right\|_{L^{\infty} H^N} \left\|   \nabla U_0^{2, 3} \right\|_{L^2 H^N} \left\| \nabla U_0^3 \right\|_{L^2 H^{N+1}} \nonumber \\
\lesssim & C_0 \varepsilon  \nu^{-1}   \left( C_1 \varepsilon  \nu^{-1} \nu^{-\frac{1}{2}} + C_0 \varepsilon  \nu^{-1} \nu^{-\frac{1}{2}} \right) C_0 \varepsilon  \nu^{-1} \nu^{-\frac{1}{2}}   \nonumber \\
\lesssim &  \left( C_0 \varepsilon  \nu^{-1}\right)^2 \left(  C_1 \varepsilon  \nu^{-2}+C_0 \varepsilon  \nu^{-2}  \right) \nonumber \\
\lesssim & \left( C_0 \varepsilon  \nu^{-1}\right)^2.
\end{align} 
For  $\mathcal{NLS}2 (\neq, \neq)$, using Corollary \ref{cor4.1}, \eqref{4.11}, \eqref{4..7} and \eqref{4..23}--\eqref{4..24} yields
\begin{align*}
\mathcal{NLS}2 (\neq, \neq) \leqslant & \left\|  Q_0^3 \right\|_{L^{\infty} H^N} \left( \left\| \nabla_{L} U_{\neq}^2 \right\|_{L^2 H^N}+ \left\| \nabla_{L} U_{\neq}^3 \right\|_{L^2 H^N}\right) \left\|  m  \Delta_{L} U_{\neq}^3 \right\|_{L^2 H^N} \sup m^{-1} \nonumber \\
\lesssim &  C_0 \varepsilon  \nu^{-1} \left( \varepsilon \nu^{-\frac{1}{2}}+ C_0 \varepsilon  \nu^{-\frac{1}{3}} \nu^{-\frac{1}{2}} \right) C_0 \varepsilon  \nu^{-\frac{1}{3}}  \nu^{-\frac{1}{6}}  \nu^{-\frac{2}{3}} \nonumber \\
\lesssim &  \left( C_0 \varepsilon  \nu^{-1}\right)^2 \left( \varepsilon \nu^{-\frac{2}{3}} + C_0 \varepsilon  \nu^{-1}  \right)  \nonumber \\
\lesssim & \left( C_0 \varepsilon  \nu^{-1}\right)^2.
\end{align*}
 
Finally, we need to estimate the term $\mathcal{NLP}$ which divided into 
\begin{align*}
\mathcal{NLP} &= \int_{1}^{T} \int  {\left\langle D \right\rangle}^{N}  Q_0^3  {\left\langle D \right\rangle}^{N} \partial_{Z} \Big[ \left(\partial_{i} U_0^{j} \partial_{j} U_0^{i}\right)+ \left(\partial_{i}^{L} U_{\neq}^{j} \partial_{j}^{L} U_{\neq}^{i}\right)_0 \Big] dVdt  \nonumber \\
&\overset{\Delta}{=} \mathcal{NLP}(0, 0)+\mathcal{NLP} (\neq, \neq).
\end{align*}
For $\mathcal{NLP}(0, 0)$, note that $\partial_{i} U_0^{j} \partial_{j} U_0^{i}= \partial_{i j} ( U_0^{j} U_0^{i})=0$ if $i=1$ or $j=1$. By \eqref{4..7} and \eqref{4..17}--\eqref{4..18}, we obtain
\begin{align}
\mathcal{NLP}(0, 0)\leqslant & \left\| Q_0^3 \right\|_{L^{\infty} H^N} \left\| \nabla U_0^j \right\|_{L^2 H^{N+1}} \left\| \nabla U_0^i \right\|_{L^2 H^{N+1}} \cdot 1_{i, j \neq 1} \nonumber \\
\leqslant & \left\| Q_0^3 \right\|_{L^{\infty} H^N} \left\| \nabla U_0^{2, 3} \right\|_{L^2 H^{N+1}}^2  \nonumber \\
\lesssim & C_0 \varepsilon  \nu^{-1}  \left( C_1 \varepsilon  \nu^{-1} \nu^{-\frac{1}{2}}+C_0 \varepsilon  \nu^{-1} \nu^{-\frac{1}{2}}\right)^2  \nonumber\\
\lesssim &  \left( C_0 \varepsilon  \nu^{-1}\right)^2 \left( C_0 \varepsilon  \nu^{-2} \right)  \nonumber \\
\lesssim & \left( C_0 \varepsilon  \nu^{-1}\right)^2.
\end{align}
For $\mathcal{NLP} (\neq, \neq)$, applying \eqref{4.11}, \eqref{4..5}, \eqref{4..7},  \eqref{4..23}--\eqref{4..24} and \eqref{5.1} yields
\begin{align*}
\mathcal{NLP} (\neq, \neq) &\leqslant   \left\| Q_0^3 \right\|_{L^{\infty} H^N} \left\| \nabla_{L} U_{\neq}^j \right\|_{L^2 H^N} \left\| \Delta_{L} U_{\neq}^i \right\|_{L^2 H^N} \cdot  1_{i, j \neq 1}   \nonumber \\
&\leqslant   \left\| Q_0^3 \right\|_{L^{\infty} H^N} \left( \left\| M \nabla_{L} \check{K}_{\neq}^{2} \right\|_{L^2 H^N} + \left\| m  \Delta_{L} U_{\neq}^3 \right\|_{L^2 H^N} \sup m^{-1}  \right)   \nonumber \\
& \quad  \times \left( \left\| \nabla_{L} U_{\neq}^{2}  \right\|_{L^2 H^N}+\left\| \nabla_{L} U_{\neq}^3  \right\|_{L^2 H^N}\right) \nonumber \\
& \lesssim  C_0 \varepsilon  \nu^{-1}  \left( \varepsilon \nu^{-\frac{1}{2}} + C_0 \varepsilon  \nu^{-\frac{1}{3}} \nu^{-\frac{1}{6}}  \nu^{-\frac{2}{3}}\right) \left(\varepsilon \nu^{-\frac{1}{2}}+ C_0 \varepsilon  \nu^{-\frac{1}{3}} \nu^{-\frac{1}{2}}\right)     \nonumber \\
&\lesssim   \left( C_0 \varepsilon  \nu^{-1}\right)^2 \left( \frac{1}{C_0} \varepsilon +   \varepsilon  \nu^{-\frac{2}{3}} +  C_0 \varepsilon \nu^{-1} \right)  \nonumber \\
&\lesssim  \left( C_0 \varepsilon  \nu^{-1}\right)^2.
\end{align*}

\section{Energy estimates on zero frequency  velocity $U$}\label{sec7}
In this section, the purposes are to deduce low frequency controls on the velocity. That is,  we want to prove that under the bootstrap hypotheses of Proposition \ref{pro4.1}, the estimates on $U_0^{i}$ $(i=1, 2, 3)$  hold (i.e., \eqref{4..11}--\eqref{4..13}), with 8 replaced by 4 on the right-hand side. By the definition of new variable as in (\ref{2.5}), we can move from one coordinate system to another.  Specially,
\begin{align*}
	\left\| U_0^i \right\|_{H^N} = \left\| u_0^i \right\|_{H^N}.
\end{align*}
So, it suffices to prove these estimates on $u_{0}^{i}$, rather that $U_{0}^{i}$.
In the original coordinates, $u_0^{i}$ $(i=1, 2, 3)$ satisfy the equations
\begin{equation}\label{7.1}
\left\{\begin{array}{l}
\partial_{t} u_0^{1}-\nu \Delta u_0^{1}=-\left(u \cdot \nabla u^1 \right)_0 \\
\partial_{t} u_0^{2}-\nu \Delta  u_0^{2}+ u_0^1+ \partial_{y} \left( - \Delta \right)^{-1} \partial_{y} u_0^{1}=-\left(u \cdot \nabla u^2 \right)_0- \partial_{y}\left( - \Delta \right)^{-1} \left( \partial_{i} u^j \partial_{j} u^i  \right)_0 \\
\partial_{t} u_0^{3}-\nu \Delta u_0^{3}+  \partial_{z}  \left( - \Delta \right)^{-1} \partial_{y} u_0^{1}=-\left(u \cdot \nabla u^3 \right)_0- \partial_{z}\left( - \Delta \right)^{-1} \left( \partial_{i} u^j \partial_{j} u^i  \right)_0.
\end{array}\right.
\end{equation}
\subsection{Energy estimates on $u_0^1$}\label{4.5.1}
First we consider the equation (\ref{7.1})$_{1}$ which gives the $H^{N-1}$ estimate for $u_0^1$
\begin{align*}
\frac{1}{2} \left\|  u_0^1 (T)  \right\|_{H^{N-1}}^2+\nu \left\|   \nabla u_0^1 \right\|_{L^2 H^{N-1}}^2 &= \frac{1}{2} \left\|u_{0}^1(1)   \right\|_{H^{N-1}}^2- \int_{1}^{T} \int  {\left\langle D \right\rangle}^{{N-1}}  u_0^1  {\left\langle D \right\rangle}^{{N-1}}  \left(u \cdot \nabla u^{1} \right)_0 dVdt  \nonumber \\
&\overset{\Delta}{=}\frac{1}{2} \left\|u_{0}^1(1) \right\|_{H^{N-1}}^2+ \mathcal{T}.
\end{align*}
The term $\mathcal{T}$ can be divided into
\begin{align*}
\mathcal{T} & = - \int_{0}^{T} \int  {\left\langle D \right\rangle}^{{N-1}}  u_0^1  {\left\langle D \right\rangle}^{{N-1}}  \Big[ \left( u_0^2 \partial_{y} u_0^{1} + u_0^3 \partial_{z} u_0^{1} \right)+ \left(u_{\neq} \cdot \nabla u_{\neq}^{1} \right)_0 \Big] dVdt  \nonumber \\
& \overset{\Delta}{=} \mathcal{T}_0+ \mathcal{T}_{\neq}.
\end{align*}
Note that $\widehat{u_{0, 0}^2} (\eta)= \widehat{u^2} \left(k=0, l=0 \right)=0$. Applying frequency decomposition in subsection \ref{2.*.4} and \eqref{4..11}--\eqref{4..13} yields
\begin{align}
\mathcal{T}_0 \leqslant & \left\| u_0^1 \right\|_{L^{\infty} H^{N-1}} \left\| \nabla u_0^2 \right\|_{L^2 H^{N-1}} \left\| \nabla u_0^1 \right\|_{L^2 H^{N-1}} \nonumber \\
&+ \left\| \nabla u_0^1 \right\|_{L^2 H^{N-1}} \left\| \nabla u_0^3 \right\|_{L^2 H^{N-1}} \left\| u_0^1 \right\|_{L^{\infty} H^{N-1}}+ \left\| u_0^3 \right\|_{L^{\infty} H^{N-1}} \left\| \nabla u_0^1 \right\|_{L^2 H^{N-1}}^2 \nonumber \\
\lesssim & \varepsilon C_1 \varepsilon  \nu^{-1} \nu^{-\frac{1}{2}} \varepsilon \nu^{-\frac{1}{2}}+ \varepsilon \nu^{-\frac{1}{2}} C_0 \varepsilon  \nu^{-1} \nu^{-\frac{1}{2}} \varepsilon+ C_0 \varepsilon  \nu^{-1} \varepsilon \nu^{-\frac{1}{2}} \varepsilon \nu^{-\frac{1}{2}}  \nonumber \\
\lesssim & \varepsilon^2 \left( C_1 \varepsilon  \nu^{-2} + C_0 \varepsilon  \nu^{-2}  \right) \nonumber \\
\lesssim & \varepsilon^2,
\end{align}
which suffices for $C_0 \varepsilon  \nu^{-2}$  sufficiently small. For $\mathcal{T}_{\neq}$, we use \eqref{4..2}, \eqref{4..16}, \eqref{4..23}, \eqref{4..24}, \eqref{5.1} and Corollary \ref{cor4.1} to deduce
\begin{align*}
\mathcal{T}_{\neq} \leqslant & \left\| u_0^1 \right\|_{L^{\infty} H^{N-1}} \left\| \left( u_{\neq}^2, u_{\neq}^3 \right) \right\|_{L^2 H^{N-1}} \left\| \nabla u_{\neq}^1 \right\|_{L^2 H^{N-1}}   \nonumber \\
\leqslant & \left\| U_0^1 \right\|_{L^{\infty} H^{N-1}} \left\| \left( U_{\neq}^2, U_{\neq}^3 \right) \right\|_{L^2 H^{N-1}} \left\| M \check{K}_{\neq}^1 \right\|_{L^2 H^{N-1}}   \nonumber \\
\lesssim & \varepsilon \left( \varepsilon \nu^{-\frac{1}{6}}+ C_0\varepsilon  \nu^{-\frac{1}{3}} \nu^{-\frac{1}{6}}   \right) \varepsilon \nu^{-\frac{1}{6}} \nonumber \\
\lesssim & \varepsilon^2 \left( \varepsilon  \nu^{-\frac{1}{3}} + C_0 \varepsilon  \nu^{-\frac{2}{3}}  \right) \nonumber \\
\lesssim & \varepsilon^2.
\end{align*}

\subsection{Energy estimates on $u_0^2$}\label{4.5.2}
The estimate for $u_0^2$ is similar to $u_0^1$.   $H^{N-1}$ estimate for $u_0^2$ gives
\begin{align}\label{4..196}
&\frac{1}{2} \left\|  u_0^2 (T)  \right\|_{H^{N-1}}^2+\nu \left\|   \nabla u_0^2 \right\|_{L^2 H^{N-1}}^2  \nonumber \\
&\quad= \frac{1}{2} \left\|  u_{0}^2(1)  \right\|_{H^{N-1}}^2- \int_{1}^{T} \int  {\left\langle D \right\rangle}^{N-1}  u_0^2  {\left\langle D \right\rangle}^{N-1} \Big[ \left(  u_0^1+ \partial_{y} (-\Delta)^{-1} \partial_{y} u_0^{1} \right) \nonumber \\
&\qquad + \left(u \cdot \nabla u^2 \right)_0 +\partial_{y} (-\Delta)^{-1} \left(\partial_{i}  u^{j} \partial_{j} u^{i}\right)_0 \Big]dVdt  \nonumber \\
& \quad\overset{\Delta}{=}\frac{1}{2} \left\|  u_{0 }^2(1)  \right\|_{H^{N-1}}^2+ \mathcal{LU}2+\mathcal{T}+\mathcal{NLP}.
\end{align}
Now we  estimate the terms on the right-hand side of \eqref{4..196}. For $\mathcal{LU}2$, by \eqref{4..11}--\eqref{4..12}, it holds
\begin{align}\label{7.4}
\mathcal{LU}2=&-\int_{1}^{T} \int \left(1+|k, \eta, l|^2\right)^{\frac{N-1}{2}} \widehat{u_0^2}  \left(1+|k, \eta, l|^2\right)^{\frac{N-1}{2}} \frac{ l^2}{\eta^2+l^2} \widehat{u_0^1} d\xi dt \nonumber \\
\leqslant &  \left\| \nabla u_0^1  \right\|_{L^2 H^{N-1}} \left\| \nabla u_0^2  \right\|_{L^2 H^{N-1}} \nonumber \\
\lesssim &  \varepsilon \nu^{-\frac{1}{2}} C_1 \varepsilon \nu^{-1}  \nu^{-\frac{1}{2}} \nonumber \\
\lesssim & \left(C_1 \varepsilon  \nu^{-1} \right)^2 \left( \frac{1}{C_1}     \right) \nonumber \\
\lesssim & \left(C_1 \varepsilon  \nu^{-1} \right)^2,
\end{align}
which suffices for $C_1$ sufficiently large.   
Considering $\mathcal{T}$, it can be treated as for $u_0^1$ in the previous subsection
\begin{align}\label{7.5}
\mathcal{T}&=- \int_{1}^{T} \int  {\left\langle D \right\rangle}^{N-1}  u_0^2  {\left\langle D \right\rangle}^{N-1} \Big[  u_0^2 \partial_{y} u_0^2 + u_0^3 \partial_{z} u_0^2 + \left(u_{\neq} \cdot \nabla u_{\neq}^2 \right)_0 \Big]dVdt  \nonumber \\
&\leqslant  \left\| u_0^2 \right\|_{L^{\infty} H^{N-1}} \left\| \nabla u_0^2 \right\|_{L^2 H^{N-1}}^2+\left\| u_0^2 \right\|_{L^{\infty} H^{N-1}} \left\| \nabla u_0^3 \right\|_{L^2 H^{N-1}} \left\| \nabla u_0^2 \right\|_{L^2 H^{N-1}} \nonumber \\
&\quad+\left\| \nabla u_0^2 \right\|_{L^2 H^{N-1}} \left\| u_0^3 \right\|_{L^{\infty} H^{N-1}} \left\| \nabla u_0^2 \right\|_{L^2 H^{N-1}}  \nonumber \\
&\quad+\left\| u_0^2 \right\|_{L^{\infty} H^{N-1}} \left\| \left( u_{\neq}^2, u_{\neq}^3  \right) \right\|_{L^2 H^{N-1}} \left\| M \check{K}_{\neq}^2 \right\|_{L^2 H^{N-1}} \nonumber \\
&\lesssim  C_1 \varepsilon  \nu^{-1} \left( C_1 \varepsilon  \nu^{-1} \nu^{-\frac{1}{2}} \right)^2+ C_1 \varepsilon  \nu^{-1} C_0 \varepsilon  \nu^{-1} \nu^{-\frac{1}{2}} C_1 \varepsilon  \nu^{-1} \nu^{-\frac{1}{2}}  \nonumber \\
&\quad+ C_1 \varepsilon  \nu^{-1} \left( \varepsilon \nu^{-\frac{1}{6}}+ C_0 \varepsilon \nu^{-\frac{1}{3}}  \nu^{-\frac{1}{6}} \right) \varepsilon \nu^{-\frac{1}{6}}  \nonumber \\
&\lesssim \left(C_1 \varepsilon  \nu^{-1} \right)^2 \left( C_1 \varepsilon  \nu^{-2}+ C_0 \varepsilon  \nu^{-2}+ \frac{1}{C_1} \varepsilon \nu^{\frac{2}{3}} + \frac{C_0}{C_1} \varepsilon \nu^{\frac{1}{3}}         \right)   \nonumber \\
& \lesssim  \left(C_1 \varepsilon  \nu^{-1} \right)^2,
\end{align}
by using $U_{\neq}^2= -|\partial_{X}|^{-1} |\nabla_{L}|^{-1} \check{K}_{\neq}^2$, \eqref{4..5}, \eqref{4..16}--\eqref{4..18}, \eqref{4..23} and \eqref{4..24}.

The term $\mathcal{NLP}$ is divided into
\begin{align*}
\mathcal{NLP}&= -\int_{1}^{T} \int  {\left\langle D \right\rangle}^{N-1}  u_0^2  {\left\langle D \right\rangle}^{N-1} \partial_{y} (-\Delta)^{-1}\partial_{i} \Big[ u_0^{j} \partial_{j} u_0^{i} +\left(  u_{\neq}^{j} \partial_{j} u_{\neq}^{i}\right)_0 \Big]dVdt  \nonumber \\
& \overset{\Delta}{=} \mathcal{NLP}_0+\mathcal{NLP}_{\neq}.
\end{align*}
We use \eqref{4..12}--\eqref{4..13} and the fact $\left\| u_0^2 \right\|_{L^2 H^{N-1}} \leqslant \left\| \partial_{z} u_0^2 \right\|_{L^2 H^{N-1}}$    to obtain
\begin{align}\label{7.6}
\mathcal{NLP}_0 \leqslant & \left\| u_0^2 \right\|_{L^2 H^{N-1}} \left\| u_0^{2, 3} \right\|_{L^{\infty} H^{N-1}} \left\| \nabla u_0^{2, 3} \right\|_{L^2 H^{N-1}} \nonumber \\
\leqslant & \left\| \partial_{z} u_0^2 \right\|_{L^2 H^{N-1}} \left\| u_0^{2, 3} \right\|_{L^{\infty} H^{N-1}} \left\| \nabla u_0^{2, 3} \right\|_{L^2 H^{N-1}} \nonumber \\
\lesssim &  C_1 \varepsilon  \nu^{-1} \nu^{-\frac{1}{2}} \left( C_1 \varepsilon  \nu^{-1}+ C_0 \varepsilon  \nu^{-1}\right) C_0 \varepsilon  \nu^{-1} \nu^{-\frac{1}{2}}     \nonumber \\
\lesssim & \left(C_1 \varepsilon  \nu^{-1} \right)^2 \left( \frac{C_0^2}{C_1} \varepsilon  \nu^{-2}+ C_0 \varepsilon  \nu^{-2}   \right)   \nonumber \\
\lesssim &  \left(C_1 \varepsilon  \nu^{-1} \right)^2.
\end{align}
By \eqref{4..5},  \eqref{4..17}, \eqref{4..23}--\eqref{4..24}, \eqref{5.1} and Corollary \ref{cor4.1}, it holds
\begin{align}\label{7.7}
\mathcal{NLP}_{\neq} \leqslant &  -\int_{1}^{T} \int  {\left\langle D \right\rangle}^{N-1}  u_0^2  {\left\langle D \right\rangle}^{N-1} \partial_{y} (-\Delta)^{-1}\partial_{i} \left(  u_{\neq}^{j} \partial_{j} u_{\neq}^{i}\right)_0 dVdt \cdot 1_{i, j \neq 1} \nonumber \\
\leqslant &  \left\| U_0^2 \right\|_{L^{\infty} H^{N-1}} \left\| U_{\neq}^{2, 3} \right\|_{L^2 H^{N-1}} \left\| \nabla_{L} U_{\neq}^{2, 3} \right\|_{L^2 H^{N-1}}   \nonumber \\
\leqslant &  \left\| U_0^2 \right\|_{L^{\infty} H^{N-1}} \left\| U_{\neq}^{2, 3} \right\|_{L^2 H^{N-1}} \left( \left\| M \check{K}_{\neq}^{2} \right\|_{L^2 H^{N-1}}+\left\| \nabla_{L} U_{\neq}^{3} \right\|_{L^2 H^{N-1}} \right)   \nonumber \\
\lesssim & C_1 \varepsilon  \nu^{-1} \left( \varepsilon \nu^{-\frac{1}{6}}+ C_0 \varepsilon \nu^{-\frac{1}{3}}  \nu^{-\frac{1}{6}} \right) \left( \varepsilon \nu^{-\frac{1}{6}}+ C_0 \varepsilon \nu^{-\frac{1}{3}}  \nu^{-\frac{1}{2}} \right)        \nonumber \\
\lesssim & \left(C_1 \varepsilon  \nu^{-1} \right)^2 \left( \frac{1}{C_1} \varepsilon+ \frac{C_0}{C_1} \varepsilon+\frac{C_0^2}{C_1} \varepsilon  \nu^{-\frac{1}{3}}  \right)   \nonumber \\
\lesssim &  \left(C_1 \varepsilon  \nu^{-1} \right)^2.
\end{align}
Combining \eqref{7.4}--\eqref{7.7} and \eqref{4..196}, we can close the a priori assumption \eqref{4..12}.

\subsection{Energy estimates on $u_0^3$}\label{4.5.3}
Similar to the previous subsection, the $H^{N-1}$ energy estimate for $u_0^3$ is given
\begin{align}\label{4..217}
\frac{1}{2} \left\|  u_0^3 (T)  \right\|_{H^{N-1}}^2+\nu \left\|   \nabla u_0^3 \right\|_{L^2 H^{N-1}}^2 =& \frac{1}{2} \left\|  u_{0}^3(1)  \right\|_{H^{N-1}}^2- \int_{1}^{T} \int   {\left\langle D \right\rangle}^{N-1} \Big[ \left(  \partial_{z} (-\Delta)^{-1} \partial_{y} u_0^{1} \right) \nonumber \\
&  + \left(u \cdot \nabla u^3 \right)_0 +\partial_{z} (-\Delta)^{-1} \left(\partial_{i}  u^{j} \partial_{j} u^{i}\right)_0 \Big] {\left\langle D \right\rangle}^{N-1}  u_0^3 dVdt  \nonumber \\
\overset{\Delta}{=}& \frac{1}{2} \left\|u_{0}^3(1)  \right\|_{H^{N-1}}^2+ \mathcal{LU}3+\mathcal{T}+\mathcal{NLP},
\end{align}
which  the right-hand side will be estimated term by term as follows. First, from incompressible condition, it follows   $\partial_{y} u_0^2+ \partial_{z} u_0^3=0$. Thus, by  \eqref{4..11}--\eqref{4..12} and integration by parts, it holds
\begin{align}
\mathcal{LU}3=& \int_{1}^{T} \int  {\left\langle D \right\rangle}^{N-1}  \partial_{z} u_0^3  {\left\langle D \right\rangle}^{N-1} \left(   (-\Delta)^{-1} \partial_{y} u_0^{1} \right) dVdt \nonumber \\
=&-\int_{0}^{T} \int  {\left\langle D \right\rangle}^{N-1}  \partial_{y} u_0^2  {\left\langle D \right\rangle}^{N-1} \left(   (-\Delta)^{-1} \partial_{y} u_0^{1} \right) dVdt \nonumber \\
\leqslant &  \left\| \nabla u_0^1  \right\|_{L^2 H^{N-1}} \left\| \nabla u_0^2  \right\|_{L^2 H^{N-1}} \nonumber \\
\lesssim &  \varepsilon \nu^{-\frac{1}{2}} C_1 \varepsilon \nu^{-1}  \nu^{-\frac{1}{2}} \nonumber \\
\lesssim & \left(C_0 \varepsilon  \nu^{-1} \right)^2 \left( \frac{C_1}{C_0^2}     \right) \nonumber \\
\lesssim & \left(C_0 \varepsilon  \nu^{-1} \right)^2,
\end{align}
which suffices for $C_0 \gg C_1 >1$. The term $\mathcal{T}$ can be treated as for $u_0^2$ in the previous subsection
\begin{align}
\mathcal{T}&=- \int_{1}^{T} \int  {\left\langle D \right\rangle}^{N-1}  u_0^3  {\left\langle D \right\rangle}^{N-1} \Big[  u_0^2 \partial_{y} u_0^3 + u_0^3 \partial_{z} u_0^3 + \left(u_{\neq} \cdot \nabla u_{\neq}^3 \right)_0 \Big]dVdt  \nonumber \\
& \leqslant \left\| u_0^3 \right\|_{L^{\infty} H^{N-1}} \left\| \nabla u_0^3 \right\|_{L^2 H^{N-1}}^2+\left\| u_0^3 \right\|_{L^{\infty} H^{N-1}} \left\| \nabla u_0^2 \right\|_{L^2 H^{N-1}} \left\| \nabla u_0^3 \right\|_{L^2 H^{N-1}} \nonumber \\
&\quad+\left\| u_0^3 \right\|_{L^{\infty} H^{N-1}} \left\| \left( U_{\neq}^2, U_{\neq}^3  \right) \right\|_{L^2 H^{N-1}} \left\| \nabla_{L} U_{\neq}^3 \right\|_{L^2 H^{N-1}} \nonumber \\
&\lesssim  C_0 \varepsilon  \nu^{-1} \left( C_0 \varepsilon  \nu^{-1} \nu^{-\frac{1}{2}} \right)^2+ C_0 \varepsilon  \nu^{-1} C_1 \varepsilon  \nu^{-1} \nu^{-\frac{1}{2}} C_0 \varepsilon  \nu^{-1} \nu^{-\frac{1}{2}}  \nonumber \\
&\quad+ C_0 \varepsilon  \nu^{-1} \left( \varepsilon \nu^{-\frac{1}{6}}+ C_0 \varepsilon \nu^{-\frac{1}{3}}  \nu^{-\frac{1}{6}} \right) C_0 \varepsilon \nu^{-\frac{1}{3}}  \nu^{-\frac{1}{2}}  \nonumber \\
& \lesssim \left(C_0 \varepsilon  \nu^{-1} \right)^2 \left( C_0 \varepsilon  \nu^{-2}+ C_1 \varepsilon  \nu^{-2}+ \varepsilon + C_0 \varepsilon  \nu^{-\frac{1}{3}}       \right)   \nonumber \\
& \lesssim  \left(C_0 \varepsilon  \nu^{-1} \right)^2,
\end{align}
which suffices for $C_0 \varepsilon  \nu^{-2}$ sufficiently small. For $\mathcal{NLP}$, we have
\begin{align*}
\mathcal{NLP}&= -\int_{1}^{T} \int  {\left\langle D \right\rangle}^{N-1}  u_0^3  {\left\langle D \right\rangle}^{N-1} \partial_{z} (-\Delta)^{-1}\partial_{i} \Big[ u_0^{j} \partial_{j} u_0^{i} +\left(  u_{\neq}^{j} \partial_{j} u_{\neq}^{i}\right)_0 \Big]dVdt  \nonumber \\
& \overset{\Delta}{=} \mathcal{NLP}_0+\mathcal{NLP}_{\neq}.
\end{align*}
From frequency decomposition, \eqref{4..12}--\eqref{4..13} and the inequality $\left\| u_0^2 \right\|_{L^2 H^{N-1}} \leqslant \left\| \partial_{z} u_0^2 \right\|_{L^2 H^{N-1}}$, one obtain 
\begin{align}
\mathcal{NLP}_0 \leqslant & \left\| u_0^3 \right\|_{L^{\infty} H^{N-1}} \left\| u_0^{2} \right\|_{L^2 H^{N-1}} \left\| \nabla u_0^{2, 3} \right\|_{L^2 H^{N-1}}+\left\| u_0^3 \right\|_{L^{\infty} H^{N-1}} \left\| \nabla u_0^{3} \right\|_{L^2 H^{N-1}}^2 \nonumber \\
\lesssim &  C_0 \varepsilon  \nu^{-1} C_1 \varepsilon  \nu^{-1} \nu^{-\frac{1}{2}} C_0 \varepsilon  \nu^{-1} \nu^{-\frac{1}{2}} +C_0 \varepsilon  \nu^{-1} \left( C_0 \varepsilon  \nu^{-1} \nu^{-\frac{1}{2}}  \right)^2   \nonumber \\
\lesssim & \left(C_0 \varepsilon  \nu^{-1} \right)^2 \left( C_1 \varepsilon  \nu^{-2}+ C_0 \varepsilon  \nu^{-2}   \right)   \nonumber \\
\lesssim &  \left(C_0 \varepsilon  \nu^{-1} \right)^2.
\end{align}
For $\mathcal{NLP}_{\neq}$, applying \eqref{4..13}, \eqref{4..23}--\eqref{4..24} and Corollary \ref{cor4.1} yields 
\begin{align*}
\mathcal{NLP}_{\neq} \leqslant &  \int_{1}^{T} \int  {\left\langle D \right\rangle}^{N-1}  u_0^3  {\left\langle D \right\rangle}^{N-1} \partial_{z} (-\Delta)^{-1}\partial_{i} \left(  u_{\neq}^{j} \partial_{j} u_{\neq}^{i}\right)_0 dVdt \cdot 1_{i, j \neq 1} \nonumber \\
\leqslant &  \left\| U_0^3 \right\|_{L^{\infty} H^{N-1}} \left\| U_{\neq}^{2, 3} \right\|_{L^2 H^{N-1}} \left\| \nabla_{L} U_{\neq}^{2, 3} \right\|_{L^2 H^{N-1}}   \nonumber \\
\lesssim & C_0 \varepsilon  \nu^{-1} \left( \varepsilon \nu^{-\frac{1}{6}}+ C_0 \varepsilon  \nu^{-\frac{1}{3}} \nu^{-\frac{1}{6}} \right) \left( \varepsilon \nu^{-\frac{1}{2}}+ C_0 \varepsilon \nu^{-\frac{1}{3}}  \nu^{-\frac{1}{2}} \right)        \nonumber \\
\lesssim & \left(C_0 \varepsilon  \nu^{-1} \right)^2 \left( C_0 \varepsilon  \nu^{-\frac{1}{3}}  \right)   \nonumber \\
\lesssim &  \left(C_0 \varepsilon  \nu^{-1} \right)^2.
\end{align*}
 
Until now, the proof of Proposition \ref{pro4.1} is completed. At the same time, Propositions \ref{pro4.1}--\ref{pro4.3} imply Theorem \ref{1.2.}.
\vskip 4mm

\vspace{1mm} \noindent\textbf{Data availability statement.}
\vskip2mm

No new data were created or analysed in this study.

\vspace{4mm} \noindent\textbf{Acknowledgements.}
 
Xu was partially supported by the National Key R\&D Program of China (grant 2020YFA0712900) and the National Natural Science Foundation of China (grants 12171040, 11771045, and 12071069).

\addcontentsline{toc}{section}{References}

\end{CJK*}
\end{document}